\documentclass[12pt]{amsart}

\usepackage{amssymb,amsthm,amscd,wasysym,mathrsfs,amsxtra,mathtools}
\input xy
\xyoption{all}

\usepackage{tikz}
\usetikzlibrary{calc,intersections}
\usepackage{float}

\usepackage[margin=1in]{geometry}
\usepackage{tikz-cd}
\usepackage{enumitem}
\setlist[enumerate,1]{font=\normalfont, label=(\roman*)}

\usepackage{hyperref}
\usepackage[sort=use]{glossaries}

\renewcommand{\glsglossarymark}[1]{}

\usepackage[postpunc=;]{glossaries-extra}
\GlsXtrEnablePreLocationTag{page }{pages }

\newtheorem{theorem}{Theorem}[section]

\newtheorem{lemma}[theorem]{Lemma}

\newtheorem{proposition}[theorem]{Proposition}

\newtheorem{corollary}[theorem]{Corollary}

\theoremstyle{definition}
\newtheorem{definition}[theorem]{Definition}

\newtheorem{example}[theorem]{Example}

\newtheorem{remark}[theorem]{Remark}

\DeclareMathOperator{\cone}{cone}
\DeclareMathOperator{\Img}{Im}
\DeclareMathOperator{\Hom}{Hom}
\DeclareMathOperator{\Ext}{Ext}
\DeclareMathOperator{\id}{id}
\DeclareMathOperator{\Coh}{Coh}
\DeclareMathOperator{\qco}{qco}
\DeclareMathOperator{\num}{num}
\DeclareMathOperator{\cok}{cok}
\DeclareMathOperator{\supp}{supp}
\DeclareMathOperator{\Slice}{Slice}
\DeclareMathOperator{\ch}{ch}
\DeclareMathOperator{\NS}{NS}
\DeclareMathOperator{\Per}{Per}

\DeclareMathOperator{\Stab}{Stab}

\DeclareMathOperator{\pr}{pr}

\makeglossaries

\newglossaryentry{crepant-resolution}{
    name={\ensuremath{\pi \colon \widetilde{X} \longrightarrow X}},
    description={the crepant resolution of a surface $X$ with an ADE singularity $x_0$}
}

\newglossaryentry{exceptional-curves}{
    name={\ensuremath{C_i, \, i=1,\dots,n}},
    description={the exceptional $(-2)$-curves of $\pi$}
}

\newglossaryentry{scheme-preimage}{
    name={\ensuremath{\Pi := \pi^{-1}(x_0) = \bigcup_i C_i}},
    description={the scheme-theoretic preimage of the singular point $x_0$}
}

\newglossaryentry{projection-map}{
    name={\ensuremath{\operatorname{pr} \colon {\rm K}_{\num}(\widetilde{X}) \longrightarrow {\rm K}_{\num}(\widetilde{X})/\ker \pi_\ast}},
    description={the projection map from the numerical $K$-group of $\widetilde{X}$}
}

\newglossaryentry{cohomology-functor}{
    name={\ensuremath{H^i_{\mathcal{A}}}},
    description={the $i$-th cohomology with respect to $\mathcal{A}$. If $\mathcal{A} = \operatorname{Coh}(\widetilde{X})$, we omit the subscript $\mathcal{A}$; if $\mathcal{A} = {^{-1}\Per(\widetilde{X}/X)}$, we write $H^i_P$},
}

\newglossaryentry{derived-pullback}{
    name={\ensuremath{\mathbf{L}^i_{\mathcal{A}}\pi^\ast := H^i_{\mathcal{A}} \circ \mathbf{L}\pi^\ast}},
    description={if $\mathcal{A} = \operatorname{Coh}(\widetilde{X})$, we omit the subscript $\mathcal{A}$; if $\mathcal{A} = {^{-1}\Per(\widetilde{X}/X)}$, we write $\mathbf{L}^i_P\pi^\ast$}
}

\newglossaryentry{slope-stability}{
    name={\ensuremath{Z_H := -H^{n-1} \cdot \operatorname{ch}_1 + iH^n \operatorname{ch}_0}},
    description={the slope stability with respect to a nef divisor $H$}
}

\newglossaryentry{slope-function}{
    name={\ensuremath{\mu_{\pi^\ast H}(E)}},
    description={the slope of $E$ with respect to $\pi^\ast H$ on $\operatorname{Coh}(X)$ and ${^{-1}\operatorname{Per}}(\widetilde{X}/X)$}
}

\newglossaryentry{stability-function}{
    name={\ensuremath{Z_{\pi^\ast H,\beta,z} := -\operatorname{ch}_2 + \beta \cdot \operatorname{ch}_1 + z\operatorname{ch}_0 + i(\pi^\ast H)\cdot \operatorname{ch}_1}},
    description={the stability function on $\mathcal{B}^0$}
}

\newglossaryentry{quadratic-form}{
    name={\ensuremath{Q_{A,B} := \Delta + A(\Im Z_{\pi^\ast H,\beta,z})^2 + B (\Re Z_{\pi^\ast H,\beta,z})^2}},
    description={the quadratic form on ${\rm K}_{\num}(\widetilde{X})$ for the support property of $(Z_{\pi^\ast H,\beta,z},\mathcal{B}^0)$}
}

\newglossaryentry{sigma-epsilon}{
    name={\ensuremath{\sigma_\epsilon := (Z_{\pi^\ast H,\epsilon \beta,z}, \mathcal{B}^0)}},
    description={Bridgeland stability conditions on $\widetilde{X}$ obtained from deformation}
}

\newglossaryentry{weak-stability-function}{
    name={\ensuremath{Z_{\widetilde{X}} := Z_{\pi^\ast H,0,z}}},
    description={the weak stability function at the end point of the path in $\Stab(\widetilde{X})$}
}

\newglossaryentry{weak-stability-condition}{
    name={\ensuremath{\sigma_{\widetilde{X}} := (Z_{\widetilde{X}},\mathcal{B}^0)}},
    description={the weak stability condition lying on the boundary of $\Stab(\widetilde{X})$ }
}

\newglossaryentry{stability-function-X}{
    name={\ensuremath{Z_X}},
    description={the stability function on $\operatorname{Coh}^0_H(X)$ with the compatibility $Z_X \circ \pi_\ast = Z_{\widetilde{X}}$}
}

\newglossaryentry{sigma-X}{
    name={\ensuremath{\sigma_X := (Z_X,\operatorname{Coh}^0_H(X))}},
    description={the induced Bridgeland stability condition on $X$}
}

\newglossaryentry{quadratic-form-X}{
    name={\ensuremath{Q_{\widetilde{X}} := \Delta + A(\Im Z_{\widetilde{X}})^2}},
    description={the quadratic form on ${\rm K}_{\num}(\widetilde{X})/\ker \pi_\ast$ for the support property of $\sigma_{\widetilde{X}}$}
}

\newglossaryentry{T_P}{
    name={\ensuremath{\mathcal{T}_P := \{ T \in \operatorname{Coh}(\widetilde{X}) \mid \mathbf{R}^1\pi_\ast T = 0, \Hom(T,\ker \mathbf{R}\pi_\ast)= 0 \}}},
    description={the torsion part defining perverse coherent sheaves} 
}

\newglossaryentry{F_P}{
    name={\ensuremath{\mathcal{F}_P  := \{ F \in \operatorname{Coh}(\widetilde{X}) \mid \mathbf{R}^0\pi_\ast F = 0 \}}},
    description={the torsion-free part defining perverse coherent sheaves}
}

\newglossaryentry{T-H}{
    name={\ensuremath{\mathcal{T}^{>\beta}_{H} := \{ T \in \operatorname{Coh}(X) \mid \mu_{H}(S) > \beta \textnormal{ for all } T \twoheadrightarrow S \}}},
    description={the torsion part with respect to $\mu_H$}
}

\newglossaryentry{F-H}{
    name={\ensuremath{\mathcal{F}^{\leq \beta}_{H} := \{ F \in \operatorname{Coh}(X) \mid \mu_{H}(G) \leq \beta \textnormal{ for all } 0 \neq G \hookrightarrow F \}}},
    description={the torsion-free part with respect to $\mu_H$}
}

\newglossaryentry{tilted-heart}{
    name={\ensuremath{\operatorname{Coh}^\beta_H(X)}},
    description={the heart obtained by tilting $\operatorname{Coh}(X)$ with respect to $\mu_{H}$ at the slope $\beta$}
}

\newglossaryentry{F_0}{
    name={\ensuremath{\mathcal{F}_0 := \mathcal{F}_P}},
    description={viewed as a subcategory of $\Coh^0_{\pi^\ast H}(\widetilde{X})$}
}

\newglossaryentry{T_0}{
    name={\ensuremath{\mathcal{T}_0 := \{ T \in \operatorname{Coh}^0(\widetilde{X}) \mid \Hom(T,\mathcal{F}_0) = 0 \}}},
    description={the left orthogonal complement of $\mathcal{F}_0$ in $\operatorname{Coh}^0_{\pi^\ast H}(\widetilde{X})$}
}

\newglossaryentry{perverse-coherent-sheaves}{
    name={\ensuremath{{}^{-1}\Per(\widetilde{X}/X)}},
    description={the category of perverse coherent sheaves on $\widetilde{X}$ over $X$}
}

\newglossaryentry{T-pi-H-P}{
    name={\ensuremath{\mathcal{T}^{>0}_{\pi^\ast H,P} := \{ T \in {}^{-1}\Per(\widetilde{X}/X) \mid \mu_{\pi^\ast H}(S) > 0 \textnormal{ for all } T \twoheadrightarrow S \}}},
    description={the torsion part on ${^{-1}\Per(\widetilde{X}/X)}$ defining $\mathcal{B}^0$}
}

\newglossaryentry{F-pi-H-P}{
    name={\ensuremath{\mathcal{F}^{\leq 0}_{\pi^\ast H,P} := \{ F \in {}^{-1}\Per(\widetilde{X}/X) \mid \mu_{\pi^\ast H}(G) \leq 0 \textnormal{ for all } 0 \neq G \hookrightarrow F \}}},
    description={the torsion-free part on ${^{-1}\Per(\widetilde{X}/X)}$ defining $\mathcal{B}^0$}
}

\newglossaryentry{heart-B0}{
    name={\ensuremath{\mathcal{B}^0 = \mathcal{B}^0_{\pi^\ast H,-1} = \mathcal{B}^0_{\pi^\ast H,P} }},
    description={the heart \ensuremath{\langle \mathcal{F}^{\leq 0}_{\pi^\ast H,P}[1], \mathcal{T}^{>0}_{\pi^\ast H,P} \rangle = \langle \mathcal{F}_0[1], \mathcal{T}_0 \rangle}}
}

\newglossaryentry{torsion-part-B0}{
    name={\ensuremath{\mathcal{T}_{\mathcal{B}^0} := \operatorname{Coh}(\widetilde{X}) \cap \mathcal{B}^0}},
    description={the torsion part on $\Coh(\widetilde{X})$ defining $\mathcal{B}^0$}
}

\newglossaryentry{torsion-free-part-B0}{
    name={\ensuremath{\mathcal{F}_{\mathcal{B}^0} := \operatorname{Coh}(\widetilde{X}) \cap \mathcal{B}^0[-1]}},
    description={the torsion-free part on $\Coh(\widetilde{X})$ defining $\mathcal{B}^0$}
}

\title[Stability condition on a singular surface and its resolution]{Stability condition on a singular surface and its resolution}  
\author[Tzu-Yang Chou]{Tzu-Yang Chou}

\begin{document}

\begin{abstract} Let $X$ be a surface with an ADE-singularity and let $\widetilde{X}$ be its crepant resolution.
In this paper, we show that there exists a Bridgeland stability condition $\sigma_X$ on ${\rm D}^b(X)$ and a weak stability condition $\sigma_{\widetilde{X}}$ on the derived category of the desingularisation ${\rm D}^b(\widetilde{X})$,
such that pushforward of $\sigma_{\widetilde{X}}$-semistable objects are $\sigma_X$-semistable

We first construct Bridgeland stability conditions on  ${\rm D}^b(\widetilde{X})$ associated to the contraction $\widetilde{X} \longrightarrow X$, generalizing the results of Tramel and Xia in \cite{TX22},
Then we deform them to a weak stability condition $\sigma_{\widetilde{X}}$ and show that it descends to ${\rm D}^b(X)$, producing the stability condition $\sigma_X$.

Finally, we study the moduli spaces of $\sigma_{\pi^\ast H,\beta,z}$, of $\sigma_{\widetilde{X}}$, and of $\sigma_X$-semistable objects,
and we show that the moduli spaces satisfy boundedness and openness, and hence are all Artin stacks of finite type over $\mathbb{C}$.
\end{abstract}

\maketitle

\section{Introduction}
Bridgeland stability conditions have been constructed on curves, surfaces, and a number of smooth threefolds, including all Fano threefolds (see, for example, \cite{Bri08,BMMS12,Li19}).
Recently, Langer constructed stability conditions on normal surfaces \cite{Lan24}.

In this article,  we construct a stability condition on a singular surface with an ADE singularity which is   compatible with a weak stability condition on resolution. 

Our main result is the following:
\begin{theorem}\label{main1}
Let $X$ be a projective surface with a single ADE singularity and $\pi \colon \widetilde{X} \longrightarrow X$ be its crepant resolution. 

Then there exist a stability condition $\sigma_X=(Z_X, \mathcal{A})$ on ${\rm D}^b(X)$,
and a weak stability condition $\sigma_{\widetilde{X}} = (Z_{\widetilde{X}},\mathcal{B}^0)$ ${\rm D}^b(\widetilde{X})$
such that they are related as follows.
\begin{enumerate}
    \item[(1)] $Z_{\widetilde{X}} = Z_X \circ \pi_\ast$ 
    \item[(2)] $\pi_\ast \colon \mathcal{B}^0 \longrightarrow \mathcal{A}$ is an exact functor and for any $Z_{\widetilde{X}}$-semistable object $E$, its pushforward $\pi_\ast E$ is $Z_X$-semistable.
    \item[(3)] In particular, $\pi_\ast$ induces a morphism between the moduli spaces $\mathcal{M}_{\sigma_{\widetilde{X}}}(v) \longrightarrow \mathcal{M}_{\sigma_X}(\pi_\ast v)$.
\end{enumerate}
\end{theorem}

It is quite natural that we first study the derived category of the resolution $\widetilde{X}$.
The construction will be realized by deforming the central charge of a stability condition on ${\rm D}^b(\widetilde{X})$ to a weak stability condition $\sigma_{\widetilde{X}}$,
and then the construction of $\sigma_X$ will be based on that of $\sigma_{\widetilde{X}}$.

By generalizing \cite[Sect. 5,6]{TX22},
we will construct $\sigma_{\widetilde{X}}$ as the limit of stability conditions $\sigma_{\epsilon}$,
which are associated to the contractions $\pi \colon \widetilde{X} \longrightarrow X$,
in the sense that the skyscraper sheaf $\mathcal{O}_x$ is strictly $\sigma_{\epsilon}$-semistable for $x$ in the exceptional locus $\Pi$,
and is $\sigma_{\epsilon}$-stable for any $x \in\widetilde{X} \setminus \Pi$.

In order to do so, we need to construct a heart of bounded t-structure,
using the technique of torsion pair and tilting in the sense of \cite{HRS96}.
Combining arguments in \cite{Tod13} and in \cite[Sect. 3]{TX22}, we have a double-tilted heart $\mathcal{B}^0$ on ${\rm D}^b(\widetilde{X})$, with two different descriptions,
which are useful for constructing the central charges on this heart and the support property then after.
In \cite[Sect. 3]{Tod13}, there is an assumption that $C^2=-1$. 
This is used to prove a Bogomolov-Gieseker type result and to construct a Bridgeland stability condition,
but our argument doesn't use this assumption.

In summary, we have the following result:
\begin{theorem}
Let $X$ be a projective surface with a single ADE singularity and $\pi \colon \widetilde{X} \longrightarrow X$ be its crepant resolution. 
Then there exists a path $\sigma_\epsilon = (Z_\epsilon, \mathcal{B}^0)$ of Bridgeland stability conditions on ${\rm D}^b(\widetilde{X})$ with the same heart,
such that the end point of the path $\sigma_\epsilon$ is the weak stability condition $\sigma_{\widetilde{X}} = (Z_{\widetilde{X}},\mathcal{B}^0)$.
\end{theorem}

We then prove that the end point $\sigma_{\widetilde{X}}$ descends to a pre-stability condition $\sigma_X$ on ${\rm D}^b(X)$, viewed as the quotient of ${\rm D}^b(\widetilde{X})$, 
and show that there exists a compatability between $\sigma_{\widetilde{X}}$ and $\sigma_X$ given by the pushforward $\pi_\ast$, which also gives a surjection between the set of semistable objects.

\begin{theorem}\label{main3}
For any class $v \in {\rm K}_{\num}(X)$, there exists a class $\widetilde{v} \in {\rm K}_{\num}(\widetilde{X})$ such that there is a surjective map $\pi_\ast \colon \mathcal{M}_{\sigma_{\widetilde{X}}}(\widetilde{v}) \longrightarrow \mathcal{M}_{\sigma_X}(v)$.
\end{theorem}

Furthermore, we use the compatibility with $\pi_\ast$ to naturally induces a quadratic form on the quotient lattice ${\rm K}_{\num}(X) \simeq {\rm K}_{\num}(\widetilde{X})/ \ker \pi_\ast$ which provides the support property of $\sigma_X$.

Finally, we work out the generic flatness for the heart $\mathcal{B}^0$ 
and show that $\sigma_\epsilon$ lies in the closure of the geometric chamber.
Therefore, the moduli spaces satisfy boundedness and openness, and hence are all Artin stacks of
finite type over $\mathbb{C}$.

\subsection{Related works}
There are some recent related works.
In \cite{LR22}, the authors constructed stability conditions on the derived category of coherent sheaves on the algebraic stack associated to the projective surface with an ADE singularity.

Also, Vilches \cite{Vil25} consider stability conditions associated to more general birational contractions of surfaces.
In particular, for an ADE configuration,
he also recovers the stability conditions
$\sigma_\epsilon$ on the smooth surface $\widetilde{X}$.

By \cite{BPPW22}, the stability manifold $\Stab(\widetilde{X})$ admits a partial compactification in the space $\Hom(\Lambda, \mathbb{C})  \times  \Slice(\widetilde{X})$ by massless semistable objects.
Then, in terms of \cite{BPPW22}, $\sigma_{\widetilde{X}}$ can be thought of as a lax stability condition lying in that stratum of $\partial\Stab(\widetilde{X})$.

In \cite{CLSY24}, the authors work in the case of a curve with negative self-intersection on a K3 surface. They consider a different lax pre-stability condition (in the language of \cite{BPPW22}) with the same central charge as $\sigma_{\widetilde {X}}$,  and also study how it behaves under the spherical twist associated to the curve.

In \cite{Bol23}, Bolognese gives another local compactification of the stability manifold $\Stab(\widetilde{X})$.
The stability condition $\sigma_X$ we constructed can also be viewed as a generalized stability condition on $\widetilde{X}$ in Bolognese's sense.

Our result possibly provides a model of what could happen on higher dimensional singular varieties.
One of the applications can be to study the derived category of a smoothing of 1-nodal Fano threefolds and its Kuznetsov component (cf. \cite{KS24}) as we need fiberwise stability conditions, using the theory of stability conditions in families developed 
in \cite{BLM+}.

On the other hand, one may also want to consider the same family with the singular fiber replaced with its resolution.
We want to relate the notion of (semi)stability on ${\rm D^b}(X)$ and that on its resolution ${\rm D}^b(\widetilde{X})$,
and then our result may be applied.

\subsection{Plan of the paper}
This paper is organized as follows. Section 2 contains some preliminaries of stability conditions and the theory of tilting.

In Section 3, we first construct the heart by the double tilting of coherent sheaves.
We will moreover show that two double tilted hearts actually coincide and study the heart $\mathcal{B}^0$.
Particularly, we are interested in some simple objects which will appear in the Jordan--Hölder filtrations.

In Section 4 and 5, we respectively define the central charge and the quadratic form for support property on ${\rm D}^b(\widetilde{X})$.
Also, we verify the Harder--Narasimhan property on $\mathcal{B}^0$ with respect to $Z_{\pi^\ast H,\beta,z}$.

In Section 6, we deform the central charge $Z_\epsilon$ and get a path $\sigma_\epsilon = (Z_\epsilon, \mathcal{B}^0)$ of stability conditions.
We will show that the end point $\sigma_{\widetilde{X}} = (Z_{\widetilde{X}}, \mathcal{B}^0)$ is a weak stability condition;
we also find a tilted heart $\Coh^\beta_H(X) \subseteq {\rm D}^b(X)$ which forms a Bridgeland stability condition together with $Z_X$.
Moreover, $\sigma_X=(Z_X,\Coh^\beta_H(X))$ is compatible with $\sigma_{\widetilde{X}}$ in the sense of Theorem \ref{main1}.
We also prove that for each $Z_X$-semistable object $E$ with $\ch(E)=v$,
we can find a class $\widetilde{v}$ and 
a $Z_{\widetilde{X}}$-semistable object $\widetilde{E}$ of class $\widetilde{v}$.

In Section 7, we conclude that the induced Bridgeland stability condition on the singular surface also satisfies the support property, whose quadratic form is also induced by that of the end point $\sigma_{\widetilde{X}}$.

In Section 8, we first prove the generic flatness for the heart $\mathcal{B}^0$, 
and then we show that $\sigma_\epsilon$ is in the closure of geometric chamber,
which implies the boundedness of its moduli stacks.
We finally conclude that the moduli stacks of $\sigma_\epsilon, \sigma_{\widetilde{X}},$ and $\sigma_X$ are all Artin stacks of finite type over $\mathbb{C}$.

\subsection{Acknowledgement}
I would like to express my deep gratitude to Arend Bayer for his very valuable help and long discussions throughout the development of this work.
Thanks to Antony Maciocia and Chunyi Li for their helpful comments and suggestions.
I am also grateful to Nicolás Vilches for pointing out a mistake in the first arXiv version of this paper. 

The author was supported by the ERC Consolidator Grant Wall-CrossAG, no.\ 819864.

\printglossary[title= List of Notations]

\section{Preliminaries}
In this section, we review the notions of weak and Bridgeland stability conditions and bounded t-structures.

Let $\mathcal{D}$ be a triangulated category. We fix a finite rank lattice $\Lambda$ and a surjective group homomorphism $v\colon {\rm K}(\mathcal{D}) \longrightarrow \Lambda$.

\begin{definition}
Let $\mathcal{A}$ be an abelian category. A group homomorphism $Z\colon {\rm K}(\mathcal{A}) \longrightarrow \mathbb{C}$ is said to be a weak stability function (resp. stability function) if for any nonzero object $E\in \mathcal{A}$, we have $\Im Z(E) \geq 0$ with $\Im Z(E) = 0 \Rightarrow \Re Z(E) \leq 0$ (resp. $ \Re Z(E) < 0$)
\end{definition}

\begin{definition}\label{stability}
A weak stability condition (resp. Bridgeland stability condition) on $\mathcal{D}$ is a pair $\sigma = (Z,\mathcal{A})$ consisting of a group homomorphism (called the central charge of $\sigma$) $Z \colon \Lambda \longrightarrow \mathbb{C}$ and a heart $\mathcal{A}$ of a bounded t-structure on $\mathcal{D}$, such that the following conditions hold:
\begin{enumerate}
    \item[(a)] The composition $Z \circ v \colon {\rm K}(\mathcal{A})={\rm K}(\mathcal{D}) \longrightarrow \Lambda \longrightarrow \mathbb{C}$ is a weak stability function (resp. stability function) on $\mathcal{A}$. This gives a notion of slope: for any $E \in \mathcal{A}$, we set
    \begin{equation}
    \mu_\sigma(E)=\mu_Z(E):=
    \begin{cases}
    \frac{-\Re Z(E)}{\Im Z(E)}, & \text{if } \Im Z(E)>0, \\
    +\infty, & \text{if } \Im Z(E)=0.
    \end{cases}
    \end{equation}
    We say that an object $0\neq E \in \mathcal{A}$ is $\sigma$-semistable (resp. $\sigma$-stable) if for every nonzero proper subobject $F \subset E$, we have $\mu_\sigma (F) \leq \mu_\sigma (E)$ (resp. $\mu_\sigma (F) < \mu_\sigma (E)$).
    
    \item[(b)] Any object $0\neq E \in \mathcal{A}$ admits a Harder--Narasimhan filtration in $\sigma$-semistable ones; for any $0\neq E \in \mathcal{A}$, there is a filtration $0=E_0 \subset E_1\subset \dots \subset E_{n-1} \subset E_n=E$, such that the quotients $A_i:=E_i/E_{i-1}$ are $\sigma$-semistable (called the HN factors) and we have an inequality of their slopes $\mu_\sigma(A_1)> \dots >\mu_\sigma(A_n)$.
    \item[(c)] (Support property) There exists a quadratic form $Q$ on $\Lambda \otimes \mathbb{R}$ such that $Q|_{\ker Z}$ is negative definite, and $Q(v(E))\geq 0$, for all $\sigma$-semistable $E \in \mathcal{A}$.
\end{enumerate}
Any pair $\sigma=(Z,\mathcal{A})$ satisfying the conditions (a) and (b) is said to be a weak pre-stability condition (resp. pre-stability condition).
\end{definition}

We now briefly review the theory of torsion pairs and tilting introduced in \cite{HRS96}.

\begin{definition} 
Given an abelian category $\mathcal{A}$. A torsion pair in $\mathcal{A}$ is a pair $(\mathcal{T} , \mathcal{F})$ of full additive subcategories satisfying the following conditions:
\begin{enumerate}
    \item[(1)] $\Hom(\mathcal{T} , \mathcal{F})  = 0$
    \item[(2)] For every $A \in \mathcal{A}$, there exists a short exact sequence
    
$$ 0 \longrightarrow T \longrightarrow  A \longrightarrow  F \longrightarrow 0$$
with $ T \in \mathcal{T}$ and $F \in \mathcal{F}$.
\end{enumerate}
\end{definition}

If we have a torsion pair in an abelian category, we can obtain a tilted heart. 
More precisely, we have the following result.
\begin{theorem}[{\normalfont\cite{HRS96}}]\label{HRStilt}
Let $\mathcal{A}$ be the heart of bounded t-structure in $\mathcal{D}$ with a torsion pair $(\mathcal{T}, \mathcal{F})$. 
The category
$$\mathcal{A}^\sharp := \{ E \in \mathcal{D} \ | \ H^i(E) = 0 \ {\rm if} \ i \neq 0, -1, H^0(E) \in \mathcal{T} , H^{-1}(E) \in  \mathcal{F} \}$$ 
is the heart of a bounded t-structure on $\mathcal{D}$.
\end{theorem}

In particular, given a notion of stability and a slope, we can construct a torsion pair by the existence of Harder--Narasimhan filtration.
\begin{proposition}[{\normalfont\cite{HRS96}}]\label{HRS}
Given a weak stability condition $\sigma = (Z,\mathcal{A})$, and a real number $\mu \in \mathbb{R}$, we obtain the following two subcategories of $\mathcal{A}$:
$$ \mathcal{T}^{>\mu}_\sigma := \{E \in \mathcal{A} \ | \ {\rm All \ HN \ factors} \  F \  {\rm of} \ E \ {\rm have \ slopes} \ \mu_\sigma(F) > \mu \} $$
$$\mathcal{F}^{\leq\mu}_\sigma := \{E \in \mathcal{A} \ | \ {\rm All \ HN \ factors} \  F \  {\rm of} \ E \ {\rm have \ slopes} \ \mu_\sigma(F) \leq \mu \}$$
Then the pair $(\mathcal{T}^{>\mu}_\sigma,\mathcal{F}^{\leq \mu}_\sigma)$ gives a torsion pair.
\end{proposition}

\begin{example}

Let $X$ be a projective variety of dimension $n$ and $H$ be a nef divisor. 
Consider the lattice $\Lambda_H$ generated by vectors of the form $(H^{n} \ch_0(E), H^{n-1} \ch_1(E))$ and the natural surjection $v_H \colon {\rm K}(X) \longrightarrow \Lambda_H$.
The pair $(Z_H,\Coh(X))$ given by 
\[ \gls{slope-stability}\]
defines a weak stability condition on ${\rm D}^b(X)$ with respect to $\Lambda_H$; the quadratic form $Q$ can be chosen to be $0$.

This is called the slope stability and we write $\mu_H := \mu_{Z_H} $ for its slope function.
\end{example}

The most common example of heart constructed by tilting is the following:

\begin{definition}\label{cohbeta}
We write \gls{tilted-heart}$ \subseteq \rm{D^b}(X)$ for the tilt of coherent sheaves on a projective variety $X$ with respect to the slope stability $Z_H$ at the slope $\beta \in \mathbb{R}$.
The torsion pair defining this tilt is denoted by 
\[ \gls{T-H}\]
\[ \gls{F-H}\]
\end{definition}

\section{Construction of hearts}
Let $X$ be a projective surface with a single ADE singularity $x_0$ and \gls{crepant-resolution} be its crepant resolution.
Denote the exceptional $(-2)$-curves by \gls{exceptional-curves}, and let \gls{scheme-preimage} be the exceptional locus.

We now start to construct our Bridgeland stability conditions on $\widetilde{X}$. 
In this section, we first focus on the heart of a bounded t-structure. The construction is a generalization of \cite[Sect. 3, 4]{TX22}.

We will construct a heart $\mathcal{B}^0$ of bounded t-structure by tilting the category $^{-1}\Per(\widetilde{X}/X)$ of perverse coherent sheaves.
It is shown to coincides with a tilt of the heart $\Coh^0(\widetilde{X})$ (see Proposition \ref{HRS}).
We also classify simple objects in $\mathcal{B}^0$ supported on the exceptional locus.

Let $H$ be an ample divisor on $X$,
and we can then define the slope stability $\mu_{\pi^\ast H}$ on $\Coh(\widetilde{X})$ with respect to the nef divisor $\pi^\ast H$.
We first recall from \cite{VdB04} that \gls{perverse-coherent-sheaves} is a tilt of $\Coh(\widetilde{X})$ with respect to the torsion pair $(\mathcal{T}_P,\mathcal{F}_P)$, where
\[ \gls{T_P}\]
\[ \gls{F_P}\]

We then want to tilt the perverse coherent sheaves once more.
By abuse of notation, we define the slope stability  with respect to $\pi^\ast H$ as \[\gls{slope-function}:= \begin{cases}
    \frac{\ch_1(E)\cdot (\pi^\ast H)}{\ch_0(E)}, & \text{when } \ch_0(E) \neq 0\\
    +\infty , & \text{when } \ch_0(E)=0 
\end{cases}
\]
We need to check that $^{-1}\Per(\widetilde{X}/X)$ admits Harder--Narasimhan property with respect to $\mu_{\pi^\ast H}$.

\begin{lemma}\textnormal{(cf. \cite[Lemma 3.6]{Tod13})}
Any object in $^{-1}\Per(\widetilde{X}/X)$ admits a Harder--Narasimhan
filtration with respect to $\mu_{\pi^\ast H}$-stability.
\end{lemma}
\begin{proof}
First, by \cite[Prop. 3.2.7, 3.3.1]{VdB04}, 
we see that the category of perverse coherent sheaves is equivalent to a category $\Coh(\mathcal{M})$ of $\mathcal{M}$-modules for a certain sheaf of algebra $\mathcal{M}$ on $X$,
and in particular is Noetherian.

We will use the analogue of \cite[Prop. 2.4]{Bri07} for weak stability conditions.
Assume the contrary that there exists an infinite sequence $\cdots \subseteq E_i \subseteq \cdots \subseteq E_2 \subseteq E_1$ 
such that for all $i$, $\mu_{\pi^\ast H}(E_{i+1}) > \mu_{\pi^\ast H}(E_i/E_{i+1})$.

Since every perverse coherent sheaf has non-negative $\ch_0$,
we may assume $\ch_0(E_i)$ is constant.
But then $\mu_{\pi^\ast H}(E_i/E_{i+1}) = \infty$, which is a contradiction.
\end{proof}

The Harder--Narasimhan property guarantees that the pair of subcategories defined below is a torsion pair in  $^{-1}\Per(\widetilde{X}/X)$:
\[\gls{T-pi-H-P}\]
\[\gls{F-pi-H-P}\]

Therefore, we obtain a double tilted category $\mathcal{B}^0_{\pi^\ast H,P} =  \langle \, \mathcal{F}^{\leq 0}_{\pi^\ast H,P}[1], \mathcal{T}^{>0}_{\pi^\ast H,P} \, \rangle$.

On the other hand, note that every sheaf in $\mathcal{F}_P$ is a torsion sheaf and hence lies in $\Coh^0(\widetilde{X}):=\Coh^0_{\pi^\ast H}(\widetilde{X})$ (see Definition \ref{cohbeta}).
We can then consider the category \gls{F_0} and its left orthogonal complement \gls{T_0}.

\begin{lemma}
The pair $(\mathcal{T}_0,\mathcal{F}_0)$ is a torsion pair in $\Coh^0(\widetilde{X})$.
\end{lemma}
\begin{proof}
The semiorthogonality is clear. 
Now consider an object $E \in \Coh^0(\widetilde{X})$. 
It fits into the distinguished triangle $H^{-1}(E)[1]\longrightarrow E\longrightarrow H^0(E)$.

As $(\mathcal{T}_P,\mathcal{F}_P)$ is a torsion pair in $\Coh(\widetilde{X})$ we obtain a short exact sequence $$ 0 \longrightarrow T_0 \longrightarrow H^0(E) \longrightarrow F_0  \longrightarrow 0$$ with $T_0 \in \mathcal{T}_P$ and $F_0\in \mathcal{F}_P$. 
Note that although this triangle arises from a short exact sequence in $\Coh(\widetilde{X})$,
it also gives short exact sequence in $\Coh^0(\widetilde{X})$.
This can be seen as follows:

It suffices to show that $T_0 \in \mathcal{T}^{> 0}_{\pi^\ast H}$ as $\mathcal{T}^{> 0}_{\pi^\ast H} \subseteq \Coh^0(\widetilde{X})$.
Assume the contrary that there is a subsheaf $A \subseteq T_0$ in $\Coh(\widetilde{X})$ such that the quotient $T_0/A$ has non-positive slope.
Then the octahedral axiom gives a short exact sequence $ 0 \longrightarrow T_0/A \longrightarrow H^0(E)/A \longrightarrow F_0  \longrightarrow 0$.

As $\mathbf{R}^0\pi_\ast F_0 = 0$, $F_0$ does not contribute to the slope of $H^0(E)/A$, we see that 
\[\mu_{\pi^\ast H}(H^0(E)/A) = \mu_{\pi^\ast H}(T_0/A)\leq0,\]
which is a contradiction as $H^0(E) \in \mathcal{T}^{> 0}_{\pi^\ast H}$.

In particular, $H^0(E) \longrightarrow F_0$ is surjective in $\Coh^0(\widetilde{X})$.
By the octahedral axiom we have a triangle $C \longrightarrow E \longrightarrow F_0$. As $E \longrightarrow H^0(E) \longrightarrow F_0$ is a surjection in $\Coh^0(\widetilde{X})$, we see that $C$ lies in $\Coh^0(\widetilde{X})$.
Finally, it is trivial that $\Hom(C,\mathcal{F}_0) = 0$ and so $C \in \mathcal{T}_0$.
\end{proof}

Denote the tilted heart of $\Coh^0(\widetilde{X})$ with respect to this pair by $ \mathcal{B}^0_{\pi^\ast H,-1}$,
and denote the $i$-th cohomology with respect to a heart $\mathcal{A}$ by \gls{cohomology-functor}.
For simplicity, we write $H^i$ for $H^i_{\Coh(\widetilde{X})}$ and $H^i_P$ for $H^i_{^{-1}\Per(\widetilde{X}/X)}$.

We then compare two double tilted hearts above using the following lemma.

\begin{lemma}\label{doubletilt}
Let $\mathcal{A}$ be an abelian category  with a torsion pair $(\mathcal{T},\mathcal{F})$.
Assume that on the tilted category $\mathcal{A}^\sharp = \langle \mathcal{F}[1],\mathcal{T}\rangle$ we also have another torsion pair $(\mathcal{T}^\sharp,\mathcal{F}^\sharp)$ with $\mathcal{F}^\sharp \subseteq \mathcal{A}$.

Then the double tilted category $(\mathcal{A}^\sharp)^\sharp$ is the tilt of $\mathcal{A}$ with respect to the torsion pair $(\mathcal{T}\cap \mathcal{T}^\sharp,  [\mathcal{F}, \mathcal{F}^\sharp])$, where the bracket $[ \ ]$ denotes the extension closure.
\end{lemma}
\begin{proof}
It is clear that $(\mathcal{A}^\sharp)^\sharp \subseteq \langle \mathcal{A}, \mathcal{A}[1] \rangle$.
By \cite[Lemma 1.1.2]{Pol07} the heart $(\mathcal{A}^\sharp)^\sharp$ is a tilt of $\mathcal{A}$, and moreover the torsion part of this tilting is $(\mathcal{A}^\sharp)^\sharp \cap \mathcal{A} =\mathcal{T}^\sharp\cap \mathcal{A} = \mathcal{T}^\sharp \cap\mathcal{T}$.

It remains the show that $(\mathcal{T}\cap \mathcal{T}^\sharp)^\perp = [\mathcal{F}, \mathcal{F}^\sharp]$. The inclusion $(\mathcal{T}\cap \mathcal{T}^\sharp)^\perp \supseteq [\mathcal{F}, \mathcal{F}^\sharp]$ is clear.

To show that $(\mathcal{T}\cap \mathcal{T}^\sharp)^\perp \subseteq [\mathcal{F}, \mathcal{F}^\sharp]$, we consider an object $E \in (\mathcal{T}\cap \mathcal{T}^\sharp)^\perp = (\mathcal{A}^\sharp)^\sharp[-1] \cap \mathcal{A}$.
As $E[1] \in (\mathcal{A}^\sharp)^\sharp$, there is a distinguished triangle $$ H^{-1}_{\mathcal{A}^\sharp}(E[1])[1]\longrightarrow E[1] \longrightarrow H^0_{\mathcal{A}^\sharp}(E[1]).$$

Now we look at the long exact cohomology sequence with respect to the heart $\mathcal{A}$. It turns out to be the following exact sequence:
$$ 0 \longrightarrow H^0_{\mathcal{A}}(H^{-1}_{\mathcal{A}^\sharp}(E[1])) \longrightarrow E \longrightarrow H^{-1}_{\mathcal{A}}(H^0_{\mathcal{A}^\sharp}(E[1])) \longrightarrow 0. $$

It is clear that we have $H^0_{\mathcal{A}}(H^{-1}_{\mathcal{A}^\sharp}(E[1]))=H^{-1}_{\mathcal{A}^\sharp}(E[1]) \in \mathcal{F}^\sharp$ and $H^{-1}_{\mathcal{A}}(H^0_{\mathcal{A}^\sharp}(E[1])) \in \mathcal{F}$ and the assertion follows.
\end{proof}

\begin{proposition}\label{B0}
These two hearts coincide, that is, $\mathcal{B}^0_{\pi^\ast H,P}= \mathcal{B}^0_{\pi^\ast H,-1}$. 
\end{proposition}

\begin{proof}
Note first that $\mathcal{F}^{\leq 0}_{\pi^\ast H,P} \subset \Coh(\widetilde{X})$.
Indeed, for $F \in \mathcal{F}^{\leq 0}_{\pi^\ast H,P} \subseteq ^{-1}\Per(\widetilde{X}/X)$, 
we consider the distinguished triangle $H^{-1}(F)[1] \longrightarrow F \longrightarrow H^0(F)$.
As $H^{-1}(F) \in \mathcal{F}_P$, it is either $0$ or a torsion coherent sheaf and hence has slope $+\infty$, but as $F \in \mathcal{F}^{\leq 0}_{\pi^\ast H,P}$ there cannot be such a torsion subobject.

We then see that $H^{-1}(F) = 0$, and hence the assumptions of Lemma \ref{doubletilt} are all satisfied.
Applying Lemma \ref{doubletilt} twice, we see that both $\mathcal{B}^0_{\pi^\ast H,P}$ and $\mathcal{B}^0_{\pi^\ast H,-1}$ are obtained from $\Coh(\widetilde{X})$ by a single tilt. 

As a torsion pair is determined by its torsion part,
it suffices to prove that the torsion parts of $\mathcal{B}^0_{\pi^\ast H,P}$ and $\mathcal{B}^0_{\pi^\ast H,-1}$ as tilts of $\Coh(\widetilde{X})$ are the same, that is, 
\[\mathcal{B}^0_{\pi^\ast H,P} \cap \Coh(\widetilde{X}) = \mathcal{B}^0_{\pi^\ast H,-1} \cap \Coh(\widetilde{X}).\]

Given $E \in \mathcal{B}^0_{\pi^\ast H,P} \cap \Coh(\widetilde{X})$, we first consider the exact triangle $F[1] \longrightarrow E \longrightarrow T$ with $F \in \mathcal{F}^{\leq 0}_{\pi^\ast H,P}$ and $T \in \mathcal{T}^{>0}_{\pi^\ast H,P}$ given by the definition of $\mathcal{B}^0_{\pi^\ast H,P}$. 
Taking the long exact cohomology sequence with respect to $\Coh(\widetilde{X})$,
we see that $E = H^0(E) = H^0(T)=T$  and $F=0$.

By Lemma \ref{doubletilt}, we know that $\mathcal{B}^0_{\pi^\ast H,-1} \cap \Coh(\widetilde{X}) = \mathcal{T}_0 \cap \mathcal{T}^0_{\pi_\ast H}$.
We first claim that $E \in \mathcal{T}^{>0}_{\pi^\ast H}$, that is, there does not exist any quotient of $E$ in $\Coh(\widetilde{X})$ with non-positive slope.

Assume the contrary that there is a quotient $f \colon E \twoheadrightarrow Q$ in $\Coh(\widetilde{X})$ with $\mu_{\pi^\ast H}(Q) \leq 0$.
Note that the torsion part $\mathcal{T}_P$ is closed under quotient so $Q \in \mathcal{T}_P \subseteq {^{-1}\Per(\widetilde{X}/X)}$.

Let $K:= \ker f$ and consider the long exact cohomology sequence of $K \longrightarrow E \longrightarrow Q$ with respect to $^{-1}\Per(\widetilde{X}/X)$.
This gives us an exact sequence
$$ 0 \longrightarrow H^0_P(K) \longrightarrow E \longrightarrow Q
\longrightarrow H^1_P(K)
\longrightarrow 0. $$

As $H^1_P(K) \in \mathcal{F}_P[1]$, it has no contribution to the slope, and therefore, we see that $\mu_{\pi^\ast H}(\Img H^0_P(f))=\mu_{\pi^\ast H}(Q) \leq 0$,
which is a contradiction as $T$ cannot have a quotient in $^{-1}\Per(\widetilde{X}/X)$ with non-positive slope.
Now, $E$ is contained in $\mathcal{T}^{>0}_{\pi^\ast H}\subseteq \Coh^0(\widetilde{X})$ and we have $\Hom(E,\mathcal{F}_0)$,that is, $E \in \mathcal{T}_0$.

Conversely, given $E \in \mathcal{B}^0_{\pi^\ast H,-1} \cap \Coh(\widetilde{X})$,
we consider similarly the distinguished triangle $F[1] \longrightarrow E \longrightarrow T$ with $F \in \mathcal{F}^{\leq 0}_{\pi^\ast H}$ and $T \in \mathcal{T}^{> 0}_{\pi^\ast H}$ given by the definition of $\mathcal{B}^0_{\pi^\ast H,-1}$. 
Taking the long exact cohomology sequence with respect to $\Coh(\widetilde{X
})$ tells us again that $E=T$ and $F=0$.
It then suffices to show that there does not exist a quotient of $E$ in $^{-1}\Per(\widetilde{X}/X)$ with non-positive slope.

Assume the contrary that there is a quotient $f \colon E \twoheadrightarrow Q$ in $^{-1}\Per(\widetilde{X}/X)$ with $\mu_{\pi^\ast H}(Q) \leq 0$.
Let $K:= \ker f$ and consider the long exact cohomology sequence of $K \longrightarrow E \longrightarrow Q$ with respect to $\Coh(\widetilde{X})$.
This gives us an exact sequence
$$ 0 \longrightarrow H^{-1}(Q) \longrightarrow K \longrightarrow E
\longrightarrow H^0(Q)
\longrightarrow 0. $$
 
Let $Z$ be the central charge of slope stability on $\Coh(\widetilde{X})$.
Then $$Z(H^0(Q))=Z(E)-[Z(K)-Z(H^{-1}(Q))]=Z(E)-Z(K),$$
and hence $\mu_{\pi^\ast H}(H^0(Q))=\mu_{\pi^\ast H}(Q)\leq0$, which is a contradiction as $E$ cannot have such a quotient in $\Coh(\widetilde{X})$.
\end{proof}
 
The result above can be understood via the picture below.

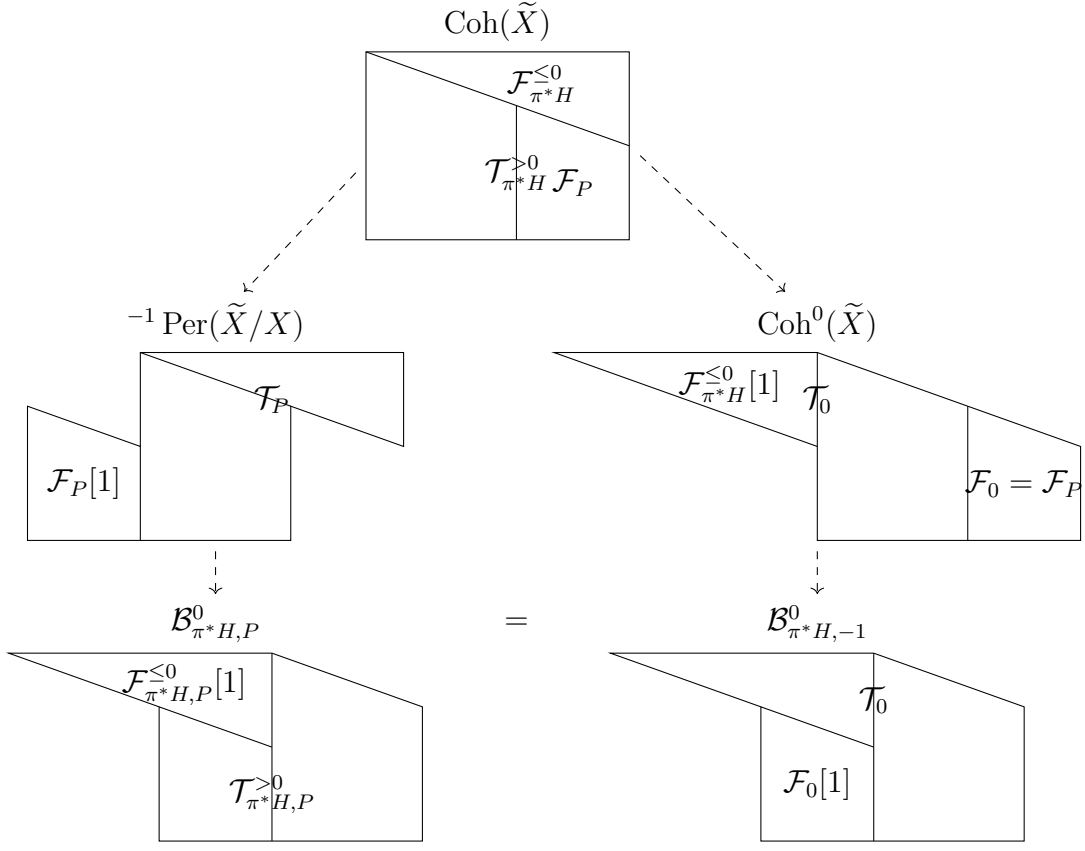
\begin{figure}[H]
\centering
\begin{tikzpicture}
\def\rectangleWidth{3.5} 
\def\rectangleHeight{2.5} 
\def\ratioA{4/7} 
\def\ratioC{1/2} 
\def\diagramColumnSep{4} 
\def\diagramRowSep{4} 
\tikzset{
  otherDiagram/.style={every coordinate/.try}
}

\coordinate (O) at (0, 0);
\coordinate (B) at (\rectangleWidth, 0);
\coordinate (E) at (0, \rectangleHeight);
\coordinate (D) at (B |- E);
\coordinate (A) at ($(O)!{\ratioA}!(B)$);
\coordinate (C) at ($(D)!{\ratioC}!(B)$);
\coordinate (F) at ($(E)!{\ratioA}!(C)$);
\coordinate (labelUR) at (barycentric cs:E=1,D=1,C=1);
\coordinate (labelLL) at (barycentric cs:O=1,A=1,F=1,E=1);
\coordinate (labelLR) at (barycentric cs:A=1,B=1,C=1,F=1);
\draw[name path=contourA] (O) rectangle (D);
\draw (C) -- (E);
\draw (A) -- (F);
\node at (labelUR) {$\mathcal{F}^{\leq 0}_{\pi^\ast H}$}; 
\node at ($(A)!.5!(F)$) {$\mathcal{T}^{> 0}_{\pi^\ast H}$}; 
\node at (labelLR) {$\mathcal{F}_P$}; 
\node[above] at ($(E)!.5!(D)$) (labelDiagramA) {$\Coh(\widetilde{X})$}; 

\begin{scope}[every coordinate/.style={shift={({-\diagramColumnSep+\ratioA*\rectangleWidth/2},-\diagramRowSep)}}]
\draw[name path global=contourB] ([otherDiagram]O) -- ([otherDiagram]A) -- ([otherDiagram]F) -- ([otherDiagram]C) -- ([otherDiagram]D) -- ([otherDiagram]E) -- ([otherDiagram]$(C)-(B)$) -- ([otherDiagram]$(F)-(B)$) -- ([otherDiagram]$(A)-(B)$) -- cycle;
\draw ([otherDiagram]F) -- ([otherDiagram]E);
\draw ([otherDiagram]$(C)-(B)$) -- ([otherDiagram]O);
\node at ([otherDiagram]$(labelLR)-(B)$) {$\mathcal{F}_P[1]$}; 
\node at ([otherDiagram]$(E)!.5!(C)$) {$\mathcal{T}_P$}; 
\node[above] at ([otherDiagram]$(E)!{\ratioA/2}!(D)$) (labelDiagramB) {$^{-1}\Per(\widetilde{X}/X)$}; 
\end{scope}

\begin{scope}[every coordinate/.style={shift={({-\diagramColumnSep+(\ratioA+1)*\rectangleWidth/2},{-2*\diagramRowSep})}}]
\draw ([otherDiagram]O) -- ([otherDiagram]A) -- ([otherDiagram]F) -- ([otherDiagram]E) -- ([otherDiagram]$(E)-(B)$) -- ([otherDiagram]$(F)-(B)$) -- ([otherDiagram]$(A)-(B)$) -- cycle;
\draw ([otherDiagram]E) -- ([otherDiagram]O);
\draw ([otherDiagram]$(F)-(B)$) -- ([otherDiagram]$(C)-(B)$);
\node at ([otherDiagram]$(O)!{\ratioC/2}!(E)$) {$\mathcal{T}^{>0}_{\pi^\ast H,P}$}; 
\node at ([otherDiagram]$(labelUR)-(B)$) {$\mathcal{F}^{\leq 0}_{\pi^\ast H,P}[1]$}; 
\node[above] at ([otherDiagram]$(E)!{(\ratioA-1)/2}!(D)$) (labelDiagramC) {$\mathcal{B}^0_{\pi^\ast H,P}$}; 
\end{scope}

\begin{scope}[every coordinate/.style={shift={({\diagramColumnSep+(\ratioA+1)*\rectangleWidth/2},{-2*\diagramRowSep})}}]
\draw ([otherDiagram]O) -- ([otherDiagram]A) -- ([otherDiagram]F) -- ([otherDiagram]E) -- ([otherDiagram]$(E)-(B)$) -- ([otherDiagram]$(F)-(B)$) -- ([otherDiagram]$(A)-(B)$) -- cycle;
\draw ([otherDiagram]E) -- ([otherDiagram]O);
\draw ([otherDiagram]$(F)-(B)$) -- ([otherDiagram]$(C)-(B)$);
\node at ([otherDiagram]$(labelLR)-(B)$) {$\mathcal{F}_0[1]$}; 
\node at ([otherDiagram]$(O)!{(1+\ratioC)/2}!(E)$) {$\mathcal{T}_0$};
\node[above] at ([otherDiagram]$(E)!{(\ratioA-1)/2}!(D)$) (labelDiagramD) {$\mathcal{B}^0_{\pi^\ast H,-1}$}; 
\end{scope}

\begin{scope}[every coordinate/.style={shift={({\diagramColumnSep+\ratioA*\rectangleWidth},-\diagramRowSep)}}]
\draw[name path global=contourE] ([otherDiagram]O) -- ([otherDiagram]B) -- ([otherDiagram]C) -- ([otherDiagram]E) -- ([otherDiagram]$(E)-(B)$) -- ([otherDiagram]$(C)-(B)$) -- cycle;
\draw ([otherDiagram]A) -- ([otherDiagram]F);
\draw ([otherDiagram]E) -- ([otherDiagram]$(C)-(B)$);
\node at ([otherDiagram]$(labelUR)-(B)$) {$\mathcal{F}^{\leq 0}_{\pi^\ast H}[1]$}; 
\node at ([otherDiagram]labelLR) {$\mathcal{F}_0=\mathcal{F}_P$};
\node at ([otherDiagram]$(O)!{(1+\ratioC)/2}!(E)$) {$\mathcal{T}_0$};
\node[above] at ([otherDiagram]E) (labelDiagramE) {$\Coh^0(\widetilde{X})$};
\end{scope}

\path[name path=arrowAB] (labelDiagramA) -- (labelDiagramB);
\path[name intersections={of=arrowAB and contourA}];
\node at (intersection-1) (arrowABStart) {};

\path[name path=arrowBC] (labelDiagramB) -- (labelDiagramC);
\path[name intersections={of=arrowBC and contourB}];
\node at (intersection-1) (arrowBCStart) {};

\path[name path=arrowAE] (labelDiagramA) -- (labelDiagramE);
\path[name intersections={of=arrowAE and contourA}];
\node at (intersection-2) (arrowAEStart) {};

\path[name path=arrowED] (labelDiagramE) -- (labelDiagramD);
\path[name intersections={of=arrowED and contourE}];
\node at (intersection-1) (arrowDEStart) {};

\draw[dashed, ->] (arrowABStart) -- (labelDiagramB);
\draw[dashed, ->] (arrowBCStart) -- (labelDiagramC);
\draw[dashed, ->] (arrowAEStart) -- (labelDiagramE);
\draw[dashed, ->] (arrowDEStart) -- (labelDiagramD);

\node at ($(labelDiagramC)!.5!(labelDiagramD)$) {$=$};
\end{tikzpicture}
\end{figure}

From now on we write \gls{heart-B0} for this heart, and \gls{torsion-part-B0} and \gls{torsion-free-part-B0} for its torsion part and torsion-free part respectively.

The same argument in Lemma \ref{B0} proves the following fact:

\begin{lemma}\label{FF}
$\mathcal{F}^{\leq \alpha}_{\pi^\ast H,P} \subseteq \mathcal{F}^{\leq \alpha}_{\pi^\ast H}$ for every $\alpha \in \mathbb{R}$.    
\end{lemma}
\begin{proof}
Let $F \in \mathcal{F}^{\leq 0}_{\pi^\ast H,P}$.
If $H^{-1}(F) \neq 0$, then $H^{-1}(F) \subseteq F$ would be a subobject with $\mu_{\pi^\ast H }(H^{-1}(F)) = + \infty$,
and therefore $F= H^0(F)$ is a coherent sheaf.

Assume the contrary that there is a subsheaf $G \subseteq F$ with $\mu_{\pi^\ast H }(G) > \alpha$.
Consider the following exact sequence:
$$ 0 \longrightarrow H^0_P(G) \longrightarrow F \longrightarrow F/G
\longrightarrow H^1_P(G)
\longrightarrow 0. $$
Let $Q:= \ker (F/G \longrightarrow H^1_P(G)$.
Then since $H^1_P(G) \in \mathcal{F}_P[1]$ has no contribution to the slope, there is a subobject  $H^0_P(G)$ of $F$ in ${^{-1}\Per(\widetilde{X}/X)}$ with $\mu_{\pi^\ast H}(H^0_P(G))= \mu_{\pi^\ast H}(G) >\alpha$,
which is a contradiction.
\end{proof}

\begin{remark}
The inclusion in Lemma \ref{FF} above is strict.
For example, let $\alpha =0$ and let $\mathcal{I}_x$ be the ideal sheaf of a point $x \in C_1$.
Consider the following commutative diagram:
\[
\begin{tikzcd}
\mathcal{I}_x \arrow[r, hook] \arrow[d, dotted] & \mathcal{O}_{\widetilde{X}} \arrow[r, two heads] \arrow[d, two heads] & \mathcal{O}_x \\
\mathcal{O}_{C_1}(-1) \arrow[r, hook]           & \mathcal{O}_{C_1} \arrow[ru, two heads]                               &              
\end{tikzcd}
\]
Then, the sheaf $\mathcal{I}_x$ admits a nonzero morphism to $\mathcal{O}_{C_1}(-1)$ so $\mathcal{I}_x \notin \mathcal{T}_P$.
Therefore, $\mathcal{I}_x \notin {^{-1}\Per(\widetilde{X}/X)}$ and hence can not lie in $\mathcal{F}^{\leq0}_{\pi^\ast H,P}$.
\end{remark}

In the remaining of this section, we will classify the simple objects in $\mathcal{B}^0$ supported on the exceptional locus $\Pi$.
Recall that a simple object in an abelian category $\mathcal{A}$ is defined to be a nonzero object which has no proper non-zero subobjects.

We will need the following lemma about the pushforward on the numerical Grothendieck group.
For the numerical Grothendieck group  of the singular surface $X$, we defined ${\rm K}_{\num}(X)$ by the quotient of ${\rm K}( \rm{D^b} (X))$ by the right kernel of the Euler characteristic 
$\chi \colon {\rm K}( \rm{D^{perf}}(X)) \times {\rm K}( \rm{D^b} (X)) \longrightarrow \mathbb{Z}$. 

\begin{lemma}\label{kernel}
The pushforward $\pi_\ast \colon {\rm K}_{\num}(\widetilde{X}) \longrightarrow {\rm K}_{\num}(X)$ is surjective, and its kernel is exactly generated by $[\mathcal{O}_{C_i}(-1)]$ for $i= 1, \dots, n$.
\end{lemma}
\begin{proof}
As $\mathbf{R}\pi_\ast \mathcal{O}_{C_i}(-1) = 0$, we know that the kernel of $\pi_\ast$ contains the classes $[\mathcal{O}_{C_i}(-1)]$ for all $i$.
Moreover, since $C_i$ are linearly independent in the Neron-Severi group of $\widetilde{X}$,
the classes $[\mathcal{O}_{C_i}(-1)]$ are linearly independent in $\mathrm{K}_{\num}(\widetilde X)$.

Recall that a class $\alpha$ in $\mathrm{K}_{\num}(\widetilde X)$ is determined by $\ch_0(\alpha), \ch_1(\alpha)$ and $\chi(\alpha)$.
Now for a class $\alpha \in \ker\pi_\ast$,
we have $\ch_0(\alpha)=0$, $\ch_1(\alpha)=\displaystyle \sum_i a_i C_i$ and $\chi(\alpha)=\chi(\pi_\ast \alpha)=0$.
Then the class $\displaystyle \sum_i a_i[\mathcal{O}_{C_i}(-1)]$ and $\alpha$ are the same class in $\mathrm{K}_{\num}(\widetilde X)$.
\end{proof}

\begin{remark}
One may also note that by projection formula,
$\mathbf{R}\pi_\ast$ is essentially surjective on ${\rm D}_{\qco}(\widetilde{X})$.
This together with the formula
$\mathbf{R}\pi_\ast \mathbf{L} \pi^\ast = \id_{{\rm D}_{\qco}(X)}$ shows that 
the functor $\mathbf{L}\pi^\ast \colon {\rm D}_{\qco}(X) \longrightarrow {\rm D}_{\qco}(\widetilde{X})$ is fully faithful and induces a semiorthogonal decomposition ${\rm D}_{\qco}(\widetilde{X}) =\langle \, \ker \mathbf{R}\pi_\ast, {\rm D}_{\qco}(X)  \, \rangle$.

This can be seen as follows. 
The semiorthogonality follows from the adjointness,
and given $E \in {\rm D}_{\qco}(\widetilde{X})$, consider the distinguished trangle $\mathbf{L}\pi^\ast \mathbf{R}\pi_\ast E \xlongrightarrow{f} E \longrightarrow \cone(f)$.
Then its clear that $\cone(f) \in \ker \mathbf{R}\pi_\ast$ and $\mathbf{L}\pi^\ast \mathbf{R}\pi_\ast E \in  \mathbf{L}\pi^\ast ({\rm D}_{\qco}(X))$.
\end{remark}

\begin{lemma}[{\normalfont\cite[Lemma 3.2]{Bri02}}]\label{Bri02}
For $E \in \rm{D}(\widetilde{X})$, $E \in {^{-1}\Per(\widetilde{X}/X)}$ if and only if $E$ satisfies the conditions below.
\begin{enumerate}
    \item $H^i(E) = 0$ for $i \neq 0,-1$.
    \item $\mathbf{R}^1\pi_\ast H^0(E) = 0$ and $\mathbf{R}^0\pi_\ast H^{-1}(E) = 0$.
    \item $\Hom_X(H^0(E),F)=0$ for any sheaf $F$ on $\widetilde{X}$ with $\mathbf{R}\pi_\ast(F)=0$.
\end{enumerate}
\end{lemma}

The above result by Bridgeland implies that $\mathbf{R}\pi_\ast$  
induces an exact functor of abelian categories $\mathbf{R}\pi_\ast \colon {^{-1}\Per(\widetilde{X}/X)} \longrightarrow \Coh(X)$.

Moreover, as the derived pushforward $\mathbf{R}\pi_\ast \colon {\rm D}^b(\widetilde{X}) \longrightarrow {\rm D}^b(X) $ is an exact triangulated functor, the assertion is equivalent to the t-exactness, i.e. $\mathbf{R}\pi_\ast( {^{-1}\Per(\widetilde{X}/X)}) \subseteq \Coh(X)$.
Recall that given $E \in {^{-1}\Per(\widetilde{X}/X)}$, to show $\mathbf{R}\pi_\ast(E) \in \Coh(X)$, it suffices to show that $\mathbf{R}\pi_\ast(H^0(E))$ and $\mathbf{R}\pi_\ast(H^{-1}(E)[1])$ are both in $\Coh(X)$.

\begin{remark}
Lemma \ref{Bri02} implies that if $G$ is a coherent sheaf on $X$,
$\mathbf{L}^0\pi^\ast G:= H^0_{\Coh(\widetilde{X})}(\mathbf{L}\pi^\ast G)$ is a perverse coherent sheaf.
This can be seen by Lemma \ref{Bri02}: 
conditions (i),(ii) are clear, and for (iii), let $F$ be a sheaf with $\mathbf{R}\pi_\ast F=0$. 
Then 
\[\Hom(\mathbf{L}^0\pi^\ast G,F) = \Hom(\mathbf{L}\pi^\ast G,F) = \Hom(G, \mathbf{R}\pi_\ast F) = 0;\]
here the first equality holds because $\Hom(H^{-i}(\mathbf{L}\pi^\ast G)[i],F)$ are negative Ext and hence are $0$.
This particularly implies that the sheaf $\mathcal{O}_{\Pi}$ is a perverse coherent sheaf.
\end{remark}

We also need the following lemma about simple objects in a tilted category.

\begin{lemma}\label{tiltsimple}
Let $E$ be a simple object in an abelian category $\mathcal{A}$.  
Assume that $\mathcal{A}$ admits a torsion pair $(\mathcal{T},\mathcal{F})$.
\begin{enumerate}
\item If $E \in \mathcal{T}$, so $E \in \mathcal{A}^\sharp = \langle \, \mathcal{F}[1], \mathcal{T} \,\rangle$,
and if there is no short exact sequence of the form $$ 0 \longrightarrow F \longrightarrow T \longrightarrow E \longrightarrow 0$$
with $0 \neq F \in \mathcal{F}$ and $T \in \mathcal{T}$,
then $E$ is a simple object in $\mathcal{A}^\sharp$.
\item If $E \in \mathcal{F}$, so $E[1] \in \mathcal{A}^\sharp = \langle \, \mathcal{F}[1], \mathcal{T} \,\rangle$,
and if there is no short exact sequence of the form $$ 0 \longrightarrow E \longrightarrow F \longrightarrow T \longrightarrow 0$$
with $F \in \mathcal{F}$ and $0 \neq T \in \mathcal{T}$,
then $E[1]$ is a simple object in $\mathcal{A}^\sharp$.
\end{enumerate}
\end{lemma}
\begin{proof}
We will prove (i) as (ii) is similar.
Assume that there is a short exact sequence $0 \longrightarrow K \longrightarrow E \longrightarrow Q \longrightarrow 0$ in $\mathcal{A}^\sharp$.
The long exact cohomology sequence with respect to the heart $\mathcal{A}$ is of the form 
$$ 0 \longrightarrow H^{-1}_{\mathcal{A}}(Q) \longrightarrow K \stackrel{f}\longrightarrow E \longrightarrow H^{0}_{\mathcal{A}}(Q) \longrightarrow 0.$$

As $E$ is simple in $\mathcal{A}$, either $\Img f = 0$ or $=E$.
In the former case, we have $H^{-1}_{\mathcal{A}}(Q) = K$,
which means that they are both $0$ as $K \in \mathcal{T}$ and $H^{-1}_{\mathcal{A}}(Q) \in \mathcal{F}$.
In the latter case, $H^0_{\mathcal{A}}(Q) = 0$ and we obtain the short exact sequence $$ 0 \longrightarrow H^{-1}_{\mathcal{A}}(Q) \longrightarrow K \longrightarrow E \longrightarrow 0.$$
\end{proof}

This immediately gives the following lemma.

\begin{lemma}\label{Oxsimple}
If $x \notin \Pi$, then $\mathcal{O}_x$ is simple in $\mathcal{B}^0$.
\end{lemma}
\begin{proof}
We start with the fact that $\mathcal{O}_x$ is simple in $\Coh(\widetilde{X})$. 

As $\ch_0(\mathcal{O}_x)=0$ and $(\pi^\ast H) \cdot \ch_1(\mathcal{O}_x)=0$, there cannot exist a short exact sequence in $\Coh(\widetilde{X})$
$$ 0 \longrightarrow F \longrightarrow T \longrightarrow E \longrightarrow 0$$
with $0 \neq F \in \mathcal{F}^{\leq 0}_{\pi^\ast H}$ and $T \in \mathcal{T}^{> 0}_{\pi^\ast H}$.
By Lemma \ref{tiltsimple}, $\mathcal{O}_x$ is simple in $\Coh^0(\widetilde{X})$.

We know that $\mathcal{O}_x \in \Coh^0(\widetilde{X})$ have no morphism to $\ker \mathbf{R}^0 \pi_\ast$ and hence $\mathcal{O}_x \in \mathcal{T}_0 \subseteq \mathcal{B}^0$.
Therefore, by Lemma \ref{tiltsimple} it remains to show that there's no short exact sequence of the form
\[ 0 \longrightarrow F \longrightarrow T \longrightarrow \mathcal{O}_x \longrightarrow 0\]
with $0\neq F \in \mathcal{F}_0=\ker \mathbf{R}^0 \pi_\ast$ and $T\in \mathcal{T}_0$.

Note that  since $F$ is supported on $\Pi$ and $x \notin Pi$, we have $\Ext^1(\mathcal{O}_x, F)=0$,
and hence $T=F \oplus \mathcal{O}_x$, which is a contradiction to $\Hom(T,F)=0$.
\end{proof}

Our claim is the following:

\begin{proposition} \label{simpleobjects}
The only simple objects in $\mathcal{B}^0$ supported on the exceptional locus $\Pi$ are $\mathcal{O}_{C_i}(-1)[1]$ and $\mathcal{O}_{\Pi}$.
\end{proposition}

We start with a useful lemma.

\begin{lemma}\label{H0}
For any $F \in \mathcal{B}^0$ with $\mathbf{R}\pi_\ast F= 0$, we have $H^0(F)=0$. 
\end{lemma}
\begin{proof}
As $H^0(F)$ is in the torsion part $\mathcal{T}_0$, by definition it can't have nonzero morphism to objects in  $\ker \mathbf{R}\pi_\ast$.

Since $\mathbf{R}\pi_\ast H^0(F)$ is a quotient of $\mathbf{R}\pi_\ast F= 0$,
we see that $\mathbf{R}\pi_\ast H^0(F)=0$ and hence $H^0(F)$ can only be $0$.
\end{proof}

We separate the proof of Proposition \ref{simpleobjects} into two parts.

\begin{lemma}\label{simple1}
$\mathcal{O}_{C_i}(-1)[1]$ and $\mathcal{O}_{\Pi}$ are simple objects in $\mathcal{B}^0$.
\end{lemma}
\begin{proof}
For the first assertion, we consider a short exact sequence in $\mathcal{B}^0$
\[0 \longrightarrow K \longrightarrow \mathcal{O}_{C_i}(-1)[1] \longrightarrow Q \longrightarrow 0.\]

As $\mathcal{O}_{C_i}(-1)\in\ker \mathbf{R}\pi_\ast$ and $\mathbf{R}\pi_\ast \colon \mathcal{B}^0 \longrightarrow \Coh^0_H(X)$ is exact (see Proposition \ref{pushforward}),
we see that $\mathbf{R}\pi_\ast K = \mathbf{R}\pi_\ast Q =0$.

Then by Lemma \ref{H0} and the long exact cohomology sequence with respect to $\Coh(\widetilde{X})$, we see that $H^{-1}(K)=K[-1]$ is a subsheaf of $\mathcal{O}_{C_i}(-1)$.
But as $K$ is also in the kernel of $\mathbf{R}\pi_\ast$, $K$ is either $0$ or exactly $\mathcal{O}_{C_i}(-1)[1]$.

For the second assertion, we will apply Lemma \ref{tiltsimple}(i) to the object $\mathcal{O}_{\Pi} \in {^{-1}\Per(\widetilde{X}/X)}$,
viewing $\mathcal{B}^0=  \mathcal{B}^0_{\pi^\ast H,P} =\langle \, \mathcal{F}^{\leq 0}_{\pi^\ast H,P}[1], \mathcal{T}^{>0}_{\pi^\ast H,P} \,\rangle$.
We first need to verify that $\mathcal{O}_{\Pi}$ is simple in ${^{-1}\Per(\widetilde{X}/X)}$.
We consider a short exact   sequence $ 0 \longrightarrow K \longrightarrow \mathcal{O}_{\Pi} \longrightarrow Q \longrightarrow 0$ in ${^{-1}\Per(\widetilde{X}/X)}$.
As $\mathbf{R}\pi_\ast \mathcal{O}_{\Pi} = \mathcal{O}_{x_0}$ is simple in $\Coh(X)$,  by the t-exactness of $\mathbf{R}\pi_\ast$ we know that either $\mathbf{R}\pi_\ast K=0, \mathbf{R}\pi_\ast Q= \mathcal{O}_{x_0}$ or $\mathbf{R}\pi_\ast Q= 0, \mathbf{R}\pi_\ast K = \mathcal{O}_{x_0}$.

By the long exact cohomology sequence we have $H^{-1}(K)=0$,
so $K$ is a coherent sheaf and hence $K \in {^{-1}\Per(\widetilde{X}/X)} \cap \Coh(\widetilde{X}) = \mathcal{T}_P$.
However, in the former case, $\mathbf{R}\pi_\ast K =0$ implies that $K \in \mathcal{F}_P$,
then $K$ can only be $0$.

In the latter case, we assume that $Q$ is nonzero.
Since $\mathbf{R}\pi_\ast Q =0 $, Lemma \ref{H0} tells us that $H^0(Q)=0$.
Now the long exact cohomology sequence gives the short exact sequence $ 0 \longrightarrow H^{-1}(Q) \longrightarrow H^0(K) \longrightarrow \mathcal{O}_{\Pi} \longrightarrow 0$ in $\Coh(\widetilde{X})$.
Consider the truncation distinguished triangle $\tau^{\leq -1}\mathbf{L}\pi^\ast \mathcal{O}_{x_0} \longrightarrow \mathbf{L}\pi^\ast \mathcal{O}_{x_0} \longrightarrow \tau^{\geq 0}\mathbf{L}\pi^\ast \mathcal{O}_{x_0} =\mathcal{O}_{\Pi}$ and apply the functor $\Hom(\bullet , H^{-1}(Q)[1])$.

By adjunction we see that $\Hom(\mathbf{L}\pi^\ast \mathcal{O}_{x_0} , H^{-1}(Q)[1])= 0$ since $\mathbf{R}\pi_\ast H^{-1}(Q)= 0$. 
It is also clear that $\Hom(\tau^{\leq -1}\mathbf{L}\pi^\ast \mathcal{O}_{x_0} , H^{-1}(Q))= 0$;
thus we have $\Ext^1(\mathcal{O}_{\Pi}, H^{-1}(Q)) = 0$,
which means that $K = Q[-1] \oplus \mathcal{O}_{\Pi}$ but this is not in the heart ${^{-1}\Per(\widetilde{X}/X)}$, a contradiction.

It remains to check that there cannot be a short exact sequence of the form
$$ 0 \longrightarrow F \longrightarrow T \longrightarrow \mathcal{O}_{\Pi} \longrightarrow 0,$$
with $F \in \mathcal{F}^{\leq 0}_{\pi^\ast H,P}$ and $T \in \mathcal{T}^{>0}_{\pi^\ast H,P}$.
This is trivial as $\mathcal{O}_{\Pi}$ has no contribution to the slope function $\mu_{\pi^\ast H}$ on perverse coherent sheaves.
\end{proof}

\begin{lemma}\label{simple2}
$\mathcal{O}_{C_i}(-1)[1]$ and $\mathcal{O}_{\Pi}$ are the only simple objects in $\mathcal{B}^0$ supported on the exceptional locus $\Pi$.
\end{lemma}
\begin{proof}
Given a simple object $E$ in $\mathcal{B}^0$ with $\supp E=\Pi$, we know that $E$ must be equal to either $H^0(E)$ or $H^{-1}(E)[1]$.

We first deal with the case that $E=H^0(E)$.
Note that $\mathbf{R}\pi_\ast E$ can not be $0$ otherwise $E=0$, 
so there exists a nonzero morphism $\mathcal{O}_{x_0} \longrightarrow \mathbf{R}\pi_\ast E$, 
which induces by adjunction a nonzero morphism $\mathbf{L}\pi^\ast \mathcal{O}_{x_0} \longrightarrow E$.
Then, applying the functor $\Hom (\bullet, E)$ to the distinguished triangle $\tau^{\leq -1}\mathbf{L}\pi^\ast \mathcal{O}_{x_0} \longrightarrow \mathbf{L}\pi^\ast \mathcal{O}_{x_0} \longrightarrow \mathcal{O}_{\Pi}$, we get a nonzero morphism $\mathcal{O}_{\Pi} \longrightarrow E$.
As they are both simple, this must be an isomorphism.

Now for the case that $E=H^{-1}(E)[1]$, the assertion follows directly from the lemma below.
\end{proof}

\begin{lemma}\label{generatorofker}
Given a coherent sheaf $F$ with $\mathbf{R}\pi_\ast F=0$, then $F$ is an extension of the sheaves $\mathcal{O}_{C_i}(-1)$. 
\end{lemma}
\begin{proof}
Note first that $\ch_1(F) =  a_1C_1+ \cdots a_nC_n$ with all $a_i \geq 0$.
By induction on $\displaystyle \sum_i a_i$, 
it is enough to show that either $\Hom(\mathcal{O}_{C_i}(-1),F)$ or $\Hom(F,\mathcal{O}_{C_i}(-1))$ is nonzero for some $i$,
as $\mathcal{O}_{C_i}(-1)$ are simple in the category $\ker \mathbf{R}\pi_\ast \cap \Coh(\widetilde{X})$.

Assume the contrary that for all $i$, we have $\Hom(\mathcal{O}_{C_i}(-1),F) = 0 = \Hom(F,\mathcal{O}_{C_i}(-1))$.
By the adjunction formula we can compute that the restrictions of canonical divisor $K_{\widetilde{X}}$ on the curves $C_i$ are trivial for all $i$.
Therefore, by Serre duality we have that $\Ext^2(F,\mathcal{O}_{C_i}(-1)) = 0$.

Hence, for every $i$, $\chi(F,\mathcal{O}_{C_i}(-1)) \leq 0$, which by Hirzebruch-Riemann-Roch implies that $\ch_1(F) \cdot C_i \geq 0$.
However, since the intersection matrix of the exceptional curves is negative definite (see, for example, \cite{Reid}), 
we have $$0 > \ch_1(F)^2 = \displaystyle \sum_i a_i(\ch_1(F) \cdot C_i) \geq 0,$$ which is a contradiction.
\end{proof}

Then, Lemma \ref{simple1} and \ref{simple2} combine to prove Proposition \ref{simpleobjects}.

\section{Central charges}
In this section we will show that the central charge $Z_{\pi^\ast H,\beta,z}$ is a stability function on $\mathcal{B}^0$ and that it satisfies the Harder--Narasimhan property under some suitable assumptions.
This extends the results of \cite{TX22} to resolutions of ADE singularities and fixes some gaps in the original proof.

Let $z\in \mathbb{R}$ and $\beta \in \NS_\mathbb{R}(\widetilde{X})$ be a class with $(\pi^\ast H)\cdot \beta =0$.
On ${\rm D}^b(\widetilde{X})$ we define a function 
\[\gls{stability-function} \]

We start with a lemma proving the existence of our choice of $\beta$.

\begin{lemma}\label{betaexists}
There is a class $\beta \in \NS_\mathbb{R}(\widetilde{X})$ such that $(\pi^\ast H)\cdot \beta =0, \beta \cdot C_i>0$ for all $i$, and 
$\beta \cdot \ch_1(\mathcal{O}_{\Pi}) <1$.
\end{lemma}
\begin{proof}
Note that it suffices to show that there are some $b_i>0$ such that $(\displaystyle \sum_i b_iC_i) \cdot C_j >0$ for all $j$.

The exceptional curves $C_i$ span a negative definite convex cone $\sigma \subseteq \NS(\widetilde{X}) \otimes \mathbb{R}$ (see \cite[Theorem A.7]{Reid}) 
and it suffices to prove that $\sigma \cap -\sigma^\vee$ is non-empty.
Assume the contrary. 
There must be a linear function which is positive on $\sigma$ and negative on $-\sigma^\vee$, 
which can be represented by pairing with a vector $v$.

But then as $v$ is postive on $\sigma^\vee$, we know that $v$ lies in $(\sigma^\vee)^\vee = \sigma$ , which implies that $(v,v)\geq 0$.   
This is a contradiction to the negative definiteness and hence our claim is proved.
\end{proof}

We prove that $Z_{\pi^\ast H,\beta,z}$ is a stability function on $\mathcal{B}^0$. 
When $x_0$ is an $A_1$ singularity,
this is a special case of \cite[Lemma 5.2]{TX22}.

\begin{lemma}\label{stabfunc}
The function $Z_{\pi^\ast H,\beta,z}$ is a stability function on $\mathcal{B}^0$, when $z$ and $\beta$ are chosen such that $\beta \cdot C_i>0$ for all $i$, 
$\beta \cdot \ch_1(\mathcal{O}_{\Pi}) <1$, and $z> -\frac{\beta^2}{2}$.
\end{lemma}
\begin{proof}
Given $E \in \mathcal{B}^0$, we need to prove that $\Im Z_{\pi^\ast H,\beta,z}(E)\geq 0$ and that if $\Im Z_{\pi^\ast H,\beta,z}(E)=0$ then $\Re Z_{\pi^\ast H,\beta,z}(E)<0$.

Such $E$ fits into a distinguished triangle $F[1] \longrightarrow E \longrightarrow T$ for some $F \in \mathcal{F}_0$ and $T \in \mathcal{T}_0$.
By the construction of $\Coh^0(\widetilde{X})$, we know that $\Im Z_{\pi^\ast H,\beta,z}(H^0(T)) \geq 0$ and that $\Im Z_{\pi^\ast H,\beta,z}(H^{-1}(T)) \leq 0$.
As $(\pi^\ast H) \cdot C_i = 0$ for all $i$ and $F$ is supported on these exceptional curves, we have 
\[\Im Z_{\pi^\ast H,\beta,z}(E) = \Im Z_{\pi^\ast H,\beta,z}(T) = \Im Z_{\pi^\ast H,\beta,z}(H^0(T)) - \Im Z_{\pi^\ast H,\beta,z}(H^{-1}(T))\geq0.\]

To prove the second assertion, we assume that $\Im Z_{\pi^\ast H,\beta,z}(E)=0$. 
Notice that this is equivalent to  $\Im Z_{\pi^\ast H,\beta,z}(F)= \Im Z_{\pi^\ast H,\beta,z}(H^{-1}(T))= \Im Z_{\pi^\ast H,\beta,z}(H^0(T)) = 0$.

Therefore, it suffices to prove that in $\mathcal{B}^0$, if $A[1] \in \mathcal{F}_0[1], B \in \mathcal{T}^{> 0}_{\pi^\ast H}, C[1] \in \mathcal{F}^{\leq 0}_{\pi^\ast H}[1]$, with $\Im Z_{\pi^\ast H,\beta,z}(A[1])=\Im Z_{\pi^\ast H,\beta,z}(B)=\Im Z_{\pi^\ast H,\beta,z}(C[1])=0$, we have  $\Re Z_{\pi^\ast H,\beta,z}(A[1]), \Re Z_{\pi^\ast H,\beta,z}(B)$, and $ \Re Z_{\pi^\ast H,\beta,z}(C[1])$ are all negative.

For $A\in \mathcal{F}_0=\mathcal{F}_P$, we have $\ch_0(A)=0$, so we only need to show that $$-\ch_2(A[1])+\beta \cdot \ch_1(A[1]) <0.$$
As $\mathbf{R}^0\pi_\ast A=0$,
we have $\mathbf{R}\pi_\ast A =\mathbf{R}^1\pi_\ast A[-1]$ and hence $[\mathbf{R}\pi_\ast A]=-n[x_0]$ for some $n\geq 0$. 
Then since $x_0$ is an ADE singularity, $C_i$ are K-trivial and the relative Todd class is 0.
Now by Grothendieck-Riemann-Roch we have $-n=\ch_2(\mathbf{R}\pi_\ast A) = \pi_\ast \ch_2(A)$.
Since $\pi_\ast$ has no effect on $\ch_2$ on surfaces we see that $-\ch_2(A[1])=\ch_2(A)=-n \leq 0$.
Also, as $\ch_1(A[1])= -\displaystyle \sum_i (a_iC_i)$ for some $a_i \geq 0$ and $\beta \cdot C_i >0$ for all $i$, we in summary see that $\Re Z_{\pi^\ast H,\beta,z}(A)>0$.

For $B \in \mathcal{T}^{> 0}_{\pi^\ast H}$ with $\Im Z_{\pi^\ast H,\beta,z}(B)=0$, we have $\ch_0(B)=0$, and hence $B$ is supported on the union of points and $\Pi$. 
We may then write 
$\ch_1(B) = \displaystyle \sum_i (b_iC_i)$ for some $b_i \geq 0$.
Our assumption is now $\displaystyle \sum_i m_i( \beta \cdot C_i) <1$,  where $m_i$ is the multiplicities of $C_i$ in the scheme-theoretic preimage $\Pi$.
Then \[-\ch_2(B)+\beta \cdot \ch_1(B) = -\ch_2(B) + \displaystyle \sum_i(b_i \beta \cdot C_i)=-\ch_2(B) + \displaystyle \sum_i \frac{b_i}{m_i} m_i(\beta \cdot C_i)<0\]
as $\ch_2(B)$ is the length of $\mathbf{R}^0\pi_\ast B$ and hence $ \ch_2(B)\geq \displaystyle\max_i \{\frac{b_i}{m_i}\} \geq \displaystyle \sum_i \frac{b_i}{m_i} m_i(\beta \cdot C_i)$.

For $C[1] \in \mathcal{F}^{\leq 0}_{\pi^\ast H}[1]$ with $\Im Z_{\pi^\ast H,\beta,z}(C)=0$, the sheaf $C$ must be slope semistable since it has maximal slope in $\mathcal{F}^{\leq 0}_{\pi^\ast H}$.
We can therefore use Bogomolov-Gieseker inequality and Hodge index theorem to see that $z>-\frac{\beta^2}{2}$ implies $\Re Z_{\pi^\ast H,\beta,z}(C[1])<0$ (cf. \cite[Lemma 5.2]{TX22}).
Indeed,
\begin{equation*} 
\begin{split}
\Re  Z_{\pi^\ast H,\beta,z}(C)  &= z\ch_0(C)+\beta \cdot \ch_1(C) - \ch_2(C) \\
&\geq \ch_0(C)(z+ \frac{\beta \cdot \ch_1(C)}{\ch_0(C)} - \frac{\ch_1^2(C)}{2\ch_0^2(C)}) \text{\ (by Bogomolov-Gieseker)}\\
&= \ch_0(C)(z- \frac{(\ch_1(C) -\ch_0(C)\beta)^2}{2\ch_0^2(C)}  +\frac{\beta^2}{2}).
\end{split}
\end{equation*}
We know that $(\ch_1(C) -\ch_0(C)\beta)^2 \leq 0$ by Hodge index theorem with respect to $\pi^\ast H$,
and that $z+\frac{\beta^2}{2}>0$,
so $\Re  Z_{\pi^\ast H,\beta,z}(C[1])<0$, which completes the proof.
\end{proof}

\begin{remark}
In the constructions, we may instead choose to tilt $\Coh(\widetilde{X})$ and $^{-1}\Per(\widetilde{X}/X)$ at any slope $\alpha \in \mathbb{R}$.
The same argument in Section 3 will give us a double tilted heart $\mathcal{B}^\alpha$.

To define a central charge on this heart, we may choose $$Z_{\pi^\ast H,\beta,z}^\alpha(E):=-\ch_2(E)+\beta \cdot \ch_1(E)+z\ch_0(E)+i[(\pi^\ast H)\cdot \ch_1(E) - \alpha \ch_0(E)].$$
The same proof works to show that it is a stability function if the assumption $z> -\frac{\beta^2}{2}$ is replaced with $z - \frac{\alpha^2}{2H^2}> -\frac{\beta^2}{2}$. 
\end{remark}

Let $\mathbb{D}(-) := \mathbf{R}\mathcal{H}om(-, \mathcal{O}_X)$ be the duality functor on $X$.
Before proving the Harder--Narasimhan property for the pair, 
we give the following lemma:

\begin{lemma}\label{dual}
If $T$ is a $0$-dimensional torsion sheaf on $X$,
then $\mathbb{D}(T)$ is contained in $\mathrm{D}^{\geq 2}(X)$,
where $\mathrm{D}^{\geq i}(X):=\{ E \in \mathrm{D}^b(X) \, | \, H^j(E)=0 \text{ if } j < i \}$.
\end{lemma}
\begin{proof}
$T$ must be an extension of skyscraper sheaves.
If $x$ is a smooth point on $X$, then $\mathbb{D}(\mathcal{O}_x) = \mathcal{O}_x[-2]$ so its clear that $\mathbb{D}(\mathcal{O}_x) \in \mathrm{D}^{\geq 2}(X)$.

It remains to show that for the singular point $x_0$,
$\mathbb{D}(\mathcal{O}_{x_0}) \in \mathrm{D}^{\geq 2}(X)$.
It is automatic that $\mathbb{D}(\mathcal{O}_{x_0}) \in \mathrm{D}^{\geq 0}(X)$, 
and $H^0(\mathbb{D}(\mathcal{O}_{x_0})) = \mathcal{H}om(\mathcal{O}_{x_0},\mathcal{O}_X)=0$ as $\mathcal{O}_{x_0}$ is torsion.

Assume the contrary that $H^1(\mathbb{D}(\mathcal{O}_{x_0}))$ is nonzero. 
Then as pushforward onto the open set $X\setminus x_0$ commutes with $\mathbb{D}$,
we see that $H^1(\mathbb{D}(\mathcal{O}_{x_0}))$ is supported on $x_0$, 
and hence there exists a nonzero morphism $\mathcal{O}_{x_0} \longrightarrow H^1(\mathbb{D}(\mathcal{O}_{x_0}))$.

On the other hand, consider a complete intersection $Z$ of Cartier divisors $D_1,D_2$ with $x_0 \in Z$.
Then there is an exact sequence 
$$ 0\longrightarrow \mathcal{O}_X(-D_1-D_2) \longrightarrow \mathcal{O}_X(-D_1) \oplus \mathcal{O}_X(-D_2) \longrightarrow \mathcal{O}_X \longrightarrow \mathcal{O}_Z \longrightarrow 0.$$
This implies that $\mathbb{D}(\mathcal{O}_Z)$ is the complex $ (\mathcal{O}_X \longrightarrow \mathcal{O}_X(D_1) \oplus \mathcal{O}_X(D_2) \longrightarrow\mathcal{O}_X(D_1+D_2))$,
which is quasi-isomorphic to $\mathcal{O}_Z(D_1+D_2)[-2] \simeq \mathcal{O}_Z[-2]$.
Now we look at the composition of morphisms $$\mathbb{D}(\mathcal{O}_Z)[1]= \mathcal{O}_Z[-1] \longrightarrow \mathcal{O}_{x_0}[-1] \longrightarrow H^1(\mathbb{D}(\mathcal{O}_{x_0})) [-1] \longrightarrow \mathbb{D}(\mathcal{O}_{x_0}).$$

First note that the composition $\mathcal{O}_Z \longrightarrow \mathcal{O}_{x_0} \longrightarrow H^1(\mathbb{D}(\mathcal{O}_{x_0}))$ is nonzero, and so is its shift.
Then, applying the functor $\Hom(
\mathcal{O}_Z[-1],\bullet)$ to the distinguished triangle given by truncation $\tau^{\geq 2}\mathbb{D}(\mathcal{O}_{x_0})[-1]\longrightarrow \tau^{\leq 1}\mathbb{D}(\mathcal{O}_{x_0}) \longrightarrow\mathbb{D}(\mathcal{O}_{x_0})$,
since $\Hom(\mathcal{O}_Z[-1],\tau^{\geq 2}\mathbb{D}(\mathcal{O}_{x_0})[-1])=0$, 
we have $\Hom(\mathcal{O}_Z[-1],\tau^{\leq 1}\mathbb{D}(\mathcal{O}_{x_0})) \longrightarrow \Hom(\mathcal{O}_Z[-1],\mathbb{D}(\mathcal{O}_{x_0}))$ is an inclusion.
This shows that $\Hom(\mathbb{D}(\mathcal{O}_Z)[1],\mathbb{D}(\mathcal{O}_{x_0}))\neq 0$.

However, by duality, we also see that $\Hom(\mathbb{D}(\mathcal{O}_Z)[1],\mathbb{D}(\mathcal{O}_{x_0})) = \Hom(\mathcal{O}_{x_0}, \mathcal{O}_Z[-1])=0$, which is a contradiction.
\end{proof}

We now turn to the Harder--Narasimhan property.

\begin{lemma}\label{HNprop}
Given $E \in \mathcal{B}^0$, any sequence of inclusions
\[ 0 = A_0 \hookrightarrow A_1 \hookrightarrow \cdots \hookrightarrow A_j \hookrightarrow \cdots \hookrightarrow E \]
in $\mathcal{B}^0$ with $\Im Z_{\pi^\ast H,\beta,z}(A_j)=0$ for all $j$ terminates.
\end{lemma}
\begin{proof}
For each $j$, consider the distinguished triangles 
\begin{equation}\label{tri:AAB}
A_j \longrightarrow A_{j+1} \longrightarrow B_{j+1}
\end{equation}
\begin{equation}\label{tri:AEE}
A_j \longrightarrow E \longrightarrow E_j.
\end{equation}

The octahedral axiom gives the following exact triangle for each $j$:
\begin{equation}\label{tri:BEE}
B_j \longrightarrow E_{j-1} \longrightarrow E_j. 
\end{equation}

The induced long exact cohomology sequence of (\ref{tri:AAB}) with respect to $\Coh(\widetilde{X})$ gives a sequence of coherent sheaves 
$$ 0 = H^{-1}(A_0) \hookrightarrow H^{-1}(A_1) \hookrightarrow \cdots \hookrightarrow H^{-1}(A_j)  \hookrightarrow \cdots \hookrightarrow H^{-1}(E). $$ 
As $\Coh(\widetilde{X})$ is Noetherian,
$H^{-1}(A_j)$ is constant for $j$ sufficiently large, say that there is some $N_1$ such that $H^{-1}(A_{N_1})=H^{-1}(A_{N_1+1})= \cdots$.

After omitting the first $N_1$ terms, we may quotient all objects in the sequence by the common subobject $H^{-1}(A_{N_1})[1]$ and assume that all $A_j=H^0(A_j)$ are sheaves.
Moreover, as $\Im Z_{\pi^\ast H,\beta,z}(A_j)=0$, we know $A_j$ is torsion for each $j$.

Considering again the long exact cohomology sequence of (\ref{tri:AAB}),
this tells us that for each $j$, the sheaf $H^{-1}(B_j)$ is torsion.
Furthermore, since $\Coh(\widetilde{X})$ is Noetherian, it remains to show that $H^0(B_j)=0$ for sufficiently large $j$.

By Lemma \ref{FF}, $\mathcal{F}^{\leq 0}_{\pi^\ast H, P} \subseteq  \mathcal{F}^{\leq 0}_{\pi^\ast H}$,
and hence there is a torsion-free subsheaf $$H^{-1}_P(B_j) = H^0(H^{-1}_P(B_j)) \subseteq H^{-1}(B_j)$$
which implies that $H^{-1}(B_j)=H^{-1}(H^0_P(B_j))\in \mathcal{F}_P$ for each $j$.
In particular, for each $j$, we see that $B_j$ is a perverse coherent sheaf  since $H^{-1}(B_j) \in \mathcal{F}_P$ and $H^0(B_j) \in \mathcal{T}_P$.

On the other hand, the long exact sequences of (\ref{tri:AEE}) and (\ref{tri:BEE}) yield a sequence of surjections in the Noetherian category $\Coh(\widetilde{X})$:
$$ H^0(E) \twoheadrightarrow H^0(E_1) \twoheadrightarrow H^0(E_2) \twoheadrightarrow \cdots$$
Therefore, $H^0(E_{N_2}) = H^0(E_{N_2+1}) = \cdots$ for some $N_2$.

We then replace the triangle (\ref{tri:AEE}) by $A_j/A_{N_2} \longrightarrow E/A_{N_2}=E_{N_2} \longrightarrow E_j$ and hence we can assume that $H^0(E)=H^0(E_j)$ is constant.
Therefore, the long exact sequence  of (\ref{tri:AEE})
becomes $0 \longrightarrow H^{-1}(E) \longrightarrow H^{-1}(E_j) \longrightarrow A_j \longrightarrow 0$.

Consider the following commutative diagram:

\[
\begin{tikzcd}
 & {H^{-1}(E)[1]} \arrow[r] \arrow[d]    & {H^{-1}(E_j)[1]} \arrow[d] \\
A_j \arrow[d] \arrow[r] \arrow[rd, "0" description] \arrow[ru, dotted] & E \arrow[d] \arrow[r]                 & E_j \arrow[d]  \\
0 \arrow[r]                                                            & H^0(E) \arrow[r, Rightarrow, no head] & H^0(E_j)                  
\end{tikzcd}
\]  

This implies that we may then replace the triangle (\ref{tri:AEE}) by $A_j \longrightarrow H^{-1}(E)[1] \longrightarrow H^{-1}(E_j)[1]$ 
to assume furthermore that $E$ and $E_j$ are shifts of sheaves, say $E=\mathcal{E}[1]$ and $E_j=\mathcal{E}_j[1]$ for some sheaves $\mathcal{E},\mathcal{E}_j$.

Consider the following distinguished triangle:
\begin{equation}\label{tri:EEB}
 \mathcal{E}_{j-1} \longrightarrow \mathcal{E}_j \longrightarrow B_j 
\end{equation}

As $B_j$ is in ${^{-1}\Per}(\widetilde{X}/X)$,
we have $\mathbf{R}^1\pi_\ast B_j=0$,
and therefore $\mathbf{R}^1\pi_\ast \mathcal{E}_j \twoheadrightarrow \mathbf{R}^1\pi_\ast \mathcal{E}_{j+1}$.
As $\pi$ is an isomorphism on $\widetilde{X} \setminus \Pi$, we know that if $F$ is a coherent sheaf, 
then $\mathbf{R}^1\pi_\ast F$ is supported on the singular point $x_0$.
We may thus assume that $\mathbf{R}^1\pi_\ast \mathcal{E}_j$ is constant for $j$ sufficiently large.
Applying $\mathbf{R}\pi_\ast$ on (\ref{tri:EEB}) and taking the long exact sequence  with respect to $\Coh(\widetilde{X})$ gives 
\begin{equation}\label{ses:REEB}
0 \longrightarrow \mathbf{R}^0\pi_\ast \mathcal{E}_{j-1} \longrightarrow \mathbf{R}^0\pi_\ast\mathcal{E}_j \longrightarrow \mathbf{R}^0\pi_\ast B_j  \longrightarrow 0.
\end{equation}

As $H^0(B_j)$ lies in $\mathcal{T}^{>0}_{\pi^\ast H}$ and satisfies $\Im Z_{\pi^\ast H,\beta,z}(B_j)=0$, it must be supported on the union of points and $\Pi$.
This implies that for each $j$, the sheaf $\mathbf{R}^0\pi_\ast B_j = \mathbf{R}^0\pi_\ast H^0(B_j)$ is a $0$-dimensional torsion sheaf on $X$.

Moreover, for each $j$, the sheaf $\mathbf{R}^0\pi_\ast\mathcal{E}_j$ is torsion-free.
Indeed, assume the contrary. As $\mathcal{E}_j|_{\widetilde{X}\setminus \Pi}$ is torsion-free,
there must exist a nonzero morphism $ \mathcal{O}_{x_0} \longrightarrow \mathbf{R}^0\pi_\ast\mathcal{E}_j$.
The underived adjointness gives a nonzero morphism $\mathcal{O}_\Pi \longrightarrow \mathcal{E}_j$ which is a contradiction, 
since $\mathcal{O}_\Pi \in \mathcal{T}_{\mathcal{B}^0}$ and $\mathcal{E}_j \in \mathcal{F}_{\mathcal{B}^0}$.
This means that for each $j$, $\mathbf{R}^0\pi_\ast\mathcal{E}_j$ is a subsheaf of $(\mathbf{R}^0\pi_\ast\mathcal{E}_j)^{\vee\vee}$ and the quotient $(\mathbf{R}^0\pi_\ast\mathcal{E}_j)^{\vee\vee})/\mathbf{R}^0\pi_\ast\mathcal{E}_j$ is supported on points.

We consider the dual triangle
of (\ref{ses:REEB}).
By Lemma \ref{dual},
$H^0(\mathbb{D} (\mathbf{R}^0\pi_\ast B_j))$ and $ H^1(\mathbb{D} (\mathbf{R}^0\pi_\ast B_j))$ are both $0$, 
then $(\mathbf{R}^0\pi_\ast \mathcal{E}_{j})^{\vee \vee}$ is constant.
We can thus use the double dual argument (cf. \cite[Proposition 7.1]{Bri08}) as following to say that
$\mathbf{R}^0\pi_\ast B_j = 0$  for $j$ sufficiently large.
Consider the commutative diagram:
\[\begin{tikzcd}
\mathbf{R}^0\pi_\ast \mathcal{E}_{j-1} \arrow[r] \arrow[d, hook]                                 & \mathbf{R}^0\pi_\ast \mathcal{E}_{j} \arrow[r] \arrow[d, hook]                    & \mathbf{R}^0\pi_\ast B_j \arrow[d] \\
(\mathbf{R}^0\pi_\ast \mathcal{E}_{j-1})^{\vee \vee} \arrow[r, Rightarrow, no head] \arrow[d, two heads]      & (\mathbf{R}^0\pi_\ast \mathcal{E}_{j})^{\vee \vee} \arrow[r] \arrow[d, two heads] & 0 \arrow[d]      \\
(\mathbf{R}^0\pi_\ast \mathcal{E}_{j-1})^{\vee \vee}/\mathbf{R}^0\pi_\ast \mathcal{E}_{j-1} \arrow[r, "\phi", two heads] & (\mathbf{R}^0\pi_\ast \mathcal{E}_{j})^{\vee \vee}/\mathbf{R}^0\pi_\ast \mathcal{E}_{j} \arrow[r]     & 0               
\end{tikzcd}\]

The snake lemma gives us a short exact sequence
$$ 0 \longrightarrow \mathbf{R}^0\pi_\ast B_j  \longrightarrow (\mathbf{R}^0\pi_\ast \mathcal{E}_{j-1})^{\vee\vee}/\mathbf{R}^0\pi_\ast \mathcal{E}_{j-1} \longrightarrow (\mathbf{R}^0\pi_\ast\mathcal{E}_j)^{\vee\vee}/\mathbf{R}^0\pi_\ast\mathcal{E}_j \longrightarrow 0.$$
Therefore, by looking at the length of $(\mathbf{R}^0\pi_\ast\mathcal{E}_j)^{\vee\vee}/\mathbf{R}^0\pi_\ast\mathcal{E}_j$,
we see that $\mathbf{R}^0\pi_\ast B_j = 0$ for $j$ sufficiently large.

Now, we have $\mathbf{R}^0\pi_\ast B_j = 0$ for $j$ sufficiently large.
Since each $B_j$ is in ${^{-1}}\Per(\widetilde{X}/X)$, we have $\mathbf{R}\pi_\ast B_j=\mathbf{R}^0\pi_\ast B_j = 0$ for $j$ sufficiently large.
This implies by Lemma \ref{H0} that, for $j$ sufficiently large, we have $H^0(B_j)=0$.
\end{proof}

\begin{corollary}\label{Noeth}
The pair $(Z_{\pi^\ast H,\beta,z},\mathcal{B}^0)$ satisfies the Harder--Narasimhan property.
Moreover, the heart $\mathcal{B}^0$ is Noetherian.
\end{corollary}
\begin{proof}
By Lemma \ref{HNprop} and as $\Img (\Im Z_{\pi^\ast H,\beta,z})$ is discrete, 
the assumptions of \cite[Proposition B.2]{BM11} are satisfied.
For the second assertion,
we consider a sequence of surjections 
\[
E \twoheadrightarrow E_1 \twoheadrightarrow E_2 \twoheadrightarrow \cdots
\]

Since $\Img (\Im Z_{\pi^\ast H,\beta,z})$ is discrete, $\Im Z_{\pi^\ast H,\beta,z}(E_n)=\Im Z_{\pi^\ast H,\beta,z}(E_{n+1})= \cdots$ for sufficiently large $n$.
Considering the kernels corresponding to this sequence of surjections, 
we obtain a sequence of inclusions in $\mathcal{B}^0$:
\[ 0 \hookrightarrow \ker(E_n\twoheadrightarrow E_{n+1}) \hookrightarrow \ker(E_n\twoheadrightarrow E_{n+2}) \hookrightarrow \cdots \hookrightarrow E_n\]
which must terminates by Lemma \ref{HNprop}.
\end{proof}

In summary, we have proved the following theorem.

\begin{theorem}\label{prestab}
The pair $(Z_{\pi^\ast H,\beta,z},\mathcal{B}^0)$ forms a pre-stability condition when $z$ and $\beta$ is chosen so that $\beta \cdot C_i>0$ for all $i$ , $\beta \cdot \ch_1(\mathcal{O}_{\Pi}) <1$, and $z> -\frac{\beta^2}{2}$.
\end{theorem}

\section{Support property on the resolution}
The aim of this section is to find a quadratic form $Q$ such that $(Z_{\pi^\ast H,\beta, z},\mathcal{B}^0)$ satisfies the support property with respect to $Q$ (see Definition \ref{stability}(c)).
We follow the basic approach of \cite[Sect. 6]{TX22},
but we address some gaps and generalize it to the case of an ADE configuration.

Let $\widetilde{\Lambda} := {\rm K}_{\num}(\widetilde{X})$.
We choose $z$ and $\beta$ so that $\beta \cdot C_i>0$ for all $i$, 
$\beta \cdot \ch_1(\mathcal{O}_{\Pi}) <1$, and $z> \max\{-\frac{\beta^2}{2},0\}$. We start with the following:
\begin{definition}
\begin{enumerate}
\item Given $A,B \geq 0$, we define  the quadratic form $Q_{A,B}$ on $\widetilde{\Lambda}$ by
\begin{equation}
\label{eqn:QAB}
\gls{quadratic-form}
\end{equation} 
where $\Delta:=  \ch_1^2-2\ch_2\ch_0$ is the discriminant.
\item For every $s \geq 1$, we consider the following function on ${\rm D}^b(\widetilde{X})$:
$$Z_{\pi^\ast H,\beta,sz} = -\ch_2+\beta \cdot \ch_1+sz\ch_0+i(\pi^\ast H)\cdot \ch_1.$$
\end{enumerate}
\end{definition}
By Theorem \ref{prestab}, the pair $(Z_{\pi^\ast H,\beta,sz},\mathcal{B}^0)$ is a pre-stability condition.
In this section, we will show that for appropriate choices of $A$ and $B$, 
$(Z_{\pi^\ast H,\beta,z},\mathcal{B}^0)$ satiesfies the support property with respect to $Q_{A,B}.$

We will prove the support property for $(Z_{\pi^\ast H,\beta,z},\mathcal{B}^0)$ by proving it for $(Z_{\pi^\ast H,\beta,sz},\mathcal{B}^0)$ for all $s \geq 1$.
This helps as $Z_{\pi^\ast H,\beta,sz}$-semistable objects for $s \gg 1$ are easier to describe.

We may first verify the first assertion of Definition \ref{stability} (c).

\begin{lemma}\label{negdef}
The quadratic form $Q_{A,B}$ is negative definite on $\ker Z_{\pi^\ast H,\beta,sz} \subseteq \widetilde{\Lambda} \otimes \mathbb{R}$, for every $A,B \geq 0$ and $s \geq 1$.
\end{lemma}
\begin{proof}
Given a class $x=(r,c,d)$ in $\widetilde{\Lambda} \otimes \mathbb{R}$,
assume that $Z_{\pi^\ast H,\beta,sz}(x)=0$ for some $s\geq 1$.
Note first that the quadratic form $\Delta(x) = \ch_1(x)^2-2\ch_2(x)\ch_0(x)=c^2-2rd$ can also be re-written using the twisted Chern characters $\ch^\beta(x)=\ch(x) \cdot e^{-\beta}$ as $\ch^\beta_1(x)^2-2\ch^\beta_2(x)\ch^\beta_0(x)$.

As we chose $\beta$ such that $(\pi^\ast H)\cdot \beta = 0$, we see from $\Im Z_{\pi^\ast H,\beta,sz}(x)=0$ that 
\[\ch^\beta_1(x) = \ch_1(x) - \beta \ch_0(x)=c - r \beta\] is contained in $(\pi^\ast H)^\perp$, 
so we may apply Hodge index theorem to see that $\ch^\beta_1(x)^2 \leq 0$.
On the other hand, as $0 = \Re Z_{\pi^\ast H,\beta,sz}(x)= -d+\beta \cdot c+szr$, we have $\ch^\beta_2(x)=szr$.

Since we chose $z>\max\{-\frac{\beta^2}{2},0\}$,
it can be seen that 
$$Q_{A,B}(x)= \Delta(x)=\ch^\beta_1(x)^2-2\ch^\beta_2(x)\ch^\beta_0(x) \leq -2(sz+\frac{\beta^2}{2})r^2 \leq 0.$$   
\end{proof}

\begin{definition}
Define the set $\mathcal{D}:=\{E \in \mathcal{B}^0 \, | \, E $ is $Z_{\pi^\ast H,\beta,sz}$-semistable for $s$ sufficiently large$\}$.  
\end{definition}

We first prove a crucial lemma in several steps.
\begin{lemma}\label{mathcalD}
There is a constant $A_0\geq 0$ such that $Q_{A_0,0}(E) \geq 0$ for every object $E \in \mathcal{D}$ with $\Im Z_{\pi^\ast H,\beta,z}(E)>0$.
\end{lemma}

To prove this, we first classify objects in $\mathcal{D}$,
particularly those with the imaginary parts of their central charges being positive.

\begin{lemma}\label{6.5'}\textnormal{(cf. \cite[Lemma 6.4]{TX22})}
Given $E$ in the set $\mathcal{D}$, then $E$ must be one of the following forms:
\begin{enumerate}
    \item[(1)] $E=H^0(E)$ is a torsion sheaf;
    \item[(2)] $E=H^0(E)$ fits in a short exact sequence $$ 0 \longrightarrow E_{tor} \longrightarrow E \longrightarrow E_{tf} \longrightarrow 0$$
    where $E_{tf}$ is a torsion-free slope semistable sheaf in $\mathcal{T}^{> 0}_{\pi^\ast H}$, and $E_{tor}$ is a (possibly zero) torsion sheaf with $\mathbf{R}^0 \pi_\ast E_{tor}= 0$.  
    \item[(3)] $H^0(E)$ is either $0$ or a sheaf supported on the union of points and $\Pi$, and $H^{-1}(E)$ fits into a short exact sequence
    $$ 0 \longrightarrow G \longrightarrow H^{-1}(E) \longrightarrow F \longrightarrow 0$$
    where $F$ is a torsion-free slope semistable sheaf in $\mathcal{F}^{\leq 0}_{\pi^\ast H}$, and $\mathbf{R}^0 \pi_\ast G = 0$.
    Moreover, $G$ must be $0$ unless $(\pi^\ast H)\cdot \ch_1(F) =0$.
\end{enumerate}
\end{lemma}

\begin{proof}
Given $E \in \mathcal{D}$, we first decompose $E$ as $G[1] \longrightarrow E \longrightarrow T$ with $G \in \mathcal{F}_0$,
and with $T \in \mathcal{T}_0$ fitting into an exact triangle $F[1] \longrightarrow T \longrightarrow S$, 
where $F \in \mathcal{F}^{\leq 0}_{\pi^\ast H}$ and $S \in \mathcal{T}^{> 0}_{\pi^\ast H}$.

\[\begin{tikzcd}
                 &             & {F[1]} \arrow[d] \\
{G[1]} \arrow[r] & E \arrow[r] & T \arrow[d]      \\
                 &             & S               
\end{tikzcd}\]

First consider the case $\ch_0(E)>0$.
In this case, as $s$ tends to $\infty$, the phase $\phi_{\pi^\ast H,\beta,sz}(E)$ tends to $0$.
However, as $G$ is supported on the exceptional curves, $\phi_{\pi^\ast H,\beta,sz}(G[1])=1$ for every $s$.
As $E$ is $Z_{\pi^\ast H,\beta,sz}$-semistable for large $s$, this implies that $G=0$ and $E = T$.

We know that $\ch_0(F[1])<0$ and hence $\phi_{\pi^\ast H,\beta,sz}(F[1]) \rightarrow 1$ as $s$ grows large.
Since $\Hom(F[1],R)=\Ext^{-1}(F,R)=0$ for any sheaf $R \in \mathcal{F}_0$,
$F$ lies in $\mathcal{T}_0$.
Moreover, $F[1]$ is a subobject of $T=E$ in $\mathcal{B}^0$,
so $F[1]=0$ similarly and $E=S$ is a sheaf in $\mathcal{T}^{> 0}_{\pi^\ast H}$ with positive rank.

If $S$ is slope semistable, then it is  of the form (2) with $E_{tor}=0$.
We then assume the contrary and consider the Harder--Narasimhan filtration of $S$ with respect to $\mu_{\pi^\ast H}$, that is, a sequence of inclusions:
$$ 0 = S_0 \subseteq S_1 \subseteq \cdots  \subseteq S_{m-1} \subseteq S_m=S $$
with the quotients factors $S_i/S_{i-1}$ being $\mu_{\pi^\ast H}$-semistable and the slopes are decreasing.

$S_{m-1}$ is in $\Coh^0_{\pi^\ast H}(\widetilde{X})$ so it fits into a distinguished triangle $S'_{m-1} \longrightarrow S_{m-1} \longrightarrow F'$ with $S'_{m-1} \in \mathcal{T}_0\subseteq \mathcal{B}^0$ and $F' \in \mathcal{F}_0$.

If $S'_{m-1}\neq 0$, this particularly means that $\mathbf{R}^0\pi_\ast F'=0$ and the slope of $S'_{m-1}$ is the same as that of $S_{m-1}$.
The long exact cohomology sequence tells us that $S'_{m-1}$ is also a sheaf.

Now, consider the composition $S'_{m-1} \hookrightarrow S$ and let $Q:=S/S'_{m-1}$.
The quotient $Q$ is a sheaf in $\mathcal{T}^{>0}_{\pi^\ast H}\subseteq \Coh^0_{\pi^\ast H}(\widetilde{X)}$ since $S \in \mathcal{T}^{> 0}_{\pi^\ast H}$.

Moreover, $\Hom(Q,\mathcal{F}_0) = \Hom(S,\mathcal{F}_0)=0$.
These together imply that $Q \in \mathcal{T}_0 \subseteq \mathcal{B}^0$,
and hence $S'_{m-1}$ is a subobject of $S$ in $\mathcal{B}^0$.
But when $s$ is sufficiently large, 
we have $\phi_{\pi^\ast H,\beta,sz}(S'_{m-1})>\phi_{\pi^\ast H,\beta,sz}(S)$ since $\mu_{\pi^\ast H}(S'_{m-1})>\mu_{\pi^\ast H}(S)$, 
which leads to a contradiction as $S=E$ is $Z_{\pi^\ast H,\beta,sz}$-semistable.
Therefore, $S'_{m-1}$ must be zero, and $E$ is of the form (2).

Next, consider the case $\ch_0(E)<0$.
Then $\phi_{\pi^\ast H,\beta,sz}(E) \rightarrow 1 $ as $s \rightarrow \infty$.
Since $S$ is a quotient of $E$ in $\mathcal{B}^0$,
we must also have $\phi_{\pi^\ast H,\beta,sz}(S) \rightarrow 1 $ as $s \rightarrow \infty$.
This is possible only if $\ch_0
(S) = 0$ and $ (\pi^\ast H) \cdot \ch_1
(S) = 0$, which means that $S$ is a torsion sheaf supported on the union of finitely many points and the exceptional locus $\Pi$.

If $(\pi^\ast H) \cdot \ch_1(F) < 0$ then  $\Im Z_{\pi^\ast H,\beta,sz}(E)>0$, and so  $G = 0 $ by semistability of $E$ as $s$ large.
In this case, we can use the Harder--Narasimhan filtrations as above to show that $F$ is slope semistable.

If $(\pi^\ast H) \cdot  \ch_1(F) = 0$, then as every subobject of $F$ is contained in $\mathcal{F}^{\leq 0}_{\pi^\ast H}$,
there can not exist a subobject with slope greater than $0 = \mu_{\pi^\ast H}(F)$,
that is, $F$ is slope semistable.

Finally, assume that $\ch_0(E)=0$. 
If $T=0$, then $E=G[1]$.
We then assume $T\neq 0$.
Now we have either $(\pi^\ast H) \cdot \ch_1(E) >0$ or $(\pi^\ast H) \cdot \ch_1(E) =0$.

In the former case, the phase $\phi_{\pi^\ast H,\beta,sz}(E)$ is between $0$ and $1$ for every $s$, and hence $G=0$ and $F=0$ as before.
Then $E$ is a torsion sheaf supported on a curve $C' \nsubseteq \Pi$, and therefore $E=H^0(E)$ is a slope semistable sheaf with slope $\mu_{\pi^\ast H}(E)=\infty>0$.

In the latter case, 
we actually have $(\pi^\ast H) \cdot \ch_1(S) =(\pi^\ast H) \cdot \ch_1(F) = 0$.
Then $S=H^0(E)$ can only be a torsion sheaf supported on the union of points and $\Pi$.
Note that we have $F=0$.
Otherwise, $F$ must be a torsion-free sheaf,
but then $0=\ch_0(E)=\ch_0(T)=-\ch_0(F)<0$, a contradiction.
\end{proof}

\begin{lemma}\label{constIZ}
There is a positive constant $A_0$, depending only on $H$, such that for any sheaf $E$ supported on a curve $C'$ not containing any of exceptional curves $C_i$, we have $$\ch_1(E)^2 + A_0 [(\pi^\ast H)\cdot\ch_1(E)]^2 \geq 0.$$
\end{lemma}
\begin{proof}
We first claim that there exist some $b_i>0$ such that $(\displaystyle \sum_i b_iC_i) \cdot C_j <0$ for all $j$.

The exceptional curves $C_i$ span a negative definite convex cone $\sigma \subseteq \NS(\widetilde{X}) \otimes \mathbb{R}$ (again, see \cite[Theorem A.7]{Reid}) 
and it suffices to prove that $\sigma \cap -\sigma^\vee$ is non-empty.
Assume the contrary. 
There must be a linear function which is positive on $\sigma$ and negative on $-\sigma^\vee$, 
which can be represented by pairing with a vector $v$.

But then as $v$ is postive on $\sigma^\vee$, we know that $v$ lies in $(\sigma^\vee)^\vee = \sigma$ , which implies that $(v,v)\geq 0$.   
This is a contradiction to the negative definiteness and hence our claim is proved.

By \cite[Exercise 4.16]{Reid}, for sufficiently small $\epsilon>0$, the class $H':=\pi^\ast H - \epsilon(\displaystyle \sum_i b_iC_i)$ is ample.
Now let $C'$ be the support of $E$ with multiplicities and write $C'= \alpha + a_1C_1+ \cdots a_nC_n$ with $\alpha \cdot C_j=0$ for all $j$. 
As  $C' \cdot H'>0$, we obtain that $$(\pi^\ast H) \cdot \alpha > \epsilon (\displaystyle \sum_i b_iC_i) \cdot (\displaystyle \sum_i a_iC_i) = \sum_i \epsilon b_i (C_i \cdot C') \geq 0.$$

By the openness of the ample cone on $X$, there is a constant $N>0$ depending only on $H$ such that $(\pi^\ast H)^2 \alpha^2 + N(\pi^\ast H \cdot \alpha)^2 \geq 0$ 
for all effective classes $\alpha$ on $\widetilde{X}$.
It suffices to show that $(\pi^\ast H \cdot \alpha)^2 \geq M(\displaystyle \sum_i a_iC_i)^2$ for some constant $M<0$ depending only on $H$.
We can then choose $A_0:= \frac{N}{(\pi^\ast H)^2} - \frac{1}{M} >0$.

As $C_1,\cdots, C_n$ are exceptional curves of an ADE singularity, 
by \cite[Proposition 10.18(iii)]{Car05},
every entry of the inverse of the intersection matrix is negative.
This implies that if $(\displaystyle \sum_i a_iC_i)\cdot C_j \geq 0$ for all $j=1, \cdots, n$, then all $a_i \leq 0$.
Consider the unit sphere $S$ (with respect to the metric given by $D^2$)
and the convex cone $\sigma$ generated by all the curves $C_i$ in $\NS_{\mathbb{R}}(\widetilde{X})$.
Then $D \mapsto [D \cdot (\displaystyle \sum_i b_iC_i)]^2$ is a continuous positive function for $D$ in the compact set $\sigma \cap S$, and hence there is a minimum $K$.
Moreover, this number $K$ must be positive as any $D$ in $\sigma$ satisfies $D \cdot (\displaystyle \sum_i b_iC_i) > 0$.

We set $D:=(-\displaystyle \sum_i a_iC_i)$ and then $[D \cdot (\displaystyle \sum_i b_iC_i)]^2 \geq -KD^2$.
Now $$(\pi^\ast H \cdot \alpha)^2 >  [\epsilon (\displaystyle \sum_i b_iC_i) \cdot (\displaystyle \sum_i a_iC_i)]^2 \geq -\epsilon^2K[(\displaystyle \sum_i a_iC_i)^2].$$
Choose $M:=-\epsilon^2K$ and the assertion follows.
\end{proof}

\begin{lemma}\label{ch<0}
Let $E \in \mathcal{D}$ be an object with $\ch_0(E)<0$ and $\Im Z_{\pi^\ast H,\beta,z}(E)>0$.
Then $\Delta(E) \geq 0$.
\end{lemma}

\begin{proof}
By Lemma \ref{6.5'}, such $E$ must fit into an exact triangle $F[1] \longrightarrow E \longrightarrow S$, with $S$ a torsion sheaf support on the union of points and $\Pi$, and with $F$ a slope semistable sheaf with $\mu_{\pi^\ast H}(F) \geq 0$.

Firstly, we can decompose $S$ into $S_1 \oplus S_2$ where $\supp S_1 \subseteq \Pi$ and $\supp S_2$ is $0$-dimensional,
then by Hirzebruch-Riemann-Roch, we obtain
\begin{equation*} 
\begin{split}
\Delta(E) &= \Delta(F)+\ch_1(S_1\oplus S_2)^2 -2\ch_0(F[1])\ch_2(S_1\oplus S_2) +2\ch_1(F[1])\ch_1(S_1\oplus S_2)\\
&= \Delta(F)+\ch_1(S_1)^2 -2\ch_0(F[1])\ch_2(S_1)+2\ch_1(F[1])\ch_1(S_1)- 2\ch_0(F[1])\ch_2(S_2)\\
&= \Delta(F) - \chi(S_1,S_1) -2\chi(S_1,F[1]) +2 \ch_0(F)\ch_2(S_2).
\end{split}
\end{equation*}

As $\ch_0(F)\ch_2(S_2)$ is non-negative,
we may assume that $S = S_1$ is supported on the exceptional locus $\Pi$.

Since the canonical bundle $\omega_{\widetilde{X}}$ is trivial on the exceptional locus,
by Serre duality we have
$\dim\Hom_{\widetilde{X}}(S,S)=\dim\Ext^2_{\widetilde{X}}(S,S)$.
Therefore, we have $\chi_{\widetilde{X}}(S,S)\leq 2 \dim\Hom_{\widetilde{X}}(S,S)$.

If we apply the functor $\Hom_{\widetilde{X}}(S,\cdot)$ to the exact triangle $F[1] \longrightarrow E \longrightarrow S$,
we obtain an exact sequence
$$ \Hom_{\widetilde{X}}(S,E) \longrightarrow \Hom_{\widetilde{X}}(S,S) \longrightarrow \Ext^1_{\widetilde{X}}(S,F[1])$$

By semistability we see that $\Hom_{\widetilde{X}}(S,E)=0$ and similarly $\Hom_{\widetilde{X}}(S,F[1+i])=0$ for any $i\leq0$ as the phase of $S$ is $1$ and those of $E$ and $F[1+i]$ are less than $1$.
Since $\Hom_{\widetilde{X}}(S,F[1+i])=0$ for all $i\geq 2$,
this implies that 
\[\dim\Hom_{\widetilde{X}}(S,S) \leq \dim \Ext^1_{\widetilde{X}}(S,F[1])=-\chi(S,F[1]).\]

In summary, we have the inequality $\chi(S,S) \leq 2 \dim\Hom_{\widetilde{X}}(S,S) \leq -2\chi(S,F[1])$,
that is, $- \chi(S,S) -2\chi(S,F[1]) \geq 0$.
As $\Delta(F) \geq 0$ by Bogomolov-Gieseker, this completes the proof.
\end{proof}

\begin{proof}[Proof of Lemma \ref{mathcalD}]
Let $E\in \mathcal{D}$ be an  object with $\Im Z_{\pi^\ast H,\beta,z}(E)>0$.

If $\ch_0(E) >0$, then by Lemma \ref{6.5'}, $E$ fits into an exact sequence
\begin{equation}\label{SES in 5.5}
0 \longrightarrow E_{tor} \longrightarrow E\longrightarrow E_{tf} \longrightarrow 0,
\end{equation}
where $E_{tf}$ is a slope semistable torsion-free sheaf and $E_{tor} \in \ker \mathbf{R}^0\pi_\ast$.

If $E_{tor} = 0$, we apply the Bogomolov-Gieseker inequality to see that $Q_{A,0}(E) \geq \Delta(E) \geq 0$ for any $A \geq 0$. 
Otherwise, we claim that $$\Delta(E) 
= \Delta(E_{tf}) -\chi(E_{tor},E_{tor})-2\chi(E_{tf},E_{tor}) \geq 0.$$

By Bogomolov-Gieseker inequality, we have $\Delta(E_{tf})\geq 0$,
and by Serre duality we have $\chi(E_{tor},E_{tor}) \leq 2\dim \Hom_{\widetilde{X}}(E_{tor},E_{tor})$.
Thus, it suffices to show that
\[-\chi(E_{tor},E_{tf}) \geq \dim \Hom_{\widetilde{X}}(E_{tor},E_{tor}).\]

Applying $\Hom(\cdot ,E_{tor})$ to the short exact sequence (\ref{SES in 5.5}), we obtain the following long exact sequence:

\begin{tikzcd}
0 \arrow[r] & {\Hom(E_{tf},E_{tor})} \arrow[r]   & {\Hom(E,E_{tor})} \arrow[r]   & {\Hom(E_{tor},E_{tor})} \arrow[lld] &                          \\
            & {\Ext^1(E_{tf},E_{tor})} \arrow[r] & {\Ext^1(E,E_{tor})} \arrow[r] & {\Ext^1(E_{tor},E_{tor})} \arrow[r] & {\Ext^2(E_{tf},E_{tor})}
\end{tikzcd}

Note first that by Serre duality, $\Ext^2(E_{tf},E_{tor})=\Hom(E_{tor},E_{tf})=0$.
Moreover, as $\mathbf{R}^0\pi_\ast E_{tor}=0$, we know that $E_{tor}[1] \in \mathcal{P}(1) \subseteq \mathcal{B}^0$, where $\mathcal{P}$ is the slicing corresponding to the heart $\mathcal{B}^0$.

Then, as $E \in \mathcal{P}(0,1]$ and $E_{tor} \in \mathcal{P}(0)$, we see that $\Hom(E,E_{tor})=0$, 
which also implies that $\Hom(E_{tf},E_{tor})=0$.
Therefore, we see that
$$-\chi(E_{tor},E_{tf})= \dim \Ext^1(E_{tor},E_{tf}) \geq \dim \Hom_{\widetilde{X}}(E_{tor},E_{tor}),$$
and thus $\Delta(E) \geq \Delta(E_{tf})\geq0.$

If $\ch_0(E)=0$, by Lemma \ref{6.5'}, $E$ must be a torsion sheaf supported on a curve $C \nsubseteq \Pi$.
If $C$ does not contain $C_i$ for any $i$,
then Lemma \ref{constIZ} implies that $Q_{A_0,B}(E) \geq 0$ for every such torsion sheaf $E$ and any positive $B$.

Otherwise, we write $C=C'+C''$, with $C'$ a combination of exceptional curves and $C''$ a curve not containing any of $C_i$.
We then have a short exact sequence
\begin{equation}\label{SES2 in 5.5}
0 \longrightarrow K \longrightarrow E\longrightarrow E|_{C''} \longrightarrow 0,
\end{equation}
where $K \subseteq E|_{C'}$ is a torsion sheaf with $\Im Z_{\pi^\ast H,\beta,z}(K)=0$.

Looking at the long exact cohomology sequence of (\ref{SES2 in 5.5}) with respect to the heart $\mathcal{B}^0$,
we see that $E|_{C''} \in \mathcal{B}^0$.
Moreover, as we assume $E$ to be $Z_{\pi^\ast H,\beta,sz}$-semistable for $s$ sufficiently large, 
we see that $H^0_{\mathcal{B}^0}(K)$ must be $0$.
In particular, this means that $K \in \mathcal{P}(0)$ since $\Im Z_{\pi^\ast H,\beta,z}(K)=0$.

Then, the same arguments as in the previous case show that
$$-2\chi(K,E|_{C''})= 2\dim \Ext^1(K,E|_{C''}) \geq 2\dim \Hom_{\widetilde{X}}(K,K)\geq \chi(K,K).$$

In summary, we now have:
\begin{equation*} 
\begin{split}
Q_{A_0,B}(E)  &\geq \Delta(E)+A_0(\Im Z_{\pi^\ast H,\beta,z}(E))^2 \\
&= \Delta(E) +A_0(\Im Z_{\pi^\ast H,\beta,z}(E|_{C''}))^2 \text{ \ (since $\Im Z_{\pi^\ast H,\beta,z}(K)=0$)}\\
&=\Delta(E|_{C''})-\chi(K,K)-2\chi(K,E|_{C''}) +A_0(\Im Z_{\pi^\ast H,\beta,z}(E|_{C''}))^2\\
&\geq \Delta(E|_{C''})+A_0(\Im Z_{\pi^\ast H,\beta,z}(E|_{C''}))^2 \\
&\geq 0. \text{ 
\ (by Lemma \ref{constIZ})}
\end{split}
\end{equation*}

Finally, if $\ch_0(E)<0$, by Lemma \ref{ch<0}, $Q_{A,0}(E) \geq \Delta(E) \geq 0$ for any $A \geq 0$.
\end{proof}

\begin{lemma}\label{constRZ}
There exists a constant $B_0>0$ such that $Q_{0,B_0}(E) \geq 0$ for any $Z_{\pi^\ast H,\beta,z}$-stable object $E \in \mathcal{B}^0$ with $\Im Z_{\pi^\ast H,\beta,z}(E)=0$.
\end{lemma}
\begin{proof}
We consider the same decomposition as in Lemma \ref{6.5'}:

\[\begin{tikzcd}
                 &             & {F[1]} \arrow[d] \\
{G[1]} \arrow[r] & E \arrow[r] & T \arrow[d]      \\
                 &             & S               
\end{tikzcd}\]
where $G \in \mathcal{F}_0$, $F \in \mathcal{F}^{\leq 0}_{\pi^\ast H}$ and $S \in \mathcal{T}^{> 0}_{\pi^\ast H}$.
As $E$ is stable, it must be equal to $G[1]$, to $F[1],$ or to $S$.

If $E=F[1]$, then as $F$ is a torsion-free slope-stable sheaf, 
$Q_{A,B}(E) \geq \Delta(E)=\Delta(F) \geq 0$ for any $A,B>0$ by the Bogomolov-Gieseker inequality.
If $E=G[1]$, then as $E$ is stable, $E$ can only be $\mathcal{O}_{C_i}(-1)[1]$ for some $i$ by Proposition \ref{simpleobjects}.
As there are only finitely many choices, we can find some $B_1>0$ such that $Q_{A,B}(\mathcal{O}_{C_i}(-1)[1]) \geq 0$  for all $i$ and $B \geq B_1$.

Finally assume that $E=S$.
If $S$ is supported on a point then $\Delta(E)=\Delta(S) = 0$, so we may assume that $S$ is supported on $\Pi$,
and then by Proposition \ref{simpleobjects}, $S$ can only be $\mathcal{O}_{\Pi}$.
We can similarly find a desired constant $B_2>0$ such that $Q_{A,B_2}(\mathcal{O}_{\Pi}) \geq 0$ and set $B_0:=\max \{B_1,B_2\}$.
\end{proof}

Combining all above, we can finally prove our main result in this section.

\begin{theorem}\label{mainSect5}
The pair $\sigma = (Z_{\pi^\ast H,\beta,z}, \mathcal{B}^0)$ is a Bridgeland stability condition, 
with respect to the lattice ${\rm K}_{\num}(\widetilde{X})$,
and it satisfies the support property with respect to the lattice $\widetilde{\Lambda}$ with the quadratic form (\ref{eqn:QAB}) for $A,B$ sufficiently large.
\end{theorem}
\begin{proof}
Let $A_0,B_0$ be the constants given by Lemma \ref{constIZ} and \ref{constRZ}.
It remains to prove that $Q_{A_0,B_0}(E) \geq 0$ for any $Z_{\pi^\ast H,\beta,z}$-semistable object $E \in \mathcal{B}^0$. 

Note that as the proof of Lemma \ref{constIZ} and \ref{constRZ} are independent of $s$,
we are going to prove the following stronger statement that if $E \in \mathcal{B}^0$ is a $Z_{\pi^\ast H,\beta,sz}$-semistable object for some $s \geq 1$, then $Q_{A_0,B_0}(E) \geq 0$.

As the image of $\Im Z_{\pi^\ast H,\beta,z}(E)$ is discrete,
we prove by induction on $\Im Z_{\pi^\ast H,\beta,z}(E)$.
If $\Im Z_{\pi^\ast H,\beta,z}(E)=0$, 
then the assertion follows from Lemma \ref{constRZ}.

If $\Im Z_{\pi^\ast H,\beta,z}(E)>0$ is minimal,
then as any possible destabilizing subobjects must have smaller imaginary part,
$E$ must be in $\mathcal{D}$ and hence the assertion follows from Lemma \ref{mathcalD}.

Now let $E$ be an object in the set \[\{F \in \mathcal{B}^0 \, | \, F  \text{ is } Z_{\pi^\ast H,\beta,sz}\text{-semistable for some } s\geq1 \text{ and }Q_{A_0,B_0}(F)<0 \}\] with the smallest imaginary part.
Then the induction hypothesis says that for any $Z_{\pi^\ast H,\beta,sz}$-semistable object $F$ with $0 \leq \Im Z_{\pi^\ast H,\beta,z}(F)<\Im Z_{\pi^\ast H,\beta,sz}(E)$,
we have $Q_{A_0,B_0}(F) \geq 0$.

By Lemma \ref{mathcalD}, $E \notin \mathcal{D}$,
and hence by wall-crossing,
there is some $s>1$ such that $E$ is strictly $Z_{\pi^\ast H,\beta,sz}$-semistable
(cf. \cite[Theorem 3.5]{BMS16} and \cite[Sect. 9]{Bri08}).

We can then consider the Jordan--Hölder factors $E_1, \dots, E_m$ of $E$, which have smaller imaginary parts,
so that we have $Q_{A_0,B_0}(E_i)\geq 0$ for every $i$.

Now we may apply \cite[Lemma A.6]{BMS16}.
More precisely, since all of the Jordan--Hölder factors $E_i$ have the same phase, i.e. they lie on the same ray of the image of $Z_{\pi^\ast H,\beta,sz}$,
we can find constants
$a_{ij} > 0$ such that $[E_i] - a_{ij}[E_j]$ is in the kernel of $Z_{\pi^\ast H,\beta,sz}$ and hence $Q_{A_0,B_0}([E_i]-a_{ij}[E_j]) \leq 0$.

Then linear algebra tells us that $Q_{A_0,B_0}(m[E_i] + n[E_j] ) \geq 0$ for any positive $n, m$, and hence $Q_{A_0,B_0}(E) \geq 0$,
which is a contradiction.
Therefore, $Q_{A_0,B_0}(E) \geq 0$ for every $Z_{\pi^\ast H,\beta,sz}$-semistable object $E$,
which concludes the proof.
\end{proof}

\section{Deformation and compatibility of pre-stability conditions}

In this section, we will see that by deforming along a path $\sigma_\epsilon$ of stability conditions on $\widetilde{X}$,
the end point will induce a pre-stability condition $\sigma_X$ on the singular surface $X$, which is compatible with the pushforward $\mathbf{R}\pi_\ast$.

More precisely, we will give a path $Z_\epsilon$ in the space of central charges, parametrized by $\epsilon$,
which lifts to a path of Bridgeland stability conditions  \gls{sigma-epsilon} in $\Stab({\rm D}^b(\widetilde{X}))$,
but the end point, denoted by $\sigma_{\widetilde{X}} = (Z_{\widetilde{X}},\mathcal{B}^0)$, is only a weak pre-stability condition,
which will be compatible with a stability condition $\sigma_X$ on ${\rm D}^b(X)$ in the following sense.

The path is chosen so that at its end point $\sigma_{\widetilde{X}}$, the central charge of each $\mathcal{O}_{C_i}(-1)$ is $0$.
This means that there is a function $Z_X \colon {\rm K}_{\num}(X) \longrightarrow \mathbb{C}$ with $Z_{\widetilde{X}} = Z_X \circ \pi_\ast$,
where $\pi_\ast \colon {\rm K}_{\num}(\widetilde{X}) \longrightarrow {\rm K}_{\num}(X)$ is the pushforward on the numerical Grothendieck groups.

Let $v_\pi$ be the composition of $v \colon {\rm K}({\rm D}^b(\widetilde{X})) \longrightarrow \widetilde{\Lambda}$ and the projection map of the numerical Grothendieck group \gls{projection-map}, and let $\Lambda:= {\rm K}_{\num}(X)$. 
Our results in this section will give the following commutative diagram:

\[
\begin{tikzcd}
{\rm K}(\widetilde{X}) ={\rm K}(\mathcal{B}^0) \arrow[rr, "v"] \arrow[dd, "\pi_\ast"] \arrow[rrd, "v_\pi"] &  & \widetilde{\Lambda}={\rm K}_{\num}(\widetilde{X}) \arrow[d, "\pr"] \arrow[rrd, "Z_{\widetilde{X}}"]        &  &            \\
&  & \widetilde{\Lambda}/\ker \pi_\ast \arrow[d] \arrow[rr] &  & \mathbb{C} \\
{\rm K}(X)= {\rm K}(\mathcal{A}) \arrow[rr, "v'"]                                                                  &  &  \Lambda={\rm K}_{\num}(X) \arrow[rru, "Z_X"]                                                          &  &           
\end{tikzcd}\]

Our goal is to moreover obtain a result like Toda's that the pushforward gives an exact functor between the hearts $\mathcal{B}^0$ and $\mathcal{A}$.
The plan is to realize the hearts as tilts of perverse coherent sheaves and coherent sheaves respectively,
and then show that there are some relations between the torsion pairs defined them.

We will first construct $\mathcal{A}$ as the tilted heart $\Coh^\alpha_H(X)$ of the category of coherent sheaves.
Similarly as Lemma \ref{Bri02}, to prove the desired exactness, we need to find a real number $\alpha$ such that $\mathbf{R}\pi_\ast (\mathcal{B}^0) \subset \Coh^\alpha_H(X)$.

Our proof will rely on the following lemma,
the proof of which is obvious.
\begin{lemma}\label{1}
Let $\mathcal{D}_1, \mathcal{D}_2$ be triangulated categories. 
Given two abelian categories $\mathcal{A}_1 \subseteq \mathcal{D}_1$ and $\mathcal{A}_2 \subseteq \mathcal{D}_2$, assume that there are torsion pairs $\mathcal{A}_i = \langle \, \mathcal{T}_i,\mathcal{F}_i \, \rangle$ for $i=1,2$. 
If a triangulated functor $G \colon \mathcal{D}_1 \longrightarrow \mathcal{D}_2$ sends $\mathcal{A}_1,\mathcal{T}_1$, and $\mathcal{F}_1$ to $\mathcal{A}_2,\mathcal{T}_2$, and $\mathcal{F}_2$ respectively,
then $G(\mathcal{A}^\sharp_1) \subseteq \mathcal{A}^\sharp_2$.
\end{lemma}

We now make the following optimistic claim:
If $E \in {^{-1}\Per(\widetilde{X}/X)}$ is $\mu_{\pi^\ast H}$-slope semistable, then $\pi_\ast E$ will be a $\mu_{H}$-slope semistable sheaf of the same slope.
If this is the case, then Lemma \ref{1} applies and we can take $\alpha = 0$.
However, to prove this we can not use the adjunction argument as the derived pullback $\mathbf{L}\pi^\ast$ is not t-exact.
For example, $\mathbf{L}\pi^\ast \mathcal{O}_{x_0}$ is even a non-bounded complex. 
This also results in that we can not prove that the derived pullback $\mathbf{L}\pi^\ast$ is exact by Lemma \ref{1}.

Let \gls{derived-pullback} denote the $i$-th pullback with respect to a heart $\mathcal{A}$ of a bounded t-structure.

The following is the most crucial result in this section.

\begin{proposition}\label{preservess}
Let $(Z_\mathcal{A},\mathcal{A})$ and $(Z_\mathcal{B},\mathcal{B})$ be weak stability conditions on ${\rm D}^b(\widetilde{X})$ and ${\rm D}^b(X)$ respectively.
Assume that 
\begin{enumerate}
    \item[(1)] $\mathbf{R}\pi_\ast$ is t-exact with respect to the t-structures associated with $\mathcal{A}$ and $\mathcal{B}$,
    \item[(2)] $Z_\mathcal{A} =  Z_\mathcal{B} \circ \pi_\ast$.
\end{enumerate}
Then given a $Z_\mathcal{A}$-semistable object $E \in \mathcal{A}$,
the object $\mathbf{R}\pi_\ast E \in \mathcal{B}$ is $Z_\mathcal{B}$-semistable.
\end{proposition}
\begin{proof}
Let $\mu_{\mathcal{A}}$ and $\mu_{\mathcal{B}}$ be the slope with respect to the stability functions $Z_{\mathcal{A}}$ and $Z_{\mathcal{B}}$ respectively.
Assume the contrary that $\mathbf{R}\pi_\ast E$ is $Z_\mathcal{B}$-unstable.
Then either there exists a quotient $\mathbf{R}\pi_\ast E \twoheadrightarrow G$ with $\mu_{\mathcal{B}}(G)<\mu_{\mathcal{B}}(\mathbf{R}\pi_\ast E)$,
or there exists a nonzero subobject $K \subseteq \mathbf{R}\pi_\ast E$ with $Z_{\mathcal{B}}(K)=0$.
  
In the first case, as $\mathbf{R}\pi_\ast$ is t-exact with respect to the t-structures given by $\mathcal{A}$ and $\mathcal{B}$,
its left adjoint $\mathbf{L}\pi^\ast$ is right t-exact.
In particular, $\mathbf{L}^0_{\mathcal{A}}\pi^\ast \colon \mathcal{B} \longrightarrow \mathcal{A}$ is right exact.
Therefore, the natural map $f \colon \mathbf{L}^0_{\mathcal{A}}\pi^\ast \mathbf{R}\pi_\ast E \longrightarrow  \mathbf{L}^0_{\mathcal{A}}\pi^\ast G$ is a surjection in $\mathcal{A}$.
We claim that this is a $Z_\mathcal{A}$-destabilizing quotient of $\mathbf{L}^0_{\mathcal{A}}\pi^\ast \mathbf{R}\pi_\ast E$, that is $\mu_\mathcal{A}(\mathbf{L}^0_{\mathcal{A}}\pi^\ast G)<\mu_\mathcal{A}(\mathbf{L}^0_{\mathcal{A}}\pi^\ast \mathbf{R}\pi_\ast E )$.

To show this, we will prove that $\mathbf{L}^0_{\mathcal{A}}\pi^\ast $ preserves the slope, i.e. $\mu_{\mathcal{B}}(M)=\mu_{\mathcal{A}}(\mathbf{L}^0_{\mathcal{A}}\pi^\ast M)$ for any object $M \in \mathcal{B}$.
Indeed, by condition (1) and the equality $\mathbf{R}\pi_\ast \mathbf{L}\pi^\ast  = \id$,
we see that $\tau^{<0}_{\mathcal{A}}(\mathbf{L}\pi^\ast M)$ is in the kernel of $\mathbf{R}\pi_\ast$:
$$\mathbf{R}\pi_\ast(\tau^{<0}_{\mathcal{A}}(\mathbf{L}\pi^\ast M))= \tau^{<0}_{\mathcal{B}} (\mathbf{R}\pi_\ast \mathbf{L}\pi^\ast M) = \tau^{<0}_{\mathcal{B}} (M) = 0$$

Therefore $M \simeq \mathbf{R}\pi_\ast \mathbf{L}\pi^\ast M = \mathbf{R}\pi_\ast \mathbf{L}^0_{\mathcal{A}} \pi^\ast M$, 
so $\mu_{\mathcal{B}}(M) = \mu_{\mathcal{B}}(\mathbf{R}\pi_\ast \mathbf{L}^0_{\mathcal{A}} \pi^\ast M) = \mu_{\mathcal{A}}(\mathbf{L}^0_{\mathcal{A}}\pi^\ast M)$ since the condition (2) implies that $\mathbf{R}\pi_\ast$ also preserves the slope. 

Next, we consider the map $g \colon \mathbf{L}\pi^\ast \mathbf{R}\pi_\ast E \longrightarrow E$ given by adjunction.
Applying $H^0_{\mathcal{A}}$, we get a morphism $H^0_{\mathcal{A}}(g) \colon \mathbf{L}^0_{\mathcal{A}}\pi^\ast \mathbf{R}\pi_\ast E \longrightarrow E$.
By the definition, we know that $\mathbf{R} \pi_\ast (g) = \id_{\mathbf{R}\pi_\ast E}$,
and thus the pushforward of $\cone(g)$ is $0$.

Moreover, by condition (1),
$\mathbf{R} \pi_\ast$ commutes with the cohomology functors $H^i_{\mathcal{A}}$ with respect to the heart $\mathcal{A}$, 
which implies that every cohomology of $\cone(g)$ with respect to $\mathcal{A}$ lies in the kernel of $\mathbf{R}\pi_\ast$.
Particularly, by condition (2)
we see
$Z_{\mathcal{A}}(\ker H^0_{\mathcal{A}}(g))=Z_{\mathcal{A}}(\cok H^0_{\mathcal{A}}(g))=0$,

Consider the quotient $f \colon \mathbf{L}^0_{\mathcal{A}}\pi^\ast \mathbf{R}\pi_\ast E \longrightarrow  \mathbf{L}^0_{\mathcal{A}}\pi^\ast G$.
By the nine lemma, we may consider the following commutative diagram in $\mathcal{A}$:
\[
\begin{tikzcd}
T \arrow[r, hook] \arrow[d, hook]    & \ker (f) \arrow[d] \arrow[r, two heads]                                                                                   & \ker (f)/T \arrow[rd,  hook] \arrow[dd, hook]                                      &                                             &    \\
\ker (H^0_{\mathcal{A}}(g)) \arrow[r, hook] \arrow[d, two heads] &  \mathbf{L}^0_{\mathcal{A}}\pi^\ast \mathbf{R}\pi_\ast E \arrow[d, "f", two heads] \arrow[rd, two heads] \arrow[rr, "H^0_{\mathcal{A}}(g)"] &                                                                               & E \arrow[r, two heads]  & \cok(H^0_{\mathcal{A}}(g)) \\
\ker(H^0_{\mathcal{A}}(g))/T\arrow[r, hook]                  & \mathbf{L}^0_{\mathcal{A}}\pi^\ast G \arrow[rd, two heads]                                                                           & \Img (H^0_{\mathcal{A}}(g)) \arrow[d, two heads] \arrow[ru, hook]                                      &    \\ &           & \mathbf{L}^0_{\mathcal{A}}\pi^\ast G/(\ker(H^0_{\mathcal{A}}(g))/T)       &     &   
\end{tikzcd}
\]
Then, we see that $Z_{\mathcal{A}}(T)=0$ and $\mu_{Z_{\mathcal{A}}}(\ker(f)/T) = \mu_{Z_{\mathcal{A}}}(\ker(f)) > \mu_{Z_{\mathcal{A}}}(\mathbf{L}^0_{\mathcal{A}}\pi^\ast \mathbf{R}\pi_\ast E ) = \mu_{Z_{\mathcal{A}}}(E)$.
Hence, $\ker(f)/T$ is a destabilizing subobject of $E$, which is a contradiction.

We then turn to the second case.
Consider the composition $$\mathbf{L}^0_{\mathcal{A}}\pi^\ast K \longrightarrow \mathbf{L}^0_{\mathcal{A}}\pi^\ast \mathbf{R}\pi_\ast E \longrightarrow E.$$
The composition is nonzero as its pushforward is the inclusion $K \hookrightarrow \mathbf{R}\pi_\ast E$.
The image of this morphism gives a $Z_\mathcal{A}$-destabilizing subobject of $E$, which is also a contradiction.
\end{proof}

\begin{remark}\label{2}
In the above proof, the equality 
$M \simeq \mathbf{R}\pi_\ast \mathbf{L}\pi^\ast M = \mathbf{R}\pi_\ast \mathbf{L}^0_{\mathcal{A}} \pi^\ast M$ follows only from the condition (1),
and it implies particularly that the functor $\mathbf{R}\pi_\ast \colon \mathcal{A}  \longrightarrow \mathcal{B}$ is essentially surjective.
In the case where $\mathcal{A}={^{-1}\Per(\widetilde{X}/X)}$ and $\mathcal{B} =\Coh(X)$, this was also proven in \cite[Theorem 2.14]{BKS18}.
\end{remark}

We fix the following notations of slope stability.
\begin{itemize}
\item Let $Z_{\pi^\ast H}$ denote the central charge of slope stability with respect to $\pi^\ast H$ on ${^{-1}\Per(\widetilde{X}/X)}$ given by $-(\pi^\ast H)\cdot \ch_1+i\ch_0$.
\item Let $Z_H$ denote the central charge of slope stability with respect to $H$ on $\Coh(X)$ given by $-H\cdot \ch_1+i\ch_0$.

\end{itemize}

\begin{corollary}\label{3}
If $E \in  {^{-1}\Per(\widetilde{X}/X)}$ is $Z_{\pi^\ast H}$-semistable,
then $\mathbf{R}\pi_\ast E$ is a $Z_H$-semistable coherent sheaf.
\end{corollary}

\begin{proof}
We will apply Proposition \ref{preservess}.
The condition (1) follows from Lemma \ref{Bri02}.

For the condition (2), by  Grothendieck-Riemann-Roch,
for any $E \in {^{-1}\Per(\widetilde{X}/X)}$,
we have $\ch_0(E)=\ch_0(\mathbf{R}\pi_\ast E)$ and $(\pi^\ast H) \cdot \ch_1(E) = H \cdot \ch_1(\mathbf{R}\pi_\ast E)$.
Therefore, $Z_H(\mathbf{R}\pi_\ast E) = Z_{\pi^\ast H}(E)$. 
\end{proof}

Combining Lemma \ref{1} and Corollary \ref{3} shows the following result.

\begin{proposition}\label{pushforward}
The pushforward functor $\mathbf{R}\pi_\ast \colon \mathcal{B}^0 \longrightarrow \Coh^0_H(X)$ is exact.
\end{proposition}

As we are going to show that the stability conditions on $\widetilde{X}$ and on $X$ are compatible via $\pi_\ast$,
the essential surjectivity will be useful.

\begin{proposition}\label{esurj}
The pushforward functor $\mathbf{R}\pi_\ast \colon \mathcal{B}^0\longrightarrow \Coh^0_H(X)$ is essentially surjective.
\end{proposition}
\begin{proof}
Similarly as in Remark \ref{2}, 
by the exactness of $\mathbf{R}\pi_\ast \colon \mathcal{B}^0 \longrightarrow \Coh^0_H(X)$,
we see that for any $E \in \Coh^0_H(X)$,
$\mathbf{R}\pi_\ast \circ H^0_{\mathcal{B}^0} \circ \mathbf{L}\pi^\ast E \simeq E$.
\end{proof}

We are going to construct the pre-stability condition on $\widetilde{X}$ and then show that it descends to the singular surface $X$.
By Section 4, we know that the pair $(Z_{\pi^\ast H,\beta,z}, \mathcal{B}^0)$ gives a Bridgeland stability condition with support property for suitable choices of $\beta$ and $z$.
Note that $Z_{\pi^\ast H ,\beta,z}([\mathcal{O}_{C_i}(-1)])=1+\frac{C_i^2}{2}+\beta \cdot C_i=\beta \cdot C_i$.
Consider the functions $Z_{\pi^\ast H,\epsilon\beta,z}$.
If we assume that $z>\max \{0, -\frac{\beta^2}{2} \}$, then for every $0<\epsilon<<1$, the pair $(Z_{\pi^\ast H,\epsilon\beta,z},\mathcal{B}^0)$ is a stability condition by Section 4 and 5. 

For simplicity we denote \gls{weak-stability-function}.
Note that at the end point $\epsilon = 0$,
we have $Z_{\pi^\ast H,0,z}([\mathcal{O}_{C_i}(-1)])=0$ for all $i$,
and hence the pair \gls{weak-stability-condition} can not be a stability condition. 
However, we claim that it is a weak pre-stability condition.

\begin{proposition}\label{endpoint}
The pair $\sigma_{\widetilde{X}} =(Z_{\widetilde{X}},\mathcal{B}^0)$ forms a weak pre-stability condition.
\end{proposition}
\begin{proof}

Let $E$ be an object in $\mathcal{B}^0$.
We note that $Z_{\widetilde{X}}$ is a weak stability function on $\mathcal{B}^0$ as follows.
First, $\Im Z_{\widetilde{X}} (E)=\Im Z_{\pi^\ast H , \beta_\epsilon}(E) \geq 0$. 
If $\Im Z_{\widetilde{X}} (E)=\Im Z_{\pi^\ast H , \beta_\epsilon}(E) = 0$,
then  $\Re Z_{\pi^\ast H , \beta_\epsilon}(E)<0$ for all $\epsilon$,
and by continuity we see that $\Re Z_{\widetilde{X}}(E)\leq 0$.

For the Harder--Narasimhan property,  
one may first note that \cite[Prop. B.2]{BM11} can also be extended to the case of weak stability condition,
as all of them can be deduced from the analogue \cite[Prop. 2.4]{Bri07};
the proof of this version is exactly the same as that of Bridgeland stability conditions. 
Then since $\Im Z_{\widetilde{X}} = \Im Z_{\pi^\ast H,\beta,z}$,
Lemma \ref{HNprop} implies that $(Z_{\widetilde{X}},\mathcal{B}^0)$ admits Harder--Narasimhan property.
\end{proof}

By Lemma \ref{kernel}, $Z_{\widetilde{X}} \colon {\rm K}_{\num}(\widetilde{X}) \longrightarrow \mathbb{C}$ factors through $\pi_\ast \colon {\rm K}_{\num}(\widetilde{X}) \longrightarrow {\rm K}_{\num}(X)$, i.e. there is a function \gls{stability-function-X}  such that $Z_{\widetilde{X}} = Z_X$ $\circ \pi_\ast$. 

\begin{proposition}
The function $Z_X$ gives a stability function on $\Coh^0_H(X)$.
\end{proposition}
\begin{proof}
Given $E \in \Coh^0_H(X)$,  by Proposition \ref{esurj} we may write $E =\mathbf{R} \pi_\ast \widetilde{E}$ for some $\widetilde{E} \in \mathcal{B}^0$.
Recall that by the definition of $\mathcal{B}^0$,
we may obtain the following diagram,
with $G \in \mathcal{F}_0$ and $T \in \mathcal{T}_0$.

\[
\begin{tikzcd}
                 &             & {F[1]} \arrow[d] \\
{G[1]} \arrow[r] & \widetilde{E} \arrow[r] & T \arrow[d]              \\
                 &             & S             
\end{tikzcd}
\]

Since $Z_{\widetilde{X}}$ is a weak stability function and $Z_X(E)=Z_{\widetilde{X}}(\widetilde{E})$,
it now suffices to show that if $Z_{\widetilde{X}}(\widetilde{E})=0$ then $\mathbf{R}\pi_\ast \widetilde{E}=0$.
Note that $Z_{\widetilde{X}}(\widetilde{E})=0$ is equivalent to that 
$Z_{\widetilde{X}}(G[1])=Z_{\widetilde{X}}(S)=Z_{\widetilde{X}}(F[1])=0$.

First, as $G$ is supported on the exceptional locus $\Pi$,
$\mathbf{R}\pi_\ast (G[1])$ is supported on the singular point of $X$;
as we also know $\mathbf{R}\pi_\ast (G[1]) \in \Coh^0_H(X)$ is a sheaf of finite length, 
we easily see that $Z_{\widetilde{X}}(G[1])=0$ implies that $\mathbf{R}\pi_\ast (G[1])=0$.
Secondly, assume that $Z_{\widetilde{X}}(S) = 0$ and that $S$ is nonzero.
It can be seen that $S$ can only be supported on the union of points and $\Pi$.
Therefore, $\mathbf{R}\pi_\ast (S)$ must be supported on points on $X$ and hence $\Re Z_{\widetilde{X}}(S)= \Re Z_X(\mathbf{R}\pi_\ast S)<0$,
a contradiction.

We then turn to $F=H^{-1}(T) \in \mathcal{F}^{\leq 0}_{\pi^\ast H}$.
Assume that $Z_{\widetilde{X}}(F) = 0$ and that $F$ is nonzero.
It has slope $\mu_{\pi^\ast H }(F) :=  \frac{(\pi^\ast H)  \cdot \ch_1(F)}{ \ch_0(F)} \leq 0$ and $\Im Z_{\widetilde{X}}(F) = (\pi^\ast H)  \cdot \ch_1(F) =0$,
so it has positive rank and the slope is exactly $0$ (otherwise the slope should be $+\infty$).

Note that $F$ is of maximal possible slope in the category $\mathcal{F}^{\leq 0}_{\pi^\ast H}$,
so it must be $\mu_{\pi^\ast H}$-semistable and torsion-free.
Bogomolov-Gieseker shows that $\ch^2_1(F) \geq 2\ch_0(F)\ch_2(F)$.
On the other hand, as $(\pi^\ast H) \cdot \ch_1(F) = 0$, Hodge index theorem gives that $(\ch_1(F))^2 \leq 0$.
Now we obtain the inequality
\begin{equation*}
\begin{split}
0 = \Re Z_{\widetilde{X}}(F) = z \ch_0 (F) - \ch_2(F) \geq \ch_0 (F) ( z - \frac{(\ch_1(F))^2}{ 2 \ch^2_0(F)} ) > 0.
\end{split}
\end{equation*}
which is absurd.

In summary, if $Z_{\widetilde{X}}(\widetilde{E})=0$,
then $\widetilde{E} = G[1],T=0$ and $\mathbf{R}\pi_\ast \widetilde{E} = \mathbf{R}\pi_\ast G[1] =0$ as claimed.
\end{proof}

\begin{proposition}
The pair \gls{sigma-X} is a pre-stability condition.
\end{proposition}
\begin{proof}
It remains to prove the Harder--Narasimhan property.
Again, by the essential surjectivity we may write $E = \mathbf{R}\pi_\ast G$ for some $G \in \mathcal{B}^0$.
Then $G$ admits a Harder--Narasimhan filtration $ 0 = G_0 \subset G_1 \subset \cdots \subset G_{m-1} \subset G_m = G$ with the quotients $G_i/G_{i-1}$ $Z_{\widetilde{X}}$-semistable and the slopes (with respect to $Z_{\widetilde{X}}$) are decreasing.
This induces a filtration of $ 0 = E_0 \subset E_1 \subset \cdots \subset E_{m-1} \subset E_m = E$, where $E_i:=\mathbf{R}\pi_\ast G_i$.
Note that the slopes with respect to $Z_X$ are still decreasing.
To check that this is the desired Harder--Narasimhan filtration, we still need the $Z_X$-semistability of the quotients.
As $\mathbf{R}\pi_\ast$ is exact, the quotient $E_i/E_{i-1}= \mathbf{R}\pi_\ast G_i / \mathbf{R}\pi_\ast G_{i-1}$ is isomorphic to $\mathbf{R}\pi_\ast(G_i/G_{i-1})$,
so it suffices to show the following result.
\end{proof}

\begin{corollary}
The pushforward of a $Z_{\widetilde{X}}$-semistable object in $\mathcal{B}^0$ is $Z_X$-semistable.
\end{corollary}
\begin{proof}
We can apply Proposition \ref{preservess} directly.
The condition (1) follows from Proposition \ref{pushforward} and the condition (2) is just the definition. 
\end{proof}

At the end of this section,
we are going to prove the surjectivity of the morphism between the set of semistable objects.
We start with the following lemma:

\begin{lemma}\label{converse}
If $E \in \Coh^0_H(X)$ is a $Z_X$-semistable object,
then there exists a $Z_{\widetilde{X}}$-semistable object $\widetilde{E} \in \mathcal{B}^0$ with $\mathbf{R}\pi_\ast \widetilde{E}=E$.
\end{lemma}
\begin{proof}
By Proposition \ref{esurj}, we can first find an object $G \in \mathcal{B}^0$ with $\mathbf{R}\pi_\ast G=E$.
If $G$ is $Z_{\widetilde{X}}$-semistable then the assertion is proved.

We then assume that $G$ is $Z_{\widetilde{X}}$-unstable and consider the Harder--Narasimhan filtration of $G$ with respect to $\sigma_{\widetilde{X}}$.
That is, a sequence of inclusions
\[0=G_0 \subseteq G_1 \subseteq \cdots \subseteq G_n = G\]
with $G_i/G_{i-1} Z_{\widetilde{X}}$-semistable and $\mu_{Z_{\widetilde{X}}}(G_1) > \cdots > \mu_{Z_{\widetilde{X}}}(G_n/G_{n-1})$.

As $E$ is $Z_X$-semistable, $\mathbf{R}\pi_\ast(G_i/G_{i-1})=0$ except for one $i$.
But if $\mathbf{R}\pi_\ast(G_i/G_{i-1})=0$,
then $\mu_{Z_{\widetilde{X}}}(G_i/G_{i-1}) = + \infty$.
Therefore, the filtration can only be a two-step filtration $0 \subseteq G_1 \subseteq G$,
and now we pick $\widetilde{E}:= G/G_1$.
\end{proof}

To prove Theorem \ref{main3},
it thus remains to show that we can choose $\widetilde{E}$ to have a fixed Chern character $\widetilde{v}$.

\begin{lemma}\label{cartan}
Consider a linear combination $\displaystyle \sum_i a_iC_i$ with $a_i \in \mathbb{R}$.
If some of the $a_i$ are negative,
then there exists $k$ such that $a_k <0 $ and $\displaystyle C_k \cdot \sum_i a_iC_i >0$.
\end{lemma}
\begin{proof}
Let $a_{j_1}, a_{j_2},\dots,a_{j_m}$ be the negative coefficients, 
then we have to show that there exists some $j_k$ with $\displaystyle C_{j_k} \cdot \sum_i a_iC_i >0$.

For each $k$, we have
\begin{equation*} 
\begin{split}
\displaystyle C_{j_k} \cdot \sum_i a_iC_i &= \sum_i a_i C_{j_k}\cdot C_i \\
&= (\sum_{i\neq j_1,\dots,j_m}a_i C_{j_k}\cdot C_i ) +\sum^m_{l=1}a_{j_l}C_{j_k}\cdot C_{j_l} \\
&\geq \sum^m_{l=1}a_{j_l}C_{j_k}\cdot C_{j_l} 
\end{split}
\end{equation*}

We now consider the subdiagram generated by $C_{j_i}, \dots C_{j_m}$,
whose intersection matrix $B=(b_{kl})$ is just $(C_{j_k}\cdot C_{j_l} )$.
As a subdiagram of Dynkin diagram of ADE type is a disjoint union of Dynkin diagrams of ADE type,
the matrix $B$ is a direct sum of intersection matrices for exceptional curves of some ADE singularities.

Hence, by \cite[Proposition 10.18(iii)]{Car05}, if  $\displaystyle   \sum^m_{l=1} a_{j_l} C_{j_k} \cdot C_{j_l}\leq 0$ for all $k= 1, \dots m$,
then $a_{j_l} \geq 0$ for all $l$,
which is a contradiction.
This means that there exists some $k$ such that $\displaystyle C_{j_k} \cdot \sum_i a_iC_i >0$.
\end{proof}

We will need the following result about extension by massless object.
\begin{lemma}\label{massless}
Let $\sigma=(Z,\mathcal{A})$ be a weak stability condition. 
Consider a short exact sequence in the heart $\mathcal{A}$:
\[ 0 \longrightarrow A\longrightarrow B \longrightarrow Q  \longrightarrow 0.\]
Assume that $Z(Q)=0$.
\begin{enumerate}
    \item If $B$ is $Z$-semistable, then $A$ is $Z$-semistable.
    \item If $A$ is $Z$-semistable, $Q$ is simple and $\Hom_{\mathcal{A}}(Q,B)= 0$,
    then $B$ is $Z$-semistable.
\end{enumerate}
\end{lemma}
\begin{proof}
For (i), if $B$ is $Z$-semistable, then for any nonzero subobject $E\subseteq A$, 
it is also a subobject of $B$,
and hence $\mu_Z(E) \leq \mu_Z(B)= \mu_Z(A)$.

For (ii), we assume that $A$ is $Z$-semistable.
Let $E \subseteq B$ be a nonzero subobject and let $f$ be the composition $ E \hookrightarrow B \longrightarrow Q$.

By our assumptions, $f$ can not be injective.
Then we have a short exact sequence 
$$0 \longrightarrow \ker f\longrightarrow A \longrightarrow B/E  \longrightarrow 0,$$
and hence $\mu_Z(E) = \mu_Z(\ker f) \leq \mu_Z(A) = \mu_Z(B)$. 
\end{proof}

Let $\mathcal{M}_{\sigma}(v)$ be the set of $\sigma$-semistable objects with Chern character $v$.
\begin{lemma}\label{widetildev}
Given a class $v \in \Lambda$, there exists a class $\widetilde{v} \in \widetilde{\Lambda}$ such that
for any $E \in \mathcal{M}_{\sigma_X}(v)$, 
there exists an object $\widetilde{E} \in \mathcal{M}_{\sigma_{\widetilde{X}}}(\widetilde{v})$ such that $\mathbf{R}\pi_\ast \widetilde{E}=E$.
\end{lemma}
\begin{proof}
Let $\widetilde{v_0}$ be such that $\pi_\ast \widetilde{v_0} =v$.
Then there is a unique class $\widetilde{v} \in \widetilde{v_0}+\ker \pi_\ast$ such that its projection to the kernel can be written as $\displaystyle \pr_{\ker \pi_\ast}(\widetilde{v}) = \sum_i a_i C_i$ with $0\leq a_i <1$ for all $i$;
here $\pr_{\ker \pi_\ast}\colon \widetilde{\Lambda} \longrightarrow \ker \pi_\ast$ is the orthogonal projection to the sublattice $\ker \pi_\ast$.

It then suffices to show that there exists a $Z_{\widetilde{X}}$-semistable object $\widetilde{E}$ of class $\widetilde{v}$ such that $\mathbf{R}\pi_\ast \widetilde{E}=E$.
By Lemma \ref{converse}, there exists a $Z_{\widetilde{X}}$-semistable object $E' \in \mathcal{B}^0$ such that $\mathbf{R}\pi_\ast E'=E$.
We write $\displaystyle \pr_{\ker \pi_\ast}(\ch_1(E')) = \sum_i b_i C_i$.
Note that the $b_i$ might be rational numbers, not necessarily integers.

By Hirzebruch-Riemann-Roch, $\chi(\mathcal{O}_{C_i}(-1)[1],E')= C_i\cdot \ch_1(E')$ for any $i$.
Note that by semistability and Serre duality $\Hom(\mathcal{O}_{C_i}(-1)[1],E'[j]) = 0$ unless $j=1,2$,
and hence  $$\chi(\mathcal{O}_{C_i}(-1)[1],E')= \dim \Hom(E',\mathcal{O}_{C_i}(-1)[1]) - \dim \Ext^1(\mathcal{O}_{C_i}(-1)[1], E).$$

Now if $\ch(E') \neq \widetilde{v}$, then for some $j$, we either have $b_j \geq 1$ or some $b_j < 0$.
Assuming that there is some $b_j \geq 1$, 
we see by Lemma \ref{cartan} that there exists $k$ such that $b_k > 0$ and $0 > C_k \cdot \ch_1(E')=\chi(\mathcal{O}_{C_k}(-1)[1],E')$.

This implies that $\dim \Ext^1(\mathcal{O}_{C_k}(-1)[1],E')>0$ and we can consider a nontrivial extension  $0 \longrightarrow E' \longrightarrow F \longrightarrow \mathcal{O}_{C_k}(-1)[1] \longrightarrow 0$ in $\mathcal{B}^0$.

If $\Hom(\mathcal{O}_{C_k}(-1)[1], F)$ is nonzero,
then since $F$ is a nontrivial extension and $\mathcal{O}_{C_k}(-1)[1]$ is simple,
the composition map $(\mathcal{O}_{C_k}(-1)[1] \longrightarrow F \longrightarrow \mathcal{O}_{C_k}(-1)[1])$ must be $0$.
Hence, it factors to give a nonzero morphism $\mathcal{O}_{C_k}(-1)[1] \longrightarrow E'$, which is a contradiction.
Therefore, the conditions of Lemma \ref{massless} are satisfied and so $F$ is $Z_{\widetilde{X}}$-semistable.

Replacing $E'$ by $F$,
we obtain a $Z_{\widetilde{X}}$-semistable object $F$ such that $\mathbf{R}\pi_\ast F=E$ and $\displaystyle \pr_{\ker \pi_\ast}(\ch(F)) = (b_j-1)C_j+\sum_{i\neq j} b_i C_i$.
Repeating this process if necessary,
we reach a $Z_{\widetilde{X}}$-semistable object $E''$ such that $\mathbf{R}\pi_\ast E''=E$ and $\displaystyle \pr_{\ker \pi_\ast}(\ch(F)) = \sum_{i} b'_i C_i$ with $b'_i \leq 0$ for all $i$.

Now, if there exists any $j$ with $b'_j <0$, then Lemma \ref{cartan} implies that there exists some $l$ such that $b'_l<0$ and $\chi(\mathcal{O}_{C_l}(-1)[1],E'')>0$.

Then $\dim \Hom(E'',\mathcal{O}_{C_l}(-1)[1])$ must be positive,
and hence there exists a nonzero morphism $f\colon E''\longrightarrow \mathcal{O}_{C_l}(-1)[1]$ and $\displaystyle \pr_{\ker \pi_\ast}(\ch(\ker f)) = (b'_l+1)C_l+\sum_{i\neq l} b_i C_i$.
Note that by Lemma \ref{massless}(i), $\ker f$ is still $Z_{\widetilde{X}}$-semistable.

Replacing $E''$ by $\ker f$ and repeating this process if necessary, we obtain a subobject $\widetilde{E}\subseteq E''$ so that $\displaystyle \pr_{\ker \pi_\ast}(\ch(\widetilde{E}))= \sum_i a_i C_i$ with $0\leq a_i<1$ for all $i$,
that is $\widetilde{E}$ is of class $\widetilde{v}$.
\end{proof}

Combining the lemmas above, we can prove the following result.
\begin{theorem}\label{surjmoduli}
For any class $v \in \Lambda$, there exists a class $\widetilde{v} \in \widetilde{\Lambda}$ such that there is a surjective map $\pi_\ast \colon \mathcal{M}_{\sigma_{\widetilde{X}}}(\widetilde{v}) \longrightarrow \mathcal{M}_{\sigma_X}(v)$.
\end{theorem}

One may note that by Theorem \ref{surjmoduli}, if $\mathcal{M}_{\sigma_{\widetilde{X}}}(\widetilde{v})$ satisfies boundedness (in the sense of \cite[Definition 9.4]{BLM+}), 
then so does $\mathcal{M}_{\sigma_X}(v)$ (see also Theorem \ref{Final}.

\section{Support property on the singular surface}
In this section, we prove the support property of the stability condition $\sigma_X =(Z_X,\mathcal{B}^0)$.
As above, we want to take advantage of the essential surjectivity of $\mathbf{R} \pi_\ast$.
We claim that if $\sigma_{\widetilde{X}}$ admits the support property with respect to the lattice $\widetilde{\Lambda}/ \ker \pi_\ast$,
then we can obtain an induced quadratic form for the support property on $\sigma_X$.

As in Section 6, we let $\pr \colon {\rm K}_{\num}(\widetilde{X}) \longrightarrow {\rm K}_{\num}(\widetilde{X})/\ker\pi_\ast$ be the projection. 
Note that the Chern character of $\pr E$ is the same as the Chern character of $\mathbf{R}\pi_\ast E$
on $X$.
Since $\widetilde{\Lambda}/ \ker \pi_\ast \simeq \Lambda$, a quadratic form $Q_{\widetilde{X}}$ on $\widetilde{\Lambda}/ \ker \pi_\ast$ can be identified with a quadratic form $Q_X$ on $\Lambda$.
By abuse of notations, we write $Z_X$ for the map $\widetilde{\Lambda}/ \ker \pi_\ast \longrightarrow \mathbb{C}$ that $Z_{\widetilde{X}}$ factors through.

Now given a $\sigma_X$-semistable object $F$ in $\Coh^0_H(X)$,
we can write $F= \mathbf{R} \pi_\ast G$ for some $\sigma_{\widetilde{X}}$-semistable $G$ by Lemma \ref{converse},
and then $Q_X(F)=Q_{\widetilde{X}}(G) \geq 0$.
On the other hand, by the compatibility $Z_X \circ \pi_\ast = Z_{\widetilde{X}}$, 
if $F \in \ker Z_X$
and $F= \mathbf{R} \pi_\ast G$,
then $G$ must be in $\ker Z_{\widetilde{X}}$ and thus $Q_X(F) = Q_{\widetilde{X}}(G) < 0$.

We will extend the argument in Section 5 to construct a quadratic form on $(\widetilde{\Lambda}/ \ker \pi_\ast) \otimes \mathbb{R} \simeq \Lambda \otimes \mathbb{R}$.
First, we note that on the singular surface $X$ we can define the intersection products and Chern classes (see, for example, \cite{Lan25}).

By the openness of the ample cone on $X$, there is a constant $A>0$, 
depending only on the ample divisor $H$, such that $ C^2 + A(H \cdot C)^2 \geq 0$ for any curve $C \subseteq X$.
Now we define a quadratic form on $\widetilde{\Lambda}/ \ker \pi_\ast$ by
\begin{equation}
\gls{quadratic-form-X}
\end{equation}
By abuse of notation, for $E \in \mathcal{B}^0$, we write \[Q_{\widetilde{X}}(E):= Q_{\widetilde{X}}(v_\pi(E)) = \Delta(\pr [E]) + A(\Im Z_{\widetilde{X}}(E))^2.\]

Similarly, we need to verify the followings two statements:
\begin{enumerate}
    \item $Q_{\widetilde{X}}$ is negative definite on $\ker Z_{\widetilde{X}}$.
    \item For any $E \in \mathcal{B}^0$ which is $Z_{\widetilde{X}}$-semistable, we have $Q_{\widetilde{X}}(v_\pi(E)) \geq 0$.
\end{enumerate}

As in Section 5, for any $s\geq 1$, we define the function $Z_{\widetilde{X},s}\colon \widetilde{\Lambda}/ \ker \pi_\ast \longrightarrow \mathbb{C}$ by 
\[Z_{\widetilde{X},s}(E):= Z_{\widetilde{X}}(E) + (s-1) z \ch_0(E)\]
and the pair $\sigma_{\widetilde{X},s}:=(Z_{\widetilde{X},s},\mathcal{B}^0)$.
With the assumptions of Lemma \ref{stabfunc} and $z>0$,
$\sigma_{\widetilde{X},s}$ also gives a weak stability condition for each $s$ and one can easily verify this as in Proposition \ref{endpoint}.

We can immediately prove the condition (i) by showing the following analogue of Lemma \ref{negdef}:

\begin{lemma}\label{negdefpr}
$Q_{\widetilde{X}}$ is negative definite on $\ker Z_{\widetilde{X},s}\subseteq (\widetilde{\Lambda}/ \ker \pi_\ast) \otimes \mathbb{R}$ for all $s \geq 1$.
\end{lemma}
\begin{proof}
Given a nonzero class $x$ in $(\widetilde{\Lambda}/ \ker \pi_\ast) \otimes \mathbb{R}$ and
assume that $Z_{\widetilde{X},s}(x)=0$ for some $s\geq 1$, i.e.
\begin{equation}\label{eqn:ZX=0}
 -\ch_2(x)  + sz \ch_0(x) = 0 \rm{\ and \ } (\pi^\ast H) \cdot \ch_1(x) = 0. 
\end{equation}

As $\ch_2(x) = sz \cdot \ch_0(x)$,
we see that $2 \ch_2(x) \ch_0(x) \geq 0$. 
It then suffices to show that $(\ch_1(x))^2 < 0$. 
By the equations $(\ref{eqn:ZX=0})$, we see that $\ch_1(x)$ is orthogonal to $\pi^\ast H$.
Note that although $\pi^\ast H$ is only nef, 
$(\pi^\ast H)^2 >0$ as $H$ is ample on $X$,
and hence by Hodge index theorem the assertion is proved.
\end{proof}

As for the condition (ii), as in Section 5, we can define the set $$\mathcal{D}_0:= \{E \in \mathcal{B}^0 \, |\, E {\rm \ is\ } Z_{\widetilde{X},s}{\rm {\text -}semistable \ for \ sufficiently \ large \ } s  \}$$  and classify objects in $\mathcal{D}_0$ similarly.
The proof of the following lemma is verbatim the same as Lemma \ref{6.5'}.

\begin{lemma}\label{6.5''}
Given $E$ in the set $\mathcal{D}_0$, then $E$ must be one of the following forms:
\begin{enumerate}
    \item[(1)] $E=H^0(E)$ is a torsion sheaf;
    \item[(2)] $E=H^0(E)$ fits in a short exact sequence $$ 0 \longrightarrow E_{tor} \longrightarrow E \longrightarrow E_{tf} \longrightarrow 0$$
    where $E_{tf}$ is a torsion-free slope semistable sheaf in $\mathcal{T}^{> 0}_{\pi^\ast H}$, and $E_{tor}$ is a (possibly zero) torsion sheaf with $\mathbf{R}^0 \pi_\ast E_{tor}= 0$.  
    \item[(3)] $H^0(E)$ is either $0$ or a sheaf supported on the union of points and $\Pi$, and $H^{-1}(E)$ fits into a short exact sequence
    $$ 0 \longrightarrow G \longrightarrow H^{-1}(E) \longrightarrow F \longrightarrow 0$$
    where $F$ is a torsion-free slope semistable sheaf in $\mathcal{F}^{\leq 0}_{\pi^\ast H}$, and $\mathbf{R}^0 \pi_\ast G = 0$.
    Moreover, $G$ must be $0$ unless $(\pi^\ast H)\cdot \ch_1(F) =0$.
\end{enumerate}
\end{lemma}

We deal with the $Z_{\widetilde{X}}$-semistable objects in $\mathcal{D}_0$ with positive imaginary part and the semistable objects in $\mathcal{B}^0$ with $\Im Z_{\widetilde{X}}(E)=0$ respectively.
\begin{lemma}\label{6.8'}
If $E$ is an object in $\mathcal{D}_0$ with $\Im Z_{\widetilde{X}}(E)>0$, 
then $Q_{\widetilde{X}}(E) \geq 0$.
\end{lemma}
\begin{proof}
Given $E\in \mathcal{D}_0$ with $\Im Z_{\widetilde{X}}(E)>0$, 
the same proof of Lemma \ref{mathcalD} has already shown that $\Delta(E)+A_0 (\Im Z_{\widetilde{X}}(E))^2 \geq 0.$

Note first that $X$ is $\mathbb{Q}$-factorial.
Therefore, we can still define the pullback $\pi^\ast \colon \NS(X) \longrightarrow \NS(\widetilde{X})$,
and $\pi^\ast \pi_\ast \colon \NS(\widetilde{X}) \longrightarrow \NS(\widetilde{X})$ is the orthogonal projection with kernel generated by the curves $[C_i]$,
and the kernel is negative definite with respect to the intersection pairing.

Recall that ${\rm K}_{\num}(\widetilde{X}) \simeq \mathbb{Z} \oplus \NS(\widetilde{X}) \oplus \mathbb{Z}$ and $\pi^\ast \pi_\ast$ preserves $\ch_0$ and $\ch_2$,
so we similarly obtain that $\pi^\ast \pi_\ast \colon {\rm K}_{\num}(\widetilde{X}) \longrightarrow {\rm K}_{\num}(\widetilde{X})$ is a orthogonal projection with kernel generated by the classes $[\mathcal{O}_{C_i}(-1)]$; 
furthermore, the kernel must be negative definite with respect to the quadratic form $\Delta = \ch_1^2-2\ch_0\ch_2$.

Now for any class $[E]$, 
we have $$\Delta ([E]) \leq \Delta (\pi^\ast \pi_\ast [E])+\Delta ([E] - \pi^\ast \pi_\ast [E]) =  \Delta (\pi_\ast [E])+\Delta ([E] - \pi^\ast \pi_\ast [E]).$$ 

As $[E] - \pi^\ast \pi_\ast [E] \in \ker \pi_\ast$, we have $\Delta ([E] - \pi^\ast \pi_\ast [E])\leq0$ by negative definiteness, and hence $\Delta(\pr [E]) = \Delta(\pi_\ast [E]) \geq \Delta([E])$.
To sum up, we have
\[Q_{\widetilde{X}}(E) = \Delta (\pr [E]) + A_0 (\Im Z_{\widetilde{X}}(E))^2 \geq \Delta(E)+A_0 (\Im Z_{\widetilde{X}}(E))^2 \geq 0,\]
which proves the assertion.
\end{proof}

We can also deduce the Bogomolov-Gieseker inequality on $X$.
More precisely, we have the following theorem.
\begin{theorem}\label{B-G on X}
Given a slope semistable torsion-free sheaf $G$ on $X$,
then we claim that $\Delta_X(G) := \ch_1^2(G) - 2 \ch_0(G)\ch_2(G) \geq 0$.
\end{theorem}
\begin{proof}
We start with considering a perverse coherent sheaf $E \in {^{-1}}\Per(\widetilde{X}/X)$ such that $G= \mathbf{R} \pi_\ast E$.
By definition, $E$ fits into a triangle $F[1] \longrightarrow E \longrightarrow T$ where $F=H^{-1}(E)$ and $T = H^0(E)$.

First, by considering the pushforward of the triangle $F[1] \longrightarrow E \longrightarrow T$,
we see that $\mathbf{R} \pi_\ast F[1] = 0$; otherwise it will be a torsion subsheaf of $G$.
We may then assume $E=T \in \mathcal{T}_P$.

Then we note that $E$ must be a slope semistable perverse coherent sheaf (that is, the converse of Corollary \ref{3} holds),
otherwise the pushforward of a destabilizing subobject of $E$ gives a destabilizing subobject of $G$.

Now if $E$ is a torsion-free slope semistable coherent sheaf, then Bogomolov-Gieseker on $\widetilde{X}$ shows that $\Delta_{\widetilde{X}}(T) \geq 0$,
and hence by the same computation in Lemma \ref{6.8'},
we see $\Delta_X(G) \geq \Delta_{\widetilde{X}}(T) \geq 0$.

Otherwise, if there is a quotient $f\colon E\longrightarrow Q$ in $\Coh(\widetilde{X})$ with $\mu_{\pi^\ast H}(E)<\mu_{\pi^\ast H}(Q)$.
We then consider the long exact cohomology sequence with respect to perverse coherent sheaf:

\[0 \longrightarrow H^0_P(\ker f) \longrightarrow E \longrightarrow Q
\longrightarrow H^1_P(\ker f)
\longrightarrow 0\]
and then the quotient $H^0_{P}(f) \colon E \longrightarrow E/H^0_P(\ker f)$ in ${^{-1}}\Per(\widetilde{X}/X)$ , which is a contradiction.

We may then assume that there is no quotient sheaf of $E$ with smaller slope.
Since $E$ is not torsion-free slope semistable, the only possibility is that there exists a subobject $K \subseteq E$ in $\Coh(\widetilde{X})$ with $Z_{\pi^\ast H}(K)=0$.

We now pick the maximal torsion subsheaf $K \subseteq E$.
Then the quotient $E/K$ must be a torsion-free slope semistable coherent sheaf.

Note that $\mathbf{R}^0\pi_\ast K$ is a subsheaf of $G$ with rank $0$, so it can only be $0$ and hence $K \in \mathcal{F}_P$.
This gives a short exact sequence in ${^{-1}}\Per(\widetilde{X}/X)$:
\[ 0 \longrightarrow E \longrightarrow E/K\longrightarrow K[1]
\longrightarrow 0.\]

Applying $\mathbf{R}\pi_\ast$ to this exact sequence,
we obtain the short exact sequence in $\Coh(X)$:
\[ 0 \longrightarrow Q \longrightarrow \mathbf{R}\pi_\ast(E/K) \longrightarrow \mathbf{R}\pi_\ast(K[1])
\longrightarrow 0.\]

As $\mathbf{R}\pi_\ast(K[1])$ is supported on the singular point $x_0$, we finally see that
\[\Delta_X(G)\geq \Delta_X(\mathbf{R}\pi_\ast(E/K))\geq \Delta_{\widetilde{X}}(E/K).\]

This proves the assertion as the Bogomolov-Gieseker on $\widetilde{X}$ implies that $\Delta_{\widetilde{X}}(E/K)\geq 0$.
\end{proof}

\begin{lemma}\label{I=0}
If $E \in \mathcal{B}^0$ is a $Z_{\widetilde{X}}$-stable object with $\Im Z_{\widetilde{X}}(E)=0$,
then $Q_{\widetilde{X}}(E) \geq 0$.
\end{lemma}
\begin{proof}
Consider $E \in \mathcal{B}^0$ with $E$ being $Z_{\widetilde{X}}$-stable and $\Im Z_{\widetilde{X}}(E) = 0$.
For such $E$, we again take advantage of the decomposition:
\[\begin{tikzcd}
                 &             & {F[1]} \arrow[d] \\
{G[1]} \arrow[r] & E \arrow[r] & T \arrow[d]      \\
                 &             & S               
\end{tikzcd}\]
where $E$ is equal to one of $G[1], F[1],$ and $S$.

If $E=F[1]$, then $Q_{\widetilde{X}}(E)\geq \Delta(\pr [F]) \geq \Delta(F) \geq 0$ as in Lemma \ref{6.8'}.
If $E=G[1]$, then $E$ must be one of $\mathcal{O}_{C_i}(-1)[1]$, so $Q_{\widetilde{X}}(E)\geq \Delta(\pr [E]) = 0$.
If $E=S$, then $\mathbf{R}\pi_\ast E$ can only be supported on points, and hence $Q_{\widetilde{X}}(E)\geq \Delta(\pr [E]) = \Delta_X(\mathbf{R}\pi_\ast E)=0$
\end{proof}

We can now conclude the proof of (ii) and obtain the support property of $\sigma_{\widetilde{X}}$.
\begin{theorem}
The stability condition $\sigma_{\widetilde{X}}=(Z_{\widetilde{X}},\mathcal{B}^0)$ satisfies the support property,
with respect to the lattice $\widetilde{\Lambda}/ \ker \pi_\ast$ and the quadratic form $Q_{\widetilde{X}}$,
and it induces a stability condition $\sigma_X=(Z_X,\Coh^0_H(X))$ and $\Stab(X) \neq \varnothing$.
\end{theorem}
\begin{proof}
With Lemma \ref{negdefpr}, \ref{6.8'} and \ref{I=0},
the proof is almost exactly the same as Theorem \ref{mainSect5}.
The only thing to notice is that every $Z_{\widetilde{X}}$-semistable object $E$ with $\Im Z_{\widetilde{X}}(E)=0$ admits a Jordan--Hölder filtration so we can pass to the $Z_{\widetilde{X}}$-stable objects similarly.
This is because $\sigma_{\widetilde{X}}$ and $\sigma_\epsilon$ have the same heart $\mathcal{B}^0$ and $\mathcal{P}_{\sigma_{\widetilde{X}}}(1) =\mathcal{P}_{\sigma_\epsilon}(1)$. 

We now see that $Q_{\widetilde{X}}$ is the desired quadratic form on the lattice $\widetilde{\Lambda}/ \ker \pi_\ast \simeq \Lambda$ which gives the support property of $\sigma_{\widetilde{X}}$,
and this induces the support property of $\sigma_X=(Z_X,\Coh^0_H(X))$.
\end{proof}

\section{Moduli spaces}

In this final section, our aim is to prove that the moduli spaces of semistable objects for the (weak) stability conditions $\sigma_\epsilon$, 
$\sigma_{\widetilde{X}}$ and $\sigma_X$ are Artin stacks of finite type over $\mathbb{C}$.
We will achieve this by proving boundedness and openness of the moduli stacks.

\begin{definition}
For a variety $V$, we let $\mathcal{M}_V$ be the 2-functor
\[(Sch/\mathbb{C}) \longrightarrow (Groupoid)\]
which sends $S$ to $\mathcal{M}_V(S):=\{ \mathcal{E} \in \mathrm{D}^b(V\times S) \  | \ \mathcal{E}$ is relatively perfect and $\Ext^i(\mathcal{E}_s, \mathcal{E}_s)=0 $ for any $i<0, s\in S\}$,
where $\mathcal{E}_s:=\mathbf{L}{i_s}^\ast \mathcal{E}$.
\end{definition}

We define some properties for the moduli stacks:
\begin{definition}

Let $\mathcal{M}$ be an arbitrary subgroupoid of $\mathcal{M}_V$.
\begin{enumerate}
    \item[(1)] We say that $\mathcal{M}$ satisfies boundedness if there exists a scheme $S$ of finite type over $\mathbb{C}$ and a family $\mathcal{E} \in \mathrm{D}^b(\widetilde{X} \times S)$ such that for every $E \in \mathcal{M}(\mathbb{C})$,
    $E \simeq \mathcal{E}_s$ for some $s \in S$. 
    \item[(2)] We say that $\mathcal{M}$ satisfies openness if for any scheme $S$ and a family $\mathcal{E} \in \mathrm{D}^b(\widetilde{X} \times S)$, the subset $\{ s\in S \ | \ \mathcal{E}_s \in \mathcal{M}(\mathbb{C}) \} $ is open in $S$.
\end{enumerate}
\end{definition}

Next, we introduce Toda's notion of generic flatness.
For a heart of a bounded t-structure $\mathcal{A}$ on ${\rm D}^b(V)$ which is a tilt of a Noetherian heart,
Polishchuk constructs a heart on $V\times S$ associated to $\mathcal{A}$.

\begin{theorem}[{\normalfont\cite[Theorem 3.3.6]{Pol07}}]\label{Pol07}
Take a smooth projective variety $S$ and an ample line bundle $\mathcal{L}$ on $S$.
Let $\mathcal{A}\subseteq {\rm D}^b(V)$ be a  heart of a bounded t-structure so that $\mathcal{A}$ is a tilt of a Noetherian heart. 

If we set $\mathcal{A}_S:= \{F \in \mathrm{D}^b(V \times S) \ |\ \mathbf{R}p_\ast (F \boxtimes \mathcal{L}^n) \in \mathcal{A} \text{ for all } n \gg 0\}$,
where $p \colon V\times S \longrightarrow V$ is the projection,
then $\mathcal{A}_S$ is a heart of bounded t-structure on ${\rm D}^b(V\times S)$, independent of the choice of $\mathcal{L}$.
\end{theorem}

\begin{definition}
Let $\mathcal{A}$ be a heart of a bounded t-structure on ${\rm D}^b(V)$ which is a tilt of a Noetherian heart.

We say that the generic flatness holds for the heart $\mathcal{A}$ if, for every smooth projective variety $S$ and every $\mathcal{E}\in \mathcal{A}_S$, there exists a non-empty open subset $U\subseteq S$ such that for any $s \in U$, we have $\mathcal{E}_s \in \mathcal{A}$.
\end{definition}

For a Bridgeland stability condition $\sigma=(Z,\mathcal{A})$ on $\mathrm{D}^b(V)$ and any numerical class $v$, 
we let $\mathcal{M}_{\sigma}(v)$ be the subgroupoid of $\sigma$-semistable objects of class $v$ of $\mathcal{M}_V$.

\begin{theorem}
Given a variety $V$ and an arbitrary subgroupoid $\mathcal{M}$ of $\mathcal{M}_V$.
\begin{enumerate}
    \item \textnormal{(\cite[Theorem 4.2.1]{Lie06})} $\mathcal{M}_V$ is an Artin stack of locally finite type over $\mathbb{C}$.
    \item \textnormal{(\cite[Lemma 3.4]{Tod08})}
    If $\mathcal{M}$  satisfying boundedness and openness,
    then $\mathcal{M}$ is an Artin stack of finite type over $\mathbb{C}$.
    \item \textnormal{(\cite[Theorem 3.20]{Tod08})} For a Bridgeland stability condition $\sigma=(Z,\mathcal{A})$ with $\mathcal{A}$ Noetherian, if $\mathcal{M}_{\sigma}(v)$ is bounded for every class $v$, and the generic flatness holds for the heart $\mathcal{A}$, 
    then openness holds for every $\mathcal{M}_{\sigma}(v)$.
\end{enumerate}
\end{theorem}

We first establish the generic flatness for the heart $\mathcal{B}^0$.

As $\mathcal{B}^0$ is a Noetherian heart by Corollary \ref{Noeth}, we can consider the heart $\mathcal{B}^0_S$ defined in Theorem \ref{Pol07}.
\begin{proposition}\label{genericflat}
The heart $\mathcal{B}^0$ of ${\rm D}^b(\widetilde{X})$ satisfies the generic flatness.
\end{proposition}

\begin{proof}
Let $S$ be a smooth projective variety and $\mathcal{L}$ an ample line bundle on $S$.

Let $\mathcal{E} \in \mathcal{B}^0_S$.
By definition, $\mathbf{R}p_\ast (\mathcal{E} \boxtimes \mathcal{L}^n)$ is concentrated in degree $0,-1$, for $n$ large,
and hence we may consider the following spectral sequence (cf. \cite[Lemma 4.7]{Tod08}):
\[ E^{ij}_2= \mathbf{R}^ip_\ast   (H^j(\mathcal{E})\boxtimes \mathcal{L}^n
) \Longrightarrow \mathbf{R}^{i+j}p_\ast
(\mathcal{E}\boxtimes \mathcal{L}^n)
\]
to see that $H^i(\mathcal{E})=0$ if $i\neq 0,-1$.

Now, we take an open subset $W \subseteq S$ where $H^0(\mathcal{E})|_{\widetilde{X}\times W}$ and $H^{-1}(\mathcal{E})|_{\widetilde{X}\times W}$ are both flat over $W$, that is,
$H^0(\mathcal{E})_s=H^0(\mathcal{E}_s)$ and $H^{-1}(\mathcal{E})_s = H^{-1}(\mathcal{E}_s)$ for all $s \in W$.
It then remains to find an open subset $U \subseteq W \subseteq S$ such that for any $s \in U$, we have $H^0(\mathcal{E})_s \in \mathcal{T}_{\mathcal{B}^0}$ and $H^{-1}(\mathcal{E})_s \in \mathcal{F}_{\mathcal{B}^0}$.

For the first condition, being in $\mathcal{T}_0$ is an open condition since $E \in \mathcal{T}_0$ if and only if $\Hom_{\widetilde{X}}(E,\mathcal{O}_{C_i}(-1))=0$ for all $i=1,\dots, n$.
Moreover, by \cite[Lemma 4.7]{Tod08}, being in $\mathcal{T}^{>0}_{\pi^\ast H}$ is also open.
Therefore, we see that being in $\mathcal{T}_{\mathcal{B}^0}=\mathcal{T}_0\cap \mathcal{T}^{>0}_{\pi^\ast H}$ is an open condition.

For the second condition, we consider the maximal torsion subsheaf $F \subseteq H^{-1}(\mathcal{E})$. 
As being torsion-free is open,
we may find an open subset $V \subseteq S$ where  $F|_{\widetilde{X}\times V}$ and $H^{-1}(\mathcal{E})|_{\widetilde{X}\times V}$ are both flat over $V$ and $(H^{-1}(\mathcal{E})/F)|_V$ is torsion-free.
Then on $V$, we have $H^{-1}(\mathcal{E})_s \in \mathcal{F}_{\mathcal{B}^0}$ if and only if $(H^{-1}(\mathcal{E})/F)|_s \in \mathcal{F}^{\leq 0}_{\pi^\ast H}$ and $\mathbf{R}^0\pi_\ast (F|_s)=0$,
and these conditions are both open by \cite[Lemma 4.7]{Tod08} and semicontinuity respectively.
\end{proof}

To prove that $\mathcal{M}_{\sigma_\epsilon}(\widetilde{v})$ is bounded,
we start with the following dual version of \cite[Lemma 5.9]{BM11}.

\begin{lemma}\label{dual of 5.9}
Let $\sigma$ be a stability condition on a triangulated category $\mathcal{D}$ 
and let $E$ be a $\sigma$-semistable object with a Jordan--Hölder filtration of the form $A \hookrightarrow E \twoheadrightarrow B$,
where $A$ is $\sigma$-stable and $B$ is an r-fold iterated self-extension of a $\sigma$-stable object $C$.

Assume that $\Hom(C,E)=0$ and that the classes $[A],[C]$ are linearly independent in ${\rm K}_{\num}(\mathcal{D})$.
Then any stability condition $\sigma'$ sufficiently close to $\sigma$ with $\phi'(A)<\phi'(E)<\phi'(C)$ lies in the set $\{\sigma'  |  E \text{ is } \sigma' \text{-stable} \}$.

In particular, $\sigma$ is in the closure of $\{\sigma'  |  E \text{ is } \sigma' \text{-stable} \}$.
\end{lemma}
\begin{proof}
Let $\phi$ be the phase of $A,E,C$ with respect to $\sigma=(Z,\mathcal{P})$.
For $\sigma'$ sufficiently close to $\sigma$, 
$E$ can only be destabilized by quotients $E \twoheadrightarrow Q$ in $\mathcal{P}(\phi)$.

Now let $\sigma'=(Z',\mathcal{P}')$ be a stability condition close by with $\phi'(A)<\phi'(E)<\phi'(C)$ and $A,C$ $\sigma'$-stable,
and assume that $Q \in \mathcal{P}(\phi)$ is a $\sigma'$-destabilizing quotient of $E$ which is $\sigma'$-stable.
In particular, $\phi'(Q)<\phi'(E)$.
Since both $C$ and $Q$ are $\sigma'$-stable, 
and $\phi'(Q)<\phi'(E)<\phi'(C)$,
the composition $A\hookrightarrow E \twoheadrightarrow Q$ cannot be zero.
Furthermore, as $A$ is simple in $\mathcal{P}(\phi)$,
this composition map must be injective.

Consider the cokernel $D$ of this composition,
which is also a quotient of $B$ and thus must be a $k$-time self extension of $C$ for some $k<r$.
The kernel of $B \twoheadrightarrow D$ is a subobject of $E$, 
but it is also a self-extension of $C$,
which contradicts $\Hom(C,E)=0$,
and so such a $\sigma'$-destabilizing quotient $Q$ of $E$ can not exist.
\end{proof} 

Applying Lemma \ref{dual of 5.9} iteratively, we obtain the following proposition:

\begin{proposition}\label{5.9iterative}
Let $\sigma$ be a stability condition on a triangulated category $\mathcal{D}$ 
and let $E$ be a $\sigma$-semistable object with a Jordan--Hölder filtration of the following form:
\[
\begin{tikzcd}
0 \arrow[r, hook] & A \arrow[r, hook] & E_1 \arrow[r, hook] \arrow[d, two heads] & E_2 \arrow[r, hook] \arrow[d, two heads] & \cdots \arrow[r, hook] & E_n=E \arrow[d, two heads] \\
&     & B_1          & B_2                       &     & B_n             
\end{tikzcd}
\]
where $A$ is $\sigma$-stable and each $B_i$ is $\sigma$-semistable and an iterated self-extension of some $\sigma$-stable object $C_i, i=1,\dots,n$.

Assume that $\Hom(C_i,E_i)=0$ and that the classes $[A],[B_1], \dots [B_n]$ are linearly independent in ${\rm K}_{\num}(\mathcal{D})$.
Then any stability condition $\sigma'$ sufficiently close to $\sigma$ satisfying all the inequalities $\phi'(A)<\phi'(B_1), \phi'(E_1)<\phi'(B_2),\cdots, \phi'(E_{n-1})<\phi'(B_n)$ lies in the set $\{\sigma'  |  E \text{ is } \sigma' \text{-stable} \}$.

In particular, $\sigma$ is in the closure of the set $\{\sigma'  |  E \text{ is } \sigma' \text{-stable} \}$. 
\end{proposition}
\begin{proof}
By Lemma \ref{dual of 5.9}, we see that $\sigma$ is in the closure of $\{\sigma'  |  E_1 \text{ is } \sigma' \text{-stable} \}$.
In particular, this gives a stability condition $\sigma_1$, close to $\sigma$, such that $E_1$ is $\sigma_1$-stable.

Moreover, as all the classes $[B_i]$ are linearly independent,
the equations of the walls $\phi(B_i)=\phi(B_{i+1})$ are also linearly independent.
Hence, the stability condition $\sigma_1$ may be chosen so that it lies on the walls $\phi(B_1)=\phi(B_2)= \cdots=\phi(B_n)$.

We may then apply Lemma \ref{dual of 5.9} to $\sigma_1$ and $E_1 \hookrightarrow E_2 \twoheadrightarrow B_2$ again and obtain a stability condition $\sigma_2$, close to $\sigma_1$, lying on the walls $\phi(B_2)=\phi(B_3)= \cdots=\phi(B_n)$.

Repeating this process, we reach a stability condition $\sigma_n$, sufficiently close to $\sigma$, such that $E_n$ is $\sigma$-stable,
which completes the proof.
\end{proof}

\begin{lemma}\label{geochamber}
For any $0<\epsilon\leq1$, the stability condition $\sigma_\epsilon$ lies on the closure of the geometric chamber in $\Stab(\widetilde{X})$, i.e. $\{\sigma' \in \Stab(\widetilde{X})  |  \mathcal{O}_x \text{ is } \sigma' \text{-stable},\forall x \in \widetilde{X} \}$.
\end{lemma}
\begin{proof}
For $x \in \widetilde{X}$, we set $U_x:=\{\sigma' \in \Stab(\widetilde{X})  |  \mathcal{O}_x \text{ is } \sigma' \text{-stable} \}$.

Given $\mathcal{O}_x$ with $x \in \Pi$,
we first pick an exceptional curve $C_{a_n}$ containing $x$.
Then for each $k=1, \cdots, n-1$,
we choose an exceptional curve $C_{a_k}$ such that $C_{a_k}$ is not $C_{a_n}$, that $x \in \overline{\Pi \setminus C_{a_1}\cup\cdots \cup C_{a_k}}$, and that $\overline{\Pi \setminus C_{a_1}\cup\cdots \cup C_{a_k}}$ is connected.

This gives us the following commutative diagram:
\[
\begin{tikzcd}
K_1:=\ker \phi_1 \arrow[d]                                              & K_2:=\ker \phi_2 \arrow[d]                                                                         & K_3:=\ker \phi_3 \arrow[d]                                                                                     &                                           & K_n:=\ker \phi_n \arrow[d]                                                                                   \\
\mathcal{O}_{\Pi} \arrow[r, "\phi_1", two heads] \arrow[rrd, two heads] & \mathcal{O}_{\overline{\Pi\setminus C_{a_1}}} \arrow[r, "\phi_2", two heads] \arrow[rd, two heads] & \mathcal{O}_{\overline{\Pi\setminus C_{a_1}\cup C_{a_2} }} \arrow[r, "\phi_3", two heads] \arrow[d, two heads] & \cdots \arrow[r, "\phi_{n-1}", two heads] & \mathcal{O}_{\overline{\Pi\setminus C_{a_1}\cup C_{a_2}\cup\cdots C_{a_{n-1}}}} \arrow[lld, "\phi_n", two heads] \\
   &               & \mathcal{O}_x               &       &            
\end{tikzcd}
\]

Since the pushforward of $\mathcal{O}_\Pi$, $\mathcal{O}_{\overline{\Pi\setminus C_{a_1}}}, \cdots \mathcal{O}_{\overline{\Pi\setminus C_{a_1}\cup C_{a_2}\cup\cdots C_{a_{n-1}}}}$ are all $\mathcal{O}_{x_0}$,
then $K_1, 
 \cdots, 
K_n$ are all in the kernel of pushforward $\ker \mathbf{R}\pi_\ast$.
By Lemma \ref{generatorofker}, these sheaves are all extensions of $\mathcal{O}_{C_i}(-1)$.

We start with the surjection $\mathcal{O}_{\Pi} \twoheadrightarrow \mathcal{O}_x$,
which factors through $\mathcal{O}_{\overline{\Pi\setminus C_{a_1}}}$.
As the kernel $K_1$ is supported on the curve $C_{a_1}$,
it can only be an iterated self-extension of $\mathcal{O}_{C_{a_1}}(-1)$.
This gives a short exact sequence in $\mathcal{B}^0$
\[0\longrightarrow \mathcal{O}_{\Pi} \longrightarrow \mathcal{O}_{\overline{\Pi\setminus C_{a_1}}}\longrightarrow K_1[1] \longrightarrow 0.\]

Similarly, the surjection $\mathcal{O}_{\overline{\Pi\setminus C_{a_1}}} \twoheadrightarrow \mathcal{O}_x$ factors through $\mathcal{O}_{\overline{\Pi\setminus C_{a_1}\cup C_{a_2}}}$,
and as the kernel $K_2$ is supported on the curve $C_{a_2}$, 
it can only be an iterated self-extension of $\mathcal{O}_{C_{a_2}}(-1)$.
Hence, there exists a short exact sequence in
$\mathcal{B}^0$
\[0\longrightarrow \mathcal{O}_{\overline{\Pi\setminus C_{a_1}}}\longrightarrow \mathcal{O}_{\overline{\Pi\setminus C_{a_1}\cup C_{a_2}}} \longrightarrow K_2[1] \longrightarrow 0.\]

Repeating this process, we finally reach a surjection $\mathcal{O}_{\overline{\Pi\setminus C_{a_1}\cup C_{a_2}\cup\cdots C_{a_{n-1}}}} \twoheadrightarrow \mathcal{O}_x$ of coherent sheaves,
with kernel an iterated self-extension of $\mathcal{O}_{C_{a_n}}(-1)$.

In summary, we now obtain a $\sigma_\epsilon$-Jordan--Hölder filtration of $\mathcal{O}_x$ of the form
\[
\begin{tikzcd}
0 \arrow[r, hook] & \mathcal{O}_{\Pi} \arrow[r, hook] & \mathcal{O}_{\overline{\Pi\setminus C_{a_1}}} \arrow[r, hook] \arrow[d, two heads] & \mathcal{O}_{\overline{\Pi\setminus C_{a_1}\cup C_{a_2}}} \arrow[r, hook] \arrow[d, two heads] & \cdots \arrow[r, hook] & \mathcal{O}_x \arrow[d, two heads] \\
&     & K_1[1]          & K_2[1]                       &     & K_n[1]             
\end{tikzcd}
\]
with each $K_i[1]$ an iterated self-extension of $\mathcal{O}_{C_{a_i}}(-1)[1]$.

Note that $[\mathcal{O}_{\Pi}],[\mathcal{O}_{C_i}(-1)[1]]$ are linearly independent in ${\rm K}_{\num}(\widetilde{X})$,
and that we have $\Hom(\mathcal{O}_{C_{a_1}}(-1)[1],\mathcal{O}_{\overline{\Pi\setminus C_{a_1}}})= \cdots =\Hom(\mathcal{O}_{C_{a_n}}(-1)[1],\mathcal{O}_x)= 0$.

In particular, if we set $Z_{\epsilon,\eta}:= Z_\epsilon +i (\eta \beta \cdot \ch_1)$ where $0<\eta\ll1$,
then by Bridgeland's deformation theorem (\cite[Theorem 1.2]{Bri07}) we obtain a family of stability conditions $\sigma_{\epsilon,\eta}=(Z_{\epsilon,\eta},\mathcal{A}_\eta)$ close to $\sigma_\epsilon$.

Note that we have as $\phi_{\epsilon,\eta}(\mathcal{O}_\Pi)<\phi_{\epsilon,\eta}(\mathcal{O}_{\overline{\Pi\setminus C_{a_1}}})<\cdots<\phi_{\epsilon,\eta}(\mathcal{O}_x)=1<\phi_{\epsilon,\eta}(\mathcal{O}_{C_i}(-1)[1])$ for all $i$ by our choiche of $\beta$.
Therefore, all the assumptions of Proposition \ref{5.9iterative} hold, 
and we see that $\sigma_\epsilon \in \overline{\{\sigma_{\epsilon,\eta}\}}\subseteq \overline{U_x}$.

Now for $x \notin \Pi$, by Lemma \ref{Oxsimple}, $\mathcal{O}_x$ is simple in $\mathcal{B}^0$, and so it is $\sigma_\epsilon$-stable.
To complete the proof, by the local finiteness of walls (with respect to the class $[\mathcal{O}_x]$) we can find an open neighborhood $U$ of $\sigma_\epsilon$ so that $U\subset \bigcap_{x\notin\Pi}U_x$.

Then for each $\eta$ sufficiently small, $\sigma_{\epsilon,\eta}$ lies in $U$,
so the set $\{\sigma_{\epsilon,\eta}\}$ is  contained in the geometric chamber $\bigcap_{x\in\widetilde{X}}U_x$,
and hence $\sigma_\epsilon \in \overline{\{\sigma_{\epsilon,\eta}\}} \subseteq \overline{\bigcap_{x\in\widetilde{X}}U_x}$.
\end{proof}

\begin{corollary}\label{Bepsilon}
For every class $\widetilde{v} \in {\rm K}_{\num}(\widetilde{X})$, the moduli space $\mathcal{M}_{\sigma_\epsilon}(\widetilde{v})$ is bounded.
\end{corollary}
\begin{proof}
Note first that by \cite[Proposition 4.11]{Tod08}, the boundedness holds for geometric stability conditions on any smooth projective surface.

Then by Toda's arguments in \cite[Sect. 3]{Tod08}, boundedness and openness are preserved under deformation of stability (see, for example, \cite[Theorem 22.2]{BLM+} for a precise reference).
In particular, the assertion follows from Lemma \ref{geochamber}.
\end{proof}


We are then ready to prove the main theorem of this section.
\begin{theorem}\label{Final}
Let $\widetilde{v} \in {\rm K}_{\num}(\widetilde{X})$ and $v \in {\rm K}_{\num}(X)$. Then $\mathcal{M}_{\sigma_\epsilon}(\widetilde{v})$, $\mathcal{M}_{\sigma_{\widetilde{X}}}(\widetilde{v})$, and $\mathcal{M}_{\sigma_X}(v)$ are all Artin stacks of finite type over $\mathbb{C}$.
\end{theorem}
\begin{proof}
By \cite[Theorem 4.2.1]{Lie06} and \cite[Lemma 3.4]{Tod08}, it suffices to prove all these groupoids satisfy boundedness and openness.
First, for any class $\widetilde{v} \in {\rm K}_{\num}(\widetilde{X})$,
the groupoid $\mathcal{M}_{\sigma_\epsilon}(\widetilde{v})$ is bounded by Corollary \ref{Bepsilon} and the generic flatness holds for $\mathcal{B}^0$ by Proposition \ref{genericflat}.

We know that $\sigma_{\widetilde{X}}$ is the deformation of $\sigma_\epsilon$ where $\epsilon$ tends to $0$,
and that $Z_\epsilon$ and $Z_{\widetilde{X}}=Z_0$ have the same imaginary part.
Then, for a class $\widetilde{v}$ with $\Im Z_\epsilon(\widetilde{v})=0$, the groupoid $\mathcal{M}_{\sigma_0}(\widetilde{v})$ is exactly the same as $\mathcal{M}_{\sigma_\epsilon}(\widetilde{v})$ so it satisfies boundedness and openness.

We may then assume that $\Im Z_\epsilon(\widetilde{v})>0$.
The proof of Lemma \ref{mathcalD} and Theorem \ref{mainSect5} implies that, for semistable objects in $\mathcal{P}_\epsilon(0,1)=\mathcal{P}_{\widetilde{X}}(0,1)$, instead of $Q_{A,B}$,
the quadratic form $Q_{A_0,0}$ (both defined in Section 5) works for all $\sigma_\epsilon,0 \leq \epsilon \leq 1$.

Moreover, the same arguments in Lemma \ref{negdef} shows that $Q_{A_0,0}$ is also negative definite on $\ker Z_{\widetilde{X}} \subseteq \widetilde{\Lambda}\otimes \mathbb{R}$.
Therefore, by the arguments in \cite[Sect. 5]{Bay19}, we can treat the path $\sigma_\epsilon,0\leq \epsilon \leq 1$,
as a path in the stability manifold just as in \cite{Bay19},
that is, it only crosses finitely many walls.

Then, as boundedness and openness are preserved by deformation of stability conditions (again, see \cite[Theorem 22.2]{BLM+}),
we see that $\mathcal{M}_{\sigma_{\widetilde{X}}}(\widetilde{v})$ satisfies boundedness and openness.

Note that by Theorem \ref{surjmoduli}, this also implies that  $\mathcal{M}_{\sigma_X}(v)$ satisfies boundedness for every class $v\in {\rm K}_{\num}(X)$.
More precisely, for a given $v \in {\rm K}_{\num}(X)$, by Theorem \ref{surjmoduli},
we can pick a class $\widetilde{v} \in {\rm K}_{\num}(\widetilde{X})$ so that for any $\sigma_X$-semistable object $E$ of class $v$,
there exists some $\sigma_{\widetilde{X}}$-semistable $\widetilde{E}$ of class $\widetilde{v}$ such that $\mathbf{R}\pi_\ast\widetilde{E}=E$.

As we have proved that $\mathcal{M}_{\sigma_{\widetilde X}}(\widetilde v)$ is bounded for every $\widetilde{v}$,
there is a scheme $S$ of finite type over $\mathbb{C}$ and a family $\widetilde{\mathcal{E}} \in \mathrm{D}^b(\widetilde{X} \times S)$ such that for every $\widetilde{E} \in \mathcal{M}_{\sigma_{\widetilde X}}(\widetilde v)(\mathbb{C})$,
$\widetilde{E} \simeq \mathcal{E}_s$ for some $s \in S$.
Then we have $E \simeq \mathbf{R}\pi_\ast\widetilde{E} \simeq \mathbf{R}\pi_\ast \widetilde{\mathcal{E}}_s \simeq (\mathbf{R}(\pi \times 1_S)_\ast \widetilde{\mathcal{E}})_s$.
That is, $\mathcal{M}_{\sigma_X}(v)$ satisfies boundedness with respect to the scheme $S$ and the family $\mathbf{R}(\pi \times 1_S)_\ast \widetilde{\mathcal{E}}$ on $\mathrm{D}^b(X \times S)$.

Finally, by \cite{Pol07}, Toda's argument in \cite[Sect. 3.2]{Tod08} can be generalized to singular varieties, 
and hence \cite[Lemma 4.7]{Tod08} implies that the generic flatness holds for $\Coh^0(X)$.
This shows that $\mathcal{M}_{\sigma_X}(v)$ is an Artin stack of finite type over $\mathbb{C}$.
\end{proof}

\bibliographystyle{alpha}
\bibliography{ref}

\newcommand{\etalchar}[1]{$^{#1}$}
\begin{thebibliography}{BPPW22}

\bibitem[Bay19]{Bay19}
Arend Bayer.
\newblock A short proof of the deformation property of {B}ridgeland stability conditions.
\newblock {\em Math. Ann.}, 375(3-4):1597--1613, 2019.

\bibitem[BKS18]{BKS18}
Alexey Bondal, Mikhail Kapranov, and Vadim Schechtman.
\newblock Perverse schobers and birational geometry.
\newblock {\em Selecta Math. (N.S.)}, 24(1):85--143, 2018.

\bibitem[BLM{\etalchar{+}}21]{BLM+}
Arend Bayer, Mart\'{\i} Lahoz, Emanuele Macr\`\i, Howard Nuer, Alexander Perry, and Paolo Stellari.
\newblock Stability conditions in families.
\newblock {\em Publ. Math. Inst. Hautes \'{E}tudes Sci.}, 133:157--325, 2021.

\bibitem[BM11]{BM11}
Arend Bayer and Emanuele Macr\`\i.
\newblock The space of stability conditions on the local projective plane.
\newblock {\em Duke Math. J.}, 160(2):263--322, 2011.

\bibitem[BMMS12]{BMMS12}
Marcello Bernardara, Emanuele Macr\`\i, Sukhendu Mehrotra, and Paolo Stellari.
\newblock A categorical invariant for cubic threefolds.
\newblock {\em Adv. Math.}, 229(2):770--803, 2012.

\bibitem[BMS16]{BMS16}
Arend Bayer, Emanuele Macr\`\i, and Paolo Stellari.
\newblock The space of stability conditions on abelian threefolds, and on some {C}alabi-{Y}au threefolds.
\newblock {\em Invent. Math.}, 206(3):869--933, 2016.

\bibitem[Bol23]{Bol23}
Barbara Bolognese.
\newblock A partial compactification of the {B}ridgeland stability manifold.
\newblock {\em Adv. Geom.}, 23(4):527--541, 2023.

\bibitem[BPPW22]{BPPW22}
Nathan Broomhead, David Pauksztello, David Ploog, and Jon Woolf.
\newblock Partial compactification of stability manifolds by massless semistable objects, 2022.
\newblock arXiv:\href{https://arxiv.org/abs/2208.03173}{2208.03173}.

\bibitem[Bri02]{Bri02}
Tom Bridgeland.
\newblock Flops and derived categories.
\newblock {\em Invent. Math.}, 147(3):613--632, 2002.

\bibitem[Bri07]{Bri07}
Tom Bridgeland.
\newblock Stability conditions on triangulated categories.
\newblock {\em Ann. of Math. (2)}, 166(2):317--345, 2007.

\bibitem[Bri08]{Bri08}
Tom Bridgeland.
\newblock Stability conditions on {$K3$} surfaces.
\newblock {\em Duke Math. J.}, 141(2):241--291, 2008.

\bibitem[Car05]{Car05}
Roger Carter.
\newblock {\em Lie algebras of finite and affine type}, volume~96 of {\em Cambridge Studies in Advanced Mathematics}.
\newblock Cambridge University Press, Cambridge, 2005.

\bibitem[CLSY24]{CLSY24}
Tristan~C. Collins, Jason Lo, Yun Shi, and Shing-Tung Yau.
\newblock Bridgeland/weak stability conditions under spherical twist associated to a torsion sheaf, 2024.
\newblock arXiv:\href{https://arxiv.org/abs/2411.18554}{2411.18554}.

\bibitem[HRS96]{HRS96}
Dieter Happel, Idun Reiten, and Sverre~O. Smal\o.
\newblock Tilting in abelian categories and quasitilted algebras.
\newblock {\em Mem. Amer. Math. Soc.}, 120(575):viii+ 88, 1996.

\bibitem[KS25]{KS24}
Alexander Kuznetsov and Evgeny Shinder.
\newblock Derived categories of {F}ano threefolds and degenerations.
\newblock {\em Invent. Math.}, 239(2):377--430, 2025.

\bibitem[Lan24]{Lan24}
Adrian Langer.
\newblock Bridgeland stability conditions on normal surfaces.
\newblock {\em Ann. Mat. Pura Appl. (4)}, 203(6):2653--2664, 2024.

\bibitem[Lan25]{Lan25}
Adrian Langer.
\newblock Intersection theory and {Chern} classes on normal varieties.
\newblock {\em J. Lond. Math. Soc., II. Ser.}, 112(1):38, 2025.
\newblock Id/No e70244.

\bibitem[Li19]{Li19}
Chunyi Li.
\newblock Stability conditions on {F}ano threefolds of {P}icard number 1.
\newblock {\em J. Eur. Math. Soc. (JEMS)}, 21(3):709--726, 2019.

\bibitem[Lie06]{Lie06}
Max Lieblich.
\newblock Moduli of complexes on a proper morphism.
\newblock {\em J. Algebraic Geom.}, 15(1):175--206, 2006.

\bibitem[LR22]{LR22}
Bronson Lim and Franco Rota.
\newblock Characteristic classes and stability conditions for projective {K}leinian orbisurfaces.
\newblock {\em Math. Z.}, 300(1):827--849, 2022.

\bibitem[Pol07]{Pol07}
Alexander Polishchuk.
\newblock Constant families of {$t$}-structures on derived categories of coherent sheaves.
\newblock {\em Mosc. Math. J.}, 7(1):109--134, 167, 2007.

\bibitem[Rei97]{Reid}
Miles Reid.
\newblock Chapters on algebraic surfaces.
\newblock In {\em Complex algebraic geometry ({P}ark {C}ity, {UT}, 1993)}, volume~3 of {\em IAS/Park City Math. Ser.}, pages 3--159. Amer. Math. Soc., Providence, RI, 1997.

\bibitem[Tod08]{Tod08}
Yukinobu Toda.
\newblock Moduli stacks and invariants of semistable objects on {$K3$} surfaces.
\newblock {\em Adv. Math.}, 217(6):2736--2781, 2008.

\bibitem[Tod13]{Tod13}
Yukinobu Toda.
\newblock Stability conditions and extremal contractions.
\newblock {\em Math. Ann.}, 357(2):631--685, 2013.

\bibitem[TX22]{TX22}
Rebecca Tramel and Bingyu Xia.
\newblock Bridgeland stability conditions on surfaces with curves of negative self-intersection.
\newblock {\em Adv. Geom.}, 22(3):383--408, 2022.

\bibitem[VdB04]{VdB04}
Michel Van~den Bergh.
\newblock Three-dimensional flops and noncommutative rings.
\newblock {\em Duke Math. J.}, 122(3):423--455, 2004.

\bibitem[Vil25]{Vil25}
Nicolás Vilches.
\newblock Stability conditions on surfaces and contractions of curves, 2025.

\end{thebibliography}
\end{document}